\documentclass{article}
\usepackage{amsmath}
\usepackage{amssymb}
\usepackage{amsthm}
\usepackage{graphicx}
\usepackage{cite}
\usepackage[arrow,curve,matrix,arc,2cell]{xy}
\UseAllTwocells
\PassOptionsToPackage{hyphens}{url}\usepackage{hyperref}
\DeclareFontFamily{U}{rsfs}{} \DeclareFontShape{U}{rsfs}{n}{it}{<->
rsfs10}{} \DeclareSymbolFont{mscr}{U}{rsfs}{n}{it}
\DeclareSymbolFontAlphabet{\scr}{mscr}
\def\mathscr{\scr}
\begin{document}
\def\e#1\e{\begin{equation}#1\end{equation}}
\def\ea#1\ea{\begin{align}#1\end{align}}
\def\eq#1{{\rm(\ref{#1})}}
\theoremstyle{plain}
\newtheorem{thm}{Theorem}[section]
\newtheorem{prop}[thm]{Proposition}
\newtheorem{lem}[thm]{Lemma}
\newtheorem{cor}[thm]{Corollary}
\newtheorem{quest}[thm]{Question}
\theoremstyle{definition}
\newtheorem{dfn}[thm]{Definition}
\newtheorem{ex}[thm]{Example}
\newtheorem{rem}[thm]{Remark}
\numberwithin{equation}{section}
\def\dim{\mathop{\rm dim}\nolimits}
\def\supp{\mathop{\rm supp}\nolimits}
\def\cosupp{\mathop{\rm cosupp}\nolimits}
\def\id{\mathop{\rm id}\nolimits}
\def\Hess{\mathop{\rm Hess}\nolimits}
\def\Sym{\mathop{\rm Sym}\nolimits}
\def\Ker{\mathop{\rm Ker}}
\def\Coker{\mathop{\rm Coker}}
\def\DR{\mathop{\rm DR}\nolimits}
\def\HP{\mathop{\rm HP}\nolimits}
\def\HN{\mathop{\rm HN}\nolimits}
\def\HC{\mathop{\rm HC}\nolimits}
\def\NC{\mathop{\rm NC}\nolimits}
\def\PC{\mathop{\rm PC}\nolimits}
\def\CC{\mathop{\rm CC}\nolimits}
\def\Crit{\mathop{\rm Crit}}
\def\an{{\rm an}}
\def\inf{{\rm inf}}
\def\fib{\mathop{\rm fib}\nolimits}
\def\Aut{\mathop{\rm Aut}}
\def\Hom{\mathop{\rm Hom}\nolimits}
\def\Ext{\mathop{\rm Ext}\nolimits}
\def\coh{\mathop{\rm coh}\nolimits}
\def\Ho{\mathop{\rm Ho}\nolimits}
\def\Perv{\mathop{\rm Perv}\nolimits}
\def\Sch{\mathop{\bf Sch}\nolimits}
\def\Var{\mathop{\bf Var}\nolimits}
\def\St{\mathop{\bf St}\nolimits}
\def\Art{\mathop{\bf Art}\nolimits}
\def\dSt{\mathop{\bf dSt}\nolimits}
\def\dArt{\mathop{\bf dArt}\nolimits}
\def\dSch{\mathop{\bf dSch}\nolimits}
\def\red{{\rm red}}
\def\cl{{\rm cl}}
\def\DM{\mathop{\rm DM}\nolimits}
\def\Spec{\mathop{\rm Spec}\nolimits}
\def\bSpec{\mathop{\bs{\rm Spec}}\nolimits}
\def\rank{\mathop{\rm rank}\nolimits}
\def\qcoh{{\mathop{\rm qcoh}}}
\def\bs{\boldsymbol}
\def\ge{\geqslant}
\def\le{\leqslant\nobreak}
\def\bA{{\mathbin{\mathbb A}}}
\def\bL{{\mathbin{\mathbb L}}}
\def\bT{{\mathbin{\mathbb T}}}
\def\bR{{\bs R}}
\def\bS{{\bs S}}
\def\bU{{\bs U}}
\def\bV{{\bs V}}
\def\bW{{\bs W}}
\def\bX{{\bs X}}
\def\bY{{\bs Y}}
\def\bZ{{\bs Z}}
\def\cA{{\mathbin{\cal A}}}
\def\cE{{\mathbin{\cal E}}}
\def\cL{{\mathbin{\cal L}}}
\def\cM{{\mathbin{\cal M}}}
\def\oM{{\mathbin{\smash{\,\,\overline{\!\!\mathcal M\!}\,}}}}
\def\O{{\mathbin{\cal O}}}
\def\PV{{\mathbin{\cal{PV}}}}
\def\bG{{\mathbin{\mathbb G}}}
\def\fE{{\mathbin{\mathfrak E}}}
\def\fG{{\mathbin{\mathfrak G}}}
\def\fU{\mathbin{\mathfrak U}}
\def\C{{\mathbin{\mathbb C}}}
\def\K{{\mathbin{\mathbb K}}}
\def\N{{\mathbin{\mathbb N}}}
\def\Q{{\mathbin{\mathbb Q}}}
\def\R{{\mathbin{\mathbb R}}}
\def\Z{{\mathbin{\mathbb Z}}}
\def\al{\alpha}
\def\be{\beta}
\def\ga{\gamma}
\def\de{\delta}
\def\io{\iota}
\def\ep{\epsilon}
\def\la{\lambda}
\def\ka{\kappa}
\def\th{\theta}
\def\ze{\zeta}
\def\up{\upsilon}
\def\vp{\varphi}
\def\si{\sigma}
\def\om{\omega}
\def\De{\Delta}
\def\La{\Lambda}
\def\Th{\Theta}
\def\Om{\Omega}
\def\Ga{\Gamma}
\def\Si{\Sigma}
\def\Up{\Upsilon}
\def\pd{\partial}
\def\ts{\textstyle}
\def\st{\scriptstyle}
\def\sst{\scriptscriptstyle}
\def\sm{\setminus}
\def\bu{\bullet}
\def\op{\oplus}
\def\ot{\otimes}
\def\boxt{\boxtimes}
\def\ov{\overline}
\def\ul{\underline}
\def\bigop{\bigoplus}
\def\iy{\infty}
\def\es{\emptyset}
\def\ra{\rightarrow}
\def\ab{\allowbreak}
\def\longra{\longrightarrow}
\def\hookra{\hookrightarrow}
\def\bs{\boldsymbol}
\def\t{\times}
\def\ci{\circ}
\def\ti{\tilde}
\def\d{{\rm d}}
\def\od{\odot}
\def\bd{\boxdot}
\def\ha{{\ts\frac{1}{2}}}
\def\md#1{\vert #1 \vert}
\def\bmd#1{\big\vert #1 \big\vert}
\def\cS{{\mathbin{\cal S}}}
\def\cSz{{\mathbin{\cal S}\kern -0.1em}^{\kern .1em 0}}
\newcommand{\dd}{{\rm d}_{dR}}
\title{A Darboux theorem for derived schemes with shifted
symplectic structure}
\author{Christopher Brav, Vittoria Bussi, and Dominic Joyce}
\date{}
\maketitle

\begin{abstract}
We prove a Darboux theorem for derived schemes with symplectic forms
of degree $k<0$, in the sense of Pantev, To\"en, Vaqui\'e and
Vezzosi \cite{PTVV}. More precisely, we show that a derived scheme
$\bX$ with symplectic form $\ti\om$ of degree $k$ is locally
equivalent to $(\bSpec A,\om)$ for $\bSpec A$ an affine derived
scheme in which the cdga $A$ has Darboux-like coordinates with
respect to which the symplectic form $\om$ is standard, and in which
the differential in $A$ is given by a Poisson bracket with a
Hamiltonian function $\Phi$ of degree $k+1$.

When $k=-1$, this implies that a $-1$-shifted symplectic derived
scheme $(\bX,\ti\om)$ is Zariski locally equivalent to the derived
critical locus $\bs\Crit(\Phi)$ of a regular function $\Phi:U\ra\bA^1$ on
a smooth scheme $U$. We use this to show that the classical scheme
$X=t_0(\bX)$ has the structure of an {\it algebraic d-critical
locus}, in the sense of Joyce~\cite{Joyc}.

In the sequels \cite{BBBJ,BoJo,BBDJS,BJM,JoSa} we extend these results to (derived) Artin stacks, and discuss a Lagrangian neighbourhood theorem for shifted symplectic derived schemes, and applications to categorified and motivic Don\-ald\-son--Thomas theory of Calabi--Yau 3-folds, and to defining new Don\-ald\-son--Thomas type invariants of Calabi--Yau 4-folds, and to defining `Fukaya categories' of Lagrangians in algebraic symplectic manifolds using perverse sheaves.
\end{abstract}

\setcounter{tocdepth}{2}
\tableofcontents

\section{Introduction}
\label{da1}

In the context of To\"en and Vezzosi's theory of derived algebraic
geometry \cite{Toen,ToVe1,ToVe2}, Pantev, To\"en, Vaqui\'e and
Vezzosi \cite{PTVV,Vezz2} defined a notion of $k$-{\it shifted
symplectic structure\/} $\om$ on a derived scheme or stack $\bX$,
for $k\in\Z$. If $\bX$ is a derived scheme and $\om$ a 0-shifted
symplectic structure, then $\bX=X$ is a smooth classical scheme and
$\om\in H^0(\La^2T^*X)$ a classical symplectic structure
on~$X$.

Pantev et al.\ \cite{PTVV} introduced a notion of Lagrangian $\bs
i:\bs L\ra\bX$ in a $k$-shifted symplectic derived stack
$(\bX,\om)$, and showed that the fibre product \hbox{$\bs L\t_\bX\bs
M$} of Lagrangians $\bs i:\bs L\ra\bX$, $\bs j:\bs M\ra\bX$ is
$(k-1)$-shifted symplectic. Thus, (derived) intersections $L\cap M$
of Lagrangians $L,M$ in a classical algebraic symplectic manifold
$(S,\om)$ are $-1$-shifted symplectic. They also proved that if $Y$
is a Calabi--Yau $m$-fold then the derived moduli stacks $\bs\cM$ of
(complexes of) coherent sheaves on $Y$ carry a natural
$(2-m)$-shifted symplectic structure.

The main aim of this paper is to prove a `Darboux Theorem', Theorem
\ref{da5thm1} below, which says that if $\bX$ is a derived scheme
and $\ti\om$ a $k$-shifted symplectic structure on $\bX$ for $k<0$
with $k\not\equiv 2\mod 4$, then $(\bX,\ti\om)$ is Zariski locally
equivalent to $(\bSpec A,\om)$, for $\bSpec A$ an affine derived
scheme in which the cdga $A$ is smooth in degree zero and quasi-free
in negative degrees, and has Darboux-like coordinates
$x^i_j,y^{k-i}_j$ with respect to which the symplectic form
$\om=\sum_{i,j}\dd x^i_j\dd y^{k-i}_j$ is standard, and in which the
differential in $A$ is given by a Poisson bracket with a Hamiltonian
function $\Phi$ of degree $k+1$.

When $k<0$ with $k\equiv 2\mod 4$ we give two statements, one
Zariski local in $\bX$ in which the symplectic form $\om$ on $\bSpec
A$ is standard except for the part in the degree $k/2$ variables,
which depends on some functions $q_i$, and one \'etale local in
$\bX$ in which $\om$ is entirely standard. In the case $k=-1$, Theorem \ref{da5thm1} implies that a $-1$-shifted symplectic derived scheme $(\bX,\ti\om)$ is Zariski locally equivalent to the derived critical locus $\bs\Crit(\Phi)$ of a regular function $\Phi:U\ra\bA^1$ on a smooth scheme~$U$. 

This is the second in a series of eight papers \cite{Joyc}, \cite{BBDJS}, \cite{BJM}, \cite{BBBJ}, \cite{Buss}, \cite{BoJo}, \cite{JoSa}, with more to come. The previous paper \cite{Joyc} defined {\it algebraic d-critical loci\/} $(X,s)$, which are classical schemes $X$ with an extra (classical, not derived) geometric structure $s$ that records information on how $X$ may locally be written as a classical
critical locus $\Crit(\Phi)$ of a regular function $\Phi:U\ra\bA^1$ on a
smooth scheme~$U$.

Our second main result, Theorem \ref{da6thm4} below, says that if $(\bX,\om)$ is a $-1$-shifted symplectic derived scheme then the underlying classical scheme $X=t_0(\bX)$ extends naturally to an
algebraic d-critical locus $(X,s)$. That is, we define a truncation
functor from $-1$-shifted symplectic derived schemes to algebraic
d-critical loci.

The third and fourth papers \cite{BBDJS,BJM} will show that if $(X,s)$ is an algebraic d-critical locus with an `orientation', then we can define a natural perverse sheaf $P_{X,s}^\bu$, a $\scr D$-module
$D_{X,s}^\bu$, a mixed Hodge module $M_{X,s}^\bu$ (when $X$ is over
$\C$), and a motive $MF_{X,s}$ on $X$, such that if $(X,s)$ is
locally modelled on $\Crit(\Phi:U\ra\bA^1)$ then
$P_{X,s}^\bu,D_{X,s}^\bu, M_{X,s}^\bu$ are locally modelled on the
perverse sheaf, $\scr D$-module, and mixed Hodge module of vanishing
cycles of $\Phi$, and $MF_{X,s}$ is locally modelled on the motivic vanishing cycle of~$\Phi$.

Combining these with Theorem \ref{da6thm4} and results of Pantev et
al.\ \cite{PTVV} gives natural perverse sheaves, $\scr D$-modules,
mixed Hodge modules, and motives on classical moduli schemes $\cM$
of simple (complexes of) coherent sheaves on a Calabi--Yau 3-fold
with `orientations', and on intersections $L\cap M$ of spin
Lagrangians $L,M$ in an algebraic symplectic manifold $(S,\om)$.
These will have applications to categorified and motivic
Donaldson--Thomas theory of Calabi--Yau 3-folds, and to defining
`Fukaya categories' of Lagrangians in algebraic symplectic manifolds
using perverse sheaves.

The fifth paper \cite{BBBJ} will extend the results of this paper and \cite{BBDJS,BJM} from (derived) schemes to (derived) Artin stacks. The sixth \cite{Buss} will prove a complex analytic analogue of Corollary \ref{da6cor2} below, saying that the intersection $X=L\cap M$ of complex Lagrangians $L,M$ in a complex symplectic manifold $(S,\om)$ extends naturally to a complex analytic d-critical locus~$(X,s)$.

The seventh paper \cite{BoJo} uses Theorem \ref{da5thm1} to show that any $-2$-shifted symplectic derived $\C$-scheme $(\bX,\om)$ can be given the structure of a derived smooth manifold $\bX_{\rm dm}$. If $\bX_{\rm dm}$ is compact and oriented, it has a virtual cycle $[\bX_{\rm dm}]$ in bordism or homology. Using this, we propose to define Donaldson--Thomas style invariants `counting' (semi)stable coherent sheaves on Calabi--Yau 4-folds.

The eighth paper \cite{JoSa} proves a `Lagrangian neighbourhood theorem' which gives local models for Lagrangians $\bs L$ in $k$-shifted symplectic derived $\K$-schemes $(\bX,\om)$, relative to the `Darboux form' local models for $(\bX,\om)$ in Theorem~\ref{da5thm1}.

We begin with background material from algebra in \S\ref{da2}, and
from derived algebraic geometry in \S\ref{da3}. Section \ref{da4},
which is not particularly original, proves that any derived
$\K$-scheme $\bX$ is near any point $x\in\bX$ Zariski locally
equivalent to $\bSpec A$ for $A$ a `standard form' cdga which is
minimal at $x$, and explains how to compare two such local
presentations~$\bSpec A\simeq \bX\simeq \bSpec B$.

The heart of the paper is \S\ref{da5}, which defines our `Darboux
form' local models $A,\om$, and proves our main result Theorem
\ref{da5thm1}, that any $k$-shifted symplectic derived $\K$-scheme
$(\bX,\ti\om)$ for $k<0$ is locally of the form $(\bSpec A,\om)$.
Finally, \S\ref{da6} discusses algebraic d-critical loci from
\cite{Joyc}, and proves our second main result Theorem
\ref{da6thm4}, defining a truncation functor from $-1$-shifted
symplectic derived $\K$-schemes to algebraic d-critical loci.

Bouaziz and Grojnowski \cite{BoGr} have independently proved their own `Darboux Theorem' for $k$-shifted symplectic derived schemes for $k<0$, similar to Theorem \ref{da5thm1}, by a different method. See Joyce and Safronov \cite[Rem.~2.15]{JoSa} for an explanation of how this paper and \cite{BoGr} are related.
\medskip

\noindent{\bf Conventions.} Throughout $\K$ will be an algebraically closed field with characteristic zero. All cdgas will be graded in nonpositive degrees (i.e.\ they are {\it connective\/} cdgas). All classical $\K$-schemes are assumed locally of finite type, and all derived $\K$-schemes $\bX$ are assumed to be locally finitely presented.
\medskip

\noindent{\bf Acknowledgements.} We would like to thank Dennis Borisov, Pavel Safronov, Bal\'azs Szendr\H oi, Bertrand To\"en and Gabriele Vezzosi for helpful conversations, and two referees for useful comments. This research was supported by EPSRC Programme Grant
EP/I033343/1.

\section{Background from algebra}
\label{da2}

We begin by reviewing some fairly standard facts about commutative differential graded algebras (cdgas), and their cotangent complexes. Some references on cdgas in Derived Algebraic Geometry are are To\"en and Vezzosi \cite[\S 2.3]{ToVe1}, \cite[\S 2]{ToVe2} and Lurie \cite[\S 7.1]{Luri2}, on cdgas from other points of view are Gelfand and Manin \cite[\S 5.3]{GeMa} and Hess \cite{Hess}, and on cotangent complexes are To\"en and Vezzosi \cite[\S 1.4]{ToVe1}, \cite[\S 4.2.4--\S 4.2.5, \S 3.1.7]{Toen} and Lurie~\cite[\S 7.3]{Luri2}.

\subsection{Commutative graded algebras}
\label{da21}

In this section we introduce general definitions and conventions
about commutative graded algebras.

\begin{dfn}
\label{da2def1} 
Fix a ground field $\K$ of characteristic zero.
Eventually we shall wish to standardize a quadratic form, at which
point we need to assume in addition that $\K$ is algebraically
closed. We shall work with {\it connective commutative graded algebras} $A$
over $\K$ so $A$ has a decomposition $A=\bigop_{i\le 0} A^{i}$ (`connective' means concentrated in non-positive degrees) and an associative product
$m : A^{i} \ot A^{j} \ra A^{i+j}$ satisfying $fg=(-1)^{|f||g|}gf$
for homogeneous elements $f,g \in A$. Given a graded left module
$M$, we consider it as a {\it symmetric bimodule} by setting
\begin{equation*}
m \cdot f=(-1)^{|f||m|} f \cdot m
\end{equation*}
for homogeneous elements $f \in A$, $m \in M$.

Define a {\it derivation of degree $k$} from $A$ to a graded module
$M$ to be a $\K$-linear map $\de: A \ra M$ that is homogeneous of
degree $k$ and satisfies
\begin{equation*}
\de(fg)=\de(f)g+(-1)^{k|f|}f\de(g).
\end{equation*}
Just as for (ungraded) commutative algebras, there is a universal
derivation into a (graded) module of {\it K\"ahler differentials\/}
$\Om^1_A$, which can be constructed as $I/I^2$ for $I=\Ker(m: A\ot A
\ra A)$. The universal derivation $\de: A \ra \Om^1_A$ is then
computed as $\de(a)=a\ot 1-1 \ot a \in I/I^2$. One checks that
indeed $\de$ is a universal degree $0$ derivation, so that ${}\ci
\de:\ul{\rm Hom}_A(\Om^1_A,M) \ra \ul{\rm Der}(A,M)$ is an
isomorphism of graded modules. Here the underline denotes that we
take a sum over homogeneous morphisms of each degree.

In the particular case when $M=A$ one sometimes refers to a
derivation $X: A \ra A$ of degree $k$ as a `vector field of degree
$k$'. Define the {\it graded Lie bracket\/} of two homogeneous
vector fields $X,Y$ by
\begin{equation*}
[X,Y]:=XY-(-1)^{|X||Y|}YX.
\end{equation*}
One checks that $[X,Y]$ is a homogeneous vector field of degree
$|X|+|Y|$. On any commutative graded algebra, there is a canonical
degree $0$ {\it Euler vector field\/} $E$ which acts on a
homogeneous element $f \in A$ via $E(f)=|f|f$. In particular, it
annihilates functions $f \in A^0$ of degree $0$.

Define the {\it de Rham algebra\/} of $A$ to be the free commutative
graded algebra over $A$ on the graded module $\Om^1_A[1]$:
\e
\DR(A):=\Sym_A(\Om^1_A[1])=\ts \bigop_p\La^p \Om^1_A[p].
\label{da2eq1}
\e
We endow $\DR(A)$ with the de Rham operator $\dd$, which is the
unique square-zero derivation of degree $-1$ on the commutative
graded algebra $\DR(A)$ such that for $f \in A$, $\dd(f)=\de(f)[1]
\in \Om^1_A[1]$. Thus $\dd(fg)=\dd(f)g+(-1)^{|f|}f\,\dd(g)$ for $f,g
\in A$ and $\dd(\al \cdot \be)=\dd(\al)\be+(-1)^{|\al|}\al \dd(\be)$
for any two $\al, \be \in \DR(A)$.
\end{dfn}

\begin{rem} 
\label{da2rem1}
The de Rham algebra $\DR(A)$ has two gradings, one
induced by the grading on $A$ and on the module $\Om^1[1]$ and the
other given by $p$ in the decomposition $\DR(A):=\Sym_A(\Om^1_A[1])
\cong \bigop_p \La^p \Om^1_A[p]$. We shall refer to the first
grading as {\it degree} and the second grading as {\it weight}. Thus
the de Rham operator $\dd$ has degree $-1$ and weight $+1$.

Note that the condition that $\dd$ be a derivation of degree $-1$
{\it does not} take into account the additional grading by weight,
but nevertheless this convention does recover the usual de Rham
complex in the classical case when the algebra $A$ is concentrated
in degree $0$.
\end{rem}

\begin{ex} 
\label{da2ex1}
For us, a typical commutative graded algebra
$A=\bigop_{i\le 0} A^{i}$ will have $A^0$ smooth over $\K$ and will
be free over $A^0$ on graded variables $x^{-1}_1,\ab\ldots,\ab
x^{-1}_{m_1},\ab x^{-2}_1,\ldots,x^{-2}_{m_2},
\ldots,x^{-n}_1,\ldots,x^{-n}_{m_{n}}$ where $x^{-i}_j$ has degree
$-i$. Localizing $A^0$ if necessary, we may assume that there exist
$x^0_1,\ldots,x^0_{m_0} \in A^0$ so that
\begin{equation*}
\dd x^0_1,\ldots,\dd x^0_{m_0}
\end{equation*}
form a basis of $\Om^1_{A^0}$.

The graded $A$-module $\Om^1_A[1]$ then has a basis
\begin{equation*}
\dd x^0_1,\ldots,\dd x^0_{m_0},\dd x^{-1}_1,\ldots,
\dd x^{-1}_{m_{-1}},\ldots,\dd x^{-n}_1,\ldots,\dd x^{-n}_{m_{-n}}.
\end{equation*}
Since $\DR(A)=\Sym_A(\Om^1_A[1])$ is free commutative on
$\Om^1_A[1]$ over $A$ and since $\dd x^{-i}_{s}$ and $\dd
x^{-j}_{t}$ have degrees $-i-1$ and $-j-1$ respectively, we have
\begin{equation*}
\dd x^{-i}_{s} \dd x^{-j}_{t}=(-1)^{(-i-1)(-j-1)}\dd
x^{-j}_{t} \dd x^{-i}_{s}.
\end{equation*}
Thus, for example, $\dd x^{-i}_{s}$ and $\dd x^{-i}_{t}$ anticommute
when $-i$ is even and commute when $-i$ is odd. In particular, $\dd
x^{-i}_{s} \dd x^{-i}_{s}= 0$ when $-i$ is even and $\dd x^{-i}_{s}
\dd x^{-i}_{s}\neq 0$ when $-i$ is odd.
\end{ex}

\begin{dfn} 
\label{da2def2}
Given a homogeneous vector field $X$ on $A$ of degree
$|X|$, the {\it contraction operator\/} $\io_X$ on $\DR(A)$ is
defined to be the unique derivation of degree $|X|+1$ such that
$\io_Xf=0$ and $\io_X\dd(f)=X(f)$ for all $f \in A$. We define the
{\it Lie derivative\/} $L_X$ along a vector field $X$ by
\begin{equation*}
L_X=[\io_X,\dd]=\io_X\dd -(-1)^{|X|+1}\dd \io_X=
\io_X\dd+(-1)^{|X|}\dd\io_X.
\end{equation*}
It is a derivation of $\DR(A)$ of degree $|X|$. In particular, the
Lie derivative along $E$ is of degree $0$. Given $f \in A$, we have
\begin{align*}
L_Ef&=\io_E\dd f=E(f)=|f|f, \\
L_E\dd f&=|f|\dd f.
\end{align*}
In particular, the de Rham differential $\dd : A \ra \Om^1_A[1]$ is
injective except in degree $0$. Furthermore, we see that for a
homogeneous form $\al \in \La^p\Om^1_A[p]$
\begin{equation*}
L_E\al=\io_E\dd \al+\dd \io_E \al=(|\al|+p)\al.
\end{equation*}
If $\al$ is in addition de Rham closed, then $\dd \io_E
\al=(|\al|+p)\al$, so if $|\al|+p \neq 0$, then $\al$ is in fact de
Rham exact:
\begin{equation*}
\al=\dd\Bigl(\frac{\io_E\al}{|\al|+p}\Bigr).
\end{equation*}
The top degree of $\La^p\Om^1_A[p]$ is $-p$ and the elements of top
degree are of the form $\al \in \La^p\Om^1_{A^0}[p]$, so a de Rham
closed form $\al$ can fail to be exact only if it lives on $A^0
\subset A$.
\end{dfn}

The following relations between derivations of $\DR(A)$ can be
checked by noting that both sides of an equation have the same
degree and act in the same way on elements of weight $0$ and weight
$1$ in $\DR(A)$.

\begin{lem} 
\label{da2lem1}
Let\/ $X,Y$ be homogeneous vector fields on $A$. Then
we have the following equalities of derivations on $\DR(A)\!:$
\begin{equation*}
[\dd,L_X]=0,\quad
[\io_X,\io_Y]=0,\quad
[L_X,\io_Y]=\io_{[X,Y]},\quad
[L_X,L_Y]=L_{[X,Y]}.
\end{equation*}

\end{lem}

\begin{rem} 
\label{da2rem2}
The last statement ensures that the map $X \mapsto
L_X$ determines a dg-Lie algebra homomorphism $\ul{\rm Der}(A) \ra
\ul{\rm Der}(\DR(A))$.
\end{rem}

\subsection{Commutative differential graded algebras}
\label{da22}

\begin{dfn} 
\label{da2def3}
A {\it commutative differential graded algebra\/} or
{\it cdga\/} $(A,\d)$ is a commutative graded algebra $A$ over $\K$, as in \S\ref{da2}, endowed with a square-zero derivation $\d$ of degree 1. Usually we write $A$ rather than $(A,\d)$, leaving $\d$ implicit. Note that the cohomology $H^*(A)$ of $A$ with respect to the differential $\d$ is a commutative graded algebra. Note too that by default all of our cdgas are {\it connective\/} (that is, concentrated in non-positive cohomological degrees).

A {\it morphism of cdgas} is a map of complexes $f: A \rightarrow B$ respecting units and  multiplication. It is a {\it quasi-isomorphism} or {\it weak equivalence} if the underlying map of complexes is so, it is a {\it fibration}  if it is a degree-wise surjection, and it is a {\it cofibration} if it has the left
lifting property with respect to trivial cofibrations. A standard argument (see for instance Goerss and Schemmerhorn \cite[Th.~3.6]{GoSc}) shows that in characteristic zero this choice of weak equivalences and fibrations define a model structure on cdgas. 

We shall assume that all cdgas in this paper are (homotopically) of {\it finite presentation\/} over the ground field $\K$. As for
finitely presented classical $\K$-algebras, this can be formulated
in terms of the preservation of filtered (homotopy) colimits. For
the precise notion, see To\"en and Vezzosi \cite[\S 1.2.3]{ToVe1}. A classical $\K$-algebra that is of finite presentation in the
classical sense {\it need not be\/} of finite presentation when
considered as a cdga. Indeed, Theorem \ref{da4thm1} will show that a cdga $A$ of finite presentation admits a very strong form of finite resolution.
\end{dfn}

In much of the paper, we will work with cdgas of the following form:

\begin{ex} 
\label{da2ex2}
We will explain how to inductively construct a sequence
of cdgas $A(0),A(1),\ldots,A(n)$, where $A(0)$ is a smooth $\K$-algebra,
and $A(k)$ has underlying commutative graded algebra free over $A(0)$ on generators of degrees~$-1,\ldots,-k$.

Begin with a commutative algebra $A(0)$ smooth over $\K$. To make Proposition \ref{da2prop} below hold (which says that the cotangent complex $\bL_{A(k)}$ has a simple description) we assume the cotangent module $\Om^1_{A(0)}$ is a free $A(0)$-module, which can always be achieved by Zariski localizing $A(0)$. Choose a free $A(0)$-module $M^{-1}$ of finite rank together with a map $\pi^{-1} :M^{-1} \ra A(0)$. Define a cdga $A(1)$ whose underlying
commutative graded algebra is free over $A(0)$ with generators given
by $M^{-1}$ in degree $-1$ and with differential $\d$ determined by
the map $\pi^{-1}: M^{-1} \ra A(0)$. By construction, we have
$H^0(A(1)) = A(0)/I$, where the ideal $I\subseteq A(0)$ is the image
of the map~$\pi^{-1}:M^{-1}\ra A(0)$.

Note that $A(1)$ fits in a homotopy pushout diagram of cdgas
\begin{equation*}
\xymatrix@R=14pt@C=90pt{*+[r]{\Sym_{A(0)}(M^{-1})} \ar[r]_(0.7){0_*}
\ar[d]^{\pi^{-1}_*} & *+[l]{A(0)}
\ar[d] \\ *+[r]{A(0)} \ar[r]^(0.7){f^{-1}} & *+[l]{A(1),\!\!{}} }
\end{equation*}
with morphisms $\pi^{-1}_*,0_*$ induced by $\pi^{-1},0:M^{-1}\ra
A(0)$. Write $f^{-1}:A(0)\ra A(1)$ for the resulting map of
algebras.

Next, choose a free $A(1)$-module $M^{-2}$ of finite rank together
with a map $\pi^{-2}: M^{-2}[1] \ra A(1)$. Define a cdga $A(2)$
whose underlying commutative graded algebra is free over $A(1)$ with
generators given by $M^{-2}$ in degree $-2$ and with differential
$\d$ determined by the map $\pi^{-2}: M^{-2}[1] \ra A(1)$. Write
$f^{-2}$ for the resulting map of algebras~$A(1)\ra A(2)$.

As the underlying commutative graded algebra of $A(1)$ was free over
$A(0)$ on generators of degree $-1$, the underlying commutative
graded algebra of $A(2)$ is free over $A(0)$ on generators of
degrees $-1,-2$. Since $A(2)$ is obtained from $A(1)$ by adding
generators in degree $-2$, we have $H^0(A(1))\cong H^0(A(2)) \cong
A(0)/I$.

Note that $A(2)$ fits in a homotopy pushout diagram of cdgas
\begin{equation*}
\xymatrix@R=14pt@C=90pt{*+[r]{\Sym_{A(1)}(M^{-2}[1])} \ar[r]_(0.7){0_*}
\ar[d]^{\pi^{-2}_*} & *+[l]{A(1)} \ar[d] \\
*+[r]{A(1)} \ar[r]^(0.7){f^{-2}} & *+[l]{A(2),\!\!{}} }
\end{equation*}
with morphisms $\pi^{-2}_*,0_*$ induced by $\pi^{-2},0:M^{-2}[1]\ra
A(1)$.

Continuing in this manner inductively, we define a cdga $A(n)=A$
with $A^0=A(0)$ and $H^0(A)=A(0)/I$, whose underlying commutative
graded algebra is free over $A(0)$ on generators of
degrees~$-1,\ldots,-n$.
\end{ex}

\begin{dfn} 
\label{da2def4}
A cdga $A=A(n)$ over $\K$ constructed inductively as in
Example \ref{da2ex2} from a smooth $\K$-algebra $A(0)$ with $\Om^1_{A(0)}$ free of finite rank over $A(0)$, and free finite rank modules $M^{-1},M^{-2},\ldots,\ab M^{-n}$ over $A(0),\ldots,A(n-1)$, will be called a {\it standard form\/} cdga. 

Equivalently, $A$ is of standard form if it is finitely presented, and $A(0)$ is smooth with $\Om_{A(0)}$ a free $A(0)$-module, and the underlying graded algebra of $A$ is freely generated over $A(0)$ by finitely many generators in negative degrees.

If $A$ is of standard form, we will call a cdga $A'$ a {\it
localization\/} of $A$ if $A'=A\ot_{A^0}A^0[f^{-1}]$ for $f\in A^0$,
that is, $A'$ is obtained by inverting $f$ in $A$. Then $A'$ is also
of standard form, with $A^{\prime\, 0}\cong A^0[f^{-1}]$. If $p\in
\Spec H^0(A)$ with $f(p)\ne 0$, we call $A'$ a {\it localization
of\/ $A$ at\/}~$p$.
\end{dfn}

Standard form cdgas are a mild generalization of {\it connective semi-free cdgas}, which are standard form cdgas in which $A(0)$ is a free commutative $\K$-algebra. Semi-free cdgas are among the cofibrant objects in the model structure on $\mathop{\bf cdga}_\K^{\le 0}$, and Proposition \ref{da2prop} below is standard for semi-free cdgas.

\subsection{Cotangent complexes of cdgas}
\label{da23}

\begin{dfn} 
\label{da2def5}
Let $(A,d)$ be a cdga. Then as in Definition
\ref{da2def1}, to the underlying commutative graded algebra $A$ we
associate the module of K\"ahler differentials $\Om^1_A$ with
universal degree 0 derivation $\de:A\ra\Om^1_A$, and the de Rham
algebra $\DR(A)=\Sym_A\bigl(\Om^1_A[1]\bigr)$ in \eq{da2eq1}, with
degree $-1$ de Rham differential~$\dd:\DR(A)\ra\DR(A)$.

The differential $\d$ on $A$ induces a unique differential on
$\Om^1_A$, also denoted $\d$, satisfying $\d\ci\de=\de\ci\d:
A\ra\Om^1_A$, and making $(\Om^1_A,\d)$ into a dg-module. According
to our sign conventions, the differential $\d$ on $\Om^1_A[1]$ is
that on $\Om^1_A$ multiplied by $-1$, so $\d$ on $\Om^1_A[1]$
anti-commutes with the de Rham operator $\dd:A\ra\Om^1[1]$. We
extend the differential $\d$ uniquely to all of $\DR(A)$ by
requiring it to be a derivation of degree $1$ with respect to the
multiplication on~$\DR(A)$.

When $(A,\d)$ is sufficiently nice (as in Example \ref{da2ex2}),
then the K\"ahler differentials $(\Om^1_A,\d)$ give a model for the
{\it cotangent complex\/} $\bL_{(A,\d)}$ of $(A,\d)$. In practice,
we shall always work with such cdgas, so we shall freely identify
$\bL_{(A,\d)}$ and $(\Om^1_A,\d)$. Usually, when dealing with cdgas
and their cotangent complexes we leave $\d$ implicit, and write
$\bL_A,\Om^1_A$ rather than $\bL_{(A,\d)}$ and $(\Om^1_A,\d)$.

Similarly, given a map $A \ra B$ of cdgas, we can define the {\it
relative K\"ahler differentials\/} $\Om^1_{B/A}$, and when the map
$A \ra B$ is nice enough (for example, $B$ is obtained from $A$
adding free generators of some degree and imposing a differential,
as in Example \ref{da2ex2}), then the relative K\"ahler
differentials give a model for the {\it relative cotangent
complex\/}~$\bL_{B/A}$.
\end{dfn}

We recall some basic facts about cotangent complexes:
\begin{itemize}
\setlength{\itemsep}{0pt}
\setlength{\parsep}{0pt}
\item[(i)] Given a map of cdgas $\al:A \ra B$, there is an
induced map $\bL_\al:\bL_A \ra \bL_B$ of $A$-modules, and we
have a (homotopy) fibre sequence
\begin{equation*}
\xymatrix@C=50pt{\bL_A \ot_A B \ar[r]^(0.55){\bL_\al\ot\id_B} &
\bL_B \ar[r] & \bL_{B/A}.}
\end{equation*}
\item[(ii)] Given a (homotopy) pushout square of cdgas
\begin{equation*}
\xymatrix@R=12pt@C=80pt{*+[r]{A} \ar[r] \ar[d] & *+[l]{B} \ar[d] \\
*+[r]{A'} \ar[r] & *+[l]{B',\!\!{}} }
\end{equation*}
we have a {\it base change} equivalence
\e
\bL_{B'/A'} \simeq \bL_{B/A} \ot_B B'.
\label{da2eq2}
\e
\item[(iii)] If $B=\Sym_A(M)$ for a module $M$ over a cdga $A$,
with inclusion morphism $\io:A\hookra B=A\op M\op
S^2M\op\cdots$, we have a canonical equivalence
\begin{equation*}
\bL_{B/A} \simeq B \ot_A M.
\end{equation*}
Also, any $A$-module morphism $\rho:M\ra A$ (such as $\rho=0$)
induces a unique cdga morphism $\rho_*:B\ra A$ with
$\rho_*\vert_{\Sym^0(M)}=\id_A$ and
$\rho_*\vert_{\Sym^1(M)}=\rho$. Then we have a canonical
equivalence
\e
\bL_{A/B}\simeq M[1].
\label{da2eq3}
\e
\end{itemize}

\begin{ex} 
\label{da2ex3}
In Example \ref{da2ex2}, we constructed inductively a
sequence of cdgas $A(0) \ra A(1) \ra \cdots \ra A(n)=A$, where
$A(k)$ was obtained from $A(k-1)$ by adjoining a module of
generators $M^{-k}$ in degree $-k$ and imposing a differential. The
cdga $A(k)$ then fits into a (homotopy) pushout diagram
\e
\begin{gathered}
\xymatrix@R=16pt@C=120pt{*+[r]{B(k-1):=\Sym_{A(k-1)}(M^{-k}[k-1])}
\ar[r]_(0.76){0_*} \ar[d]^{\pi^{-k}_*} & *+[l]{A(k-1)} \ar[d] \\
*+[r]{A(k-1)} \ar[r]^(0.76){f^{-k}} & *+[l]{A(k),} }
\end{gathered}
\label{da2eq4}
\e
in which $\pi^{-k}_*,0_*$ are induced by maps
$\pi^{-k},0:M^{-k}[k-1]\ra A(k-1)$.

We will describe the relative cotangent complexes
$\bL_{A(k)/A(k-1)}$ and hence, inductively, the cotangent complex
$\bL_A$. By equation \eq{da2eq3}, the relative cotangent complex
$\bL_{A(k-1)/B(k-1)}$ for $0_*:B(k-1)\ra A(k-1)$ in \eq{da2eq4} is
given by
\begin{equation*}
\bL_{A(k-1)/B(k-1)}\simeq M^{-k}[k].
\end{equation*}
Thus by base change \eq{da2eq2}, the relative cotangent complex
$\bL_{A(k)/A(k-1)}$ of $f^{-k}:A(k-1) \ra A(k)$ in \eq{da2eq4}
satisfies
\begin{equation*}
\bL_{A(k)/A(k-1)} \simeq A(k) \ot_{A(k-1)}M^{-k}[k].
\end{equation*}
Note in particular that when $i > -k$, we have
\e
H^{i}(\bL_{A(k)/A(k-1)})=0.
\label{da2eq5}
\e

Commutative diagrams of the following form, in which the rows and
columns are fibre sequences, are very useful for computing the
cohomology of $\bL_A$:
\e
\begin{gathered}
\xymatrix@C=35pt@R=14pt{&& A \ot_{A(k)}\bL_{A(k)/A(k-1)} \ar[d] \\
A \ot_{A(k-1)}\bL_{A(k-1)} \ar[r]\ar[d] & \bL_A \ar[r]\ar@{=}[d]
& \bL_{A/A(k-1) \ar[d]} \\
A \ot_{A(k)}\bL_{A(k)} \ar[r]\ar[d] & \bL_A \ar[r] &
\bL_{A/A(k)} \\
A \ot_{A(k)}\bL_{A(k)/A(k-1)}. }
\end{gathered}
\label{da2eq6}
\e

\end{ex}

\begin{prop} 
\label{da2prop}
Let\/ $A=A(n)$ be a standard form cdga constructed
inductively as in Example\/ {\rm\ref{da2ex2}}. Then the restriction
of the cotangent complex\/ $\bL_A$ to\/ $\Spec H^0(A)$ is naturally
represented as a complex of free\/ $H^0(A)$-modules
\e
\begin{gathered}
\xymatrix@C=27pt{ 0 \ar[r] & V^{-n} \ar[r]^{\d^{-n}} & V^{1-n}
\ar[r]^{\d^{1-n}} & \cdots \ar[r]^{\d^{-2}} & V^{-1}
\ar[r]^{\d^{-1}} & V^0 \ar[r] & 0, }
\end{gathered}
\label{da2eq7}
\e
where $V^{-k}=H^{-k}\bigl(\bL_{A(k)/A(k-1)}\bigr)$ is in degree $-k$
and the differential\/ $\d^{-k}:V^{-k}\ra V^{1-k}$ is identified
with the composition
\begin{equation*}
H^{-k}\bigl(\bL_{A(k)/A(k-1)}\bigr)\longra H^{1-k}\bigl(\bL_{A(k-1)}
\bigr)\longra H^{1-k}\bigl(\bL_{A(k-1)/A(k-2)}\bigr),
\end{equation*}
in which\/ $H^{-k}\bigl(\bL_{A(k)/A(k-1)}\bigr)\ra
H^{1-k}\bigl(\bL_{A(k-1)}\bigr)$ is induced by the fibre sequence
\begin{equation*}
A(k) \ot_{A(k-1)} \bL_{A(k-1)}\longra \bL_{A(k)} \longra
\bL_{A(k)/A(k-1)},
\end{equation*}
and\/ $H^{1-k}\bigl(\bL_{A(k-1)}\bigr)\ra H^{1-k}\bigl(
\bL_{A(k-1)/A(k-2)}\bigr)$ is induced by the fibre sequence
\begin{equation*}
A(k-1)\ot_{A(k-2)}\bL_{A(k-2)}\longra\bL_{A(k-1)}\longra
\bL_{A(k-1)/A(k-2)}.
\end{equation*}

\end{prop}

\begin{proof} The proof is almost immediate by induction on $n$. When $n=0$, $\bL_{A(0)}=\Om^1_{A(0)}$ is free of finite rank over $A(0)$ by definition of standard form cdgas, so $V^0=\Om^1_{A(0)}\ot_{A(0)}H^0(A)$ is free of finite rank over $H^0(A)$. For $n=1$, tensor the fibre sequence $A(1) \ot_{A(0)}
\bL_{A(0)} \ra \bL_{A(1)} \ra \bL_{A(1)/A(0)}$ with $H^0(A)$ to get
a fibre sequence $V^0 \ra H^0(A)\ot_{A(1)} \bL_{A(1)} \ra V^{-1}[1]$
in which the connecting morphism $V^{-1} \ra V^0$ is as claimed.

Similarly, assuming we have proved the proposition for
$\bL_{A(n-1)}$, tensor the fibre sequence $A \ot_{A(n-1)}
\bL_{A(n-1)}\ra \bL_A \ra \bL_{A/A(n-1)}$ with $H^0(A)$ to get a
fibre sequence $H^0(A)\ot_{A(n-1)} \bL_{A(n-1)} \ra H^0(A) \ot_A
\bL_A \ra V^{-n}[n]$ in which the connecting morphism is as claimed.
\end{proof}

\begin{dfn} 
\label{da2def6}
Let $A$ be a standard form cdga constructed as in
Example \ref{da2ex2}. We call $A$ {\it minimal\/} at $p \in \Spec
H^0(A)$ if $\d^{-k}\vert_p=0$ for $k=1,\ldots,n$, for $d^{-k}$ the internal differential in $\bL_A$ as in \eq{da2eq7}. That is, the compositions
\e
H^{-k}\bigl(\bL_{A(k)/A(k-1)}\bigr) \longra
H^{1-k}\bigl(\bL_{A(k-1)}\bigr)\longra
H^{1-k}\bigl(\bL_{A(k-1)/A(k-2)}\bigr)
\label{da2eq8}
\e
in the cotangent complexes restricted to $\Spec H^0(A)$ vanish
at $p$ for all $k$.
\end{dfn}

\section{Background from derived algebraic geometry}
\label{da3}

Next we outline the background material from derived algebraic
geometry that we need, aiming in particular to explain the key
notion from Pantev, To\"en, Vaqui\'e and Vezzosi \cite{PTVV} of a
$k$-{\it shifted symplectic structure\/} on a {\it derived\/
$\K$-scheme}, which is central to our paper. There are two main
frameworks for derived algebraic geometry in the literature, due to
To\"en and Vezzosi \cite{Toen,ToVe1,ToVe2} and Lurie
\cite{Luri1,Luri2}, which are broadly interchangeable in
characteristic zero. Following our principal reference \cite{PTVV},
we use the To\"en--Vezzosi version.

To understand this paper (except for a few technical details), one
does not really need to study derived algebraic geometry in any
depth, or to know what a derived stack is. The main point for us is
that {\it a derived\/ $\K$-scheme is a geometric space locally
modelled on $\bSpec A$ for $A$ a cdga over\/} $\K$, just as a
classical $\K$-scheme is a space locally modelled on $\Spec A$ for
$A$ a commutative $\K$-algebra; though the meaning of `locally
modelled' is more subtle in the derived than the classical setting.
Readers who are comfortable with this description of derived schemes
can omit \S\ref{da31} and~\S\ref{da32}.

As in \S\ref{da2}, all cdgas in this paper are {\it connective\/} (graded in nonpositive degrees), as these are the allowed local models in derived algebraic geometry.

\subsection{Derived stacks}
\label{da31}

We will use To\"en and Vezzosi's theory of derived algebraic
geometry \cite{Toen,ToVe1,ToVe2}. We give a brief outline, to fix
notation. Fix a base field $\K$, of characteristic zero. In \cite[\S
3]{Toen}, \cite[\S 2.1]{ToVe1}, To\"en and Vezzosi define an
$\iy$-category of ({\it higher\/}) $\K$-{\it stacks\/} $\St_\K$.
Objects $X$ in $\St_\K$ are $\iy$-functors
\begin{equation*}
X:\{\text{commutative $\K$-algebras}\}\longra\{\text{simplicial sets}\}
\end{equation*}
satisfying sheaf-type conditions. They also define a full
$\iy$-subcategory $\Art_\K\subset\St_\K$ of ({\it higher\/}) {\it
Artin\/ $\K$-stacks}, with better geometric properties.

Classical $\K$-schemes and algebraic $\K$-spaces may be written as
functors
\begin{equation*}
X:\{\text{commutative $\K$-algebras}\}\longra\{\text{sets}\},
\end{equation*}
and classical Artin $\K$-stacks may be written as functors
\begin{equation*}
X:\{\text{commutative $\K$-algebras}\}\longra\{\text{groupoids}\}.
\end{equation*}
Since simplicial sets (an $\iy$-category) generalize both sets (a
1-category) and groupoids (a 2-category), higher $\K$-stacks
generalize $\K$-schemes, algebraic $\K$-spaces, and Artin
$\K$-stacks.

To\"en and Vezzosi define the $\iy$-category $\dSt_\K$ of {\it
derived\/ $\K$-stacks\/} (or $D^-$-{\it stacks\/})
\cite[Def.~2.2.2.14]{ToVe1}, \cite[Def.~4.2]{Toen}. Objects $\bX$ in
$\dSt_\K$ are $\iy$-functors
\begin{equation*}
\bX:\{\text{simplicial commutative $\K$-algebras}\}\longra
\{\text{simplicial sets}\}
\end{equation*}
satisfying sheaf-type conditions. They also define a full
$\iy$-subcategory $\dArt_\K\ab\subset\dSt_\K$ of {\it derived
Artin\/ $\K$-stacks}, with better geometric properties.

There is a {\it truncation functor\/} $t_0:\dSt_\K\ra\St_\K$ from
derived stacks to (higher) stacks, which maps
$t_0:\dArt_\K\ra\Art_\K$, and a fully faithful left adjoint {\it
inclusion functor\/} $i:\St_\K\ra\dSt_\K$ mapping
$\Art_\K\ra\dArt_\K$. As $i$ is fully faithful we can regard it as
embedding $\St_\K,\Art_\K$ as full subcategories of $\dSt_\K,
\dArt_\K$. Thus, we can regard classical $\K$-schemes and Artin
$\K$-stacks as examples of derived $\K$-stacks.

The adjoint property of $i,t_0$ implies that for any $X\in\St_\K$
there is a natural morphism $X\ra t_0\ci i(X)$, which is an
equivalence as $i$ is fully faithful, and for any $\bX\in\dSt_\K$
there is a natural morphism $i\ci t_0(\bX)\ra \bX$, which we may
regard as embedding the classical truncation $X=t_0(\bX)$ of $\bX$
as a substack of $\bX$. On notation: we generally write derived
schemes and stacks $\bX,\bY,\ldots$ and their morphisms $\bs f,\bs
g,\ldots$ in bold, and classical schemes, stacks, and higher stacks
$X,Y,\ldots$ and morphisms $f,g,\ldots$ not in bold, and we will
write $X=t_0(\bX)$, $f=t_0(\bs f)$, and so on, for classical
truncations of derived $\bX,\bs f,\ldots.$

\subsection{Derived schemes and cdgas}
\label{da32}

To\"en and Vezzosi \cite{Toen,ToVe1,ToVe2} base their derived
algebraic geometry on {\it simplicial commutative\/ $\K$-algebras},
but we prefer to work with {\it commutative differential graded\/
$\K$-algebras\/} ({\it cdgas\/}). As in \cite[\S 8.1.4]{Luri2} there
is a {\it normalization functor\/}
\begin{equation*}
N:\{\text{simplicial commutative $\K$-algebras}\}\longra
\{\text{cdgas over $\K$}\}
\end{equation*}
which is an equivalence of $\iy$-categories, since $\K$ has
characteristic zero by our assumption in \S\ref{da1}. So, working
with simplicial commutative $\K$-algebras and with cdgas over $\K$
are essentially equivalent.

There is a {\it spectrum functor\/}
\begin{equation*}
\Spec:\{\text{commutative $\K$-algebras}\}\longra\St_\K.
\end{equation*}
An object $X$ in $\St_\K$ is called an {\it affine $\K$-scheme\/} if
it is equivalent to $\Spec A$ for some commutative $\K$-algebra $A$,
and a $\K$-{\it scheme\/} if it may be covered by Zariski open
$Y\subseteq X$ with $Y$ an affine $\K$-scheme. Write $\Sch_\K$ for
the full \ab($\iy$-\nobreak)subcategory of $\K$-schemes in~$\St_\K$.

Similarly, there is a {\it spectrum functor\/}
\begin{equation*}
\bSpec:\{\text{commutative differential graded $\K$-algebras}\}\longra
\dSt_\K.
\end{equation*}
A derived $\K$-stack $\bX$ is called an {\it affine derived\/
$\K$-scheme\/} if $\bX$ is equivalent in $\dSt_\K$ to $\bSpec A$ for
some cdga $A$ over $\K$. (This is true if and only if $\bX$ is
equivalent to $\bSpec{}^\De(A^\De)$ for some simplicial commutative
$\K$-algebra $A^\De$, as the normalization functor $N$ is an
equivalence.) As in To\"en \cite[\S 4.2]{Toen}, a derived $\K$-stack
$\bX$ is called a {\it derived\/ $\K$-scheme\/} if it may be covered
by Zariski open $\bY\subseteq\bX$ with $\bY$ an affine derived
$\K$-scheme. Write $\dSch_\K$ for the full $\iy$-subcategory of
derived $\K$-schemes in $\dSt_\K$. Then
$\dSch_\K\subset\dArt_\K\subset\dSt_\K$.

If $A$ is a cdga over $\K$, there is an equivalence $t_0\ci\bSpec
A\simeq\Spec H^0(A)$ in $\St_\K$. Hence, if $\bX$ is an affine
derived $\K$-scheme then $X=t_0(\bX)$ is an affine $\K$-scheme, and
if $\bX$ is a derived $\K$-scheme then $X=t_0(\bX)$ is a
$\K$-scheme, and the truncation functor maps
$t_0:\dSch_\K\ra\Sch_\K$. Also the inclusion functor maps
$i:\Sch_\K\ra\dSch_\K$. As in \cite[\S 4.2]{Toen}, one can show that
if $\bX$ is a derived Artin $\K$-stack and $t_0(\bX)$ is a
$\K$-scheme, then $\bX$ is a derived $\K$-scheme.

Let $\bX$ be a derived $\K$-scheme, and $X=t_0(\bX)$ the
corresponding classical truncation, a classical $\K$-scheme. Then as
in \S\ref{da31} there is a natural inclusion morphism $i_\bX:X\ra
\bX$, which embeds $X$ as a derived $\K$-subscheme in $\bX$. A good
analogy in classical algebraic geometry is this: let $Y$ be a
non-reduced scheme, and $Y^\red$ the corresponding reduced scheme.
Then there is a natural inclusion $Y^\red\hookra Y$ of $Y^\red$ as a
subscheme of $Y$, and we can think of $Y$ as an infinitesimal
thickening of $Y^\red$. In a similar way, we can regard a derived
$\K$-scheme $\bX$ as an infinitesimal thickening of its classical
$\K$-scheme~$X=t_0(\bX)$.

We shall assume throughout this paper that any derived $\K$-scheme
$\bX$ is {\it locally finitely presented}, which means that it can be
covered by Zariski open affine $\bY\simeq\bSpec A$, where $A$ is a
cdga over $\K$ of finite presentation.

Points $x$ of a derived $\K$-scheme $\bX$ are the same as points $x$
of $X=t_0(\bX)$. If $A$ is a standard form cdga, as in Example
\ref{da2ex2}, and $A'=A\ot_{A^0}A^0[f^{-1}]$ is a localization of
$A$, as in Definition \ref{da2def4}, then $\bSpec A'$ is the Zariski
open subset of $\bSpec A$ where $f\ne 0$. If $A'$ is a localization
of $A$ at $p\in\Spec (H^0(A))=t_0(\bSpec A)$, then $\bSpec A'$ is a
Zariski open neighbourhood of $p$ in~$\bSpec A$.

A morphism $\bs f:\bX\ra\bY$ of derived $\K$-schemes is called {\it \'etale\/} if it is Zariski locally modelled on $\bSpec \phi:\bSpec A\ra\bSpec B$, for $\phi:B\ra A$ a morphism of cdgas over $\K$ such that $H^0(\phi):H^0(B)\ra H^0(A)$ is \'etale in the classical sense and $H^i(\phi):H^i(B)\ra H^i(A)$ induces an isomorphism $H^i(B)\ra H^i(A)\ot_{H^0(A)}H^0(B)$ for all $i<0$. In particular, $t_0(\bs f):t_0(\bX)\ra t_0(\bY)$ is an \'etale morphism of classical schemes.

Since the $\iy$-category of affine derived $\K$-schemes is a
localization of the $\iy$-category of cdgas over $\K$ at
quasi-isomorphisms, morphisms of affine derived $\K$-schemes need
not lift to morphisms of the corresponding cdgas. We can ask:

\begin{quest} 
\label{da3quest}
Suppose $A,B$ are standard form cdgas over $\K,$ as
in Example\/ {\rm\ref{da2ex2},} and\/ $\bs f:\bSpec A\ra\bSpec B$ is
a morphism in the $\iy$-category $\dSch_\K$. Then
\begin{itemize}
\setlength{\itemsep}{0pt}
\setlength{\parsep}{0pt}
\item[{\bf(a)}] does there exist a morphism of cdgas $\al:B\ra A$ (that is, a strict morphism of cdgas, not a morphism in the $\iy$-category) with\/ $\bs f\simeq \bSpec\al${\rm ?}
\item[{\bf(b)}] for each\/ $p\in\Spec H^0(A),$ does there exist a
localization $A'$ of\/ $A$ at\/ $p$ and a (strict) morphism of cdgas
$\al:B\ra A'$ with\/ $\bs f\vert_{\bSpec A'}\simeq
\bSpec\al${\rm ?}
\end{itemize}
\end{quest}

One can show that:
\begin{itemize}
\setlength{\itemsep}{0pt}
\setlength{\parsep}{0pt}
\item[(i)] For general $A,B$, the answers to Question
\ref{da3quest}(a),(b) may both be no.
\item[(ii)] If $A$ is general, but for $B$, the smooth
$\K$-scheme $\Spec B^0$ is isomorphic to an affine space
$\bA^n$, then the answers to Questions \ref{da3quest}(a),(b) are
both yes. Indeed, the condition that $\Spec B^0$ be isomorphic to an affine
space ensures that $B$ is a cofibrant as a cdga, and therefore every map
$B \rightarrow A$ in the $\infty$-category of of cdgas is represented by a strict map
out of $B$.

\item[(iii)] If $A$ is general, but for $B$ the smooth
$\K$-scheme $\Spec B^0$ is isomorphic to a Zariski open set in
an affine space $\bA^n$, then the answer to Question
\ref{da3quest}(a) may be no, but to Question \ref{da3quest}(b)
is yes.
\item[(iv)] For general $A,B$, the answer to Question
\ref{da3quest}(b) at $p$ is yes if and only if there exists a
localization $A^{\prime\, 0}=A^0[g^{-1}]$ of $A^0$ at $p$ and a
morphism of $\K$-algebras $\al^0:B^0\ra A^{\prime\, 0}$ such
that the following commutes:
\begin{equation*}
\xymatrix@C=160pt@R=13pt{ *+[r]{\Spec \bigl(H^0(A)[g^{-1}]\bigr)}
\ar[r]_(0.56){t_0(\bs f)\vert_{\Spec(H^0(A)[g^{-1}])}} \ar[d]^{\rm inc}
& *+[l]{\Spec H^0(B)} \ar[d]_{\rm inc} \\ *+[r]{\Spec A^{\prime\, 0}}
\ar[r]^(0.56){\Spec \al^0} & *+[l]{\Spec B^0.\!{}} }
\end{equation*}
That is, if $\bs f$ lifts to a morphism $B^0\ra A^0$, then
(possibly after further localization) it lifts to a morphism
$B^k\ra A^k$ for all~$k$.
\end{itemize}

\begin{rem} 
\label{da3rem}
Our notion of {\it standard form\/} cdgas is a
compromise. We chose to start the induction in Example \ref{da2ex2}
with a smooth $\K$-algebra $A^0$. This ensured that cotangent
complexes of $\bL_A$ behave well, as in \S\ref{da23}. We could
instead have adopted one of the stronger conditions that $\Spec A^0$
is isomorphic to $\bA^n$, or to a Zariski open subset of $\bA^n$.
This would give better lifting properties of morphisms of derived
$\K$-schemes, as in (ii),(iii) above. However, requiring $\Spec
A^0\subseteq\bA^n$ open does not fit well with the notion of $A$
{\it minimal\/} in Definition \ref{da2def6}, and Theorem
\ref{da4thm1} would be false for these stronger notions of standard
form cdga.
\end{rem}

\subsection{Cotangent complexes of derived schemes and stacks}
\label{da33}

We discuss {\it cotangent complexes\/} of derived schemes and
stacks, following To\"en and Vezzosi \cite[\S 1.4]{ToVe1}, \cite[\S
4.2.4--\S 4.2.5]{Toen}, \cite{Vezz1}, and Lurie \cite[\S
3.4]{Luri1}. We will restrict our attention to derived Artin
$\K$-stacks $\bX$, rather than general derived $\K$-stacks, which
ensures in particular that cotangent complexes exist.

Let $\bX$ be a derived Artin $\K$-stack. In classical algebraic
geometry, cotangent complexes $\bL_X$ lie in $D(\qcoh(X))$, but in
derived algebraic geometry they lie in a dg category
$L_\qcoh(\bX)$ defined by To\"en \cite[\S 3.1.7, \S 4.2.4]{Toen},
which is a generalisation of  $D(\qcoh(X))$. Here are some properties of
these:
\begin{itemize}
\setlength{\itemsep}{0pt}
\setlength{\parsep}{0pt}
\item[(i)] If $\bX$ is a derived Artin $\K$-stack then $L_\qcoh(\bX)$ is
a pre-triangulated dg category with a t-structure whose heart is the
abelian category $\qcoh(t_{0}(\bX))$ of quasicoherent sheaves on the classical
truncation $t_{0}(\bX)$ of $\bX$.
In particular, if $\bX=i(X)$ is a classical Artin $\K$-stack or $\K$-scheme
then $L_\qcoh(X)\simeq D(\qcoh(X))$, but in general
$L_\qcoh(\bX)\not\simeq D(\qcoh(t_{0}(\bX)))$.
\item[(ii)] If $\bs f:\bX\ra\bY$ is a morphism of derived Artin
$\K$-stacks it induces a pullback functor $\bs
f^*:L_\qcoh(\bY)\ra L_\qcoh(\bX)$, analogous to the left derived
pullback functor $f^*:D(\qcoh(Y))\ra D(\qcoh(X))$ in the
classical case.
\item[(iii)] If $A$ is a cdga over $\K$ and $\bX=\bSpec A$ is the corresponding derived affine $\K$-scheme then $L_\qcoh(\bX)\simeq\text{dgmod-}A$ is the pre-triangulated dg category of dg-modules over $A$. In contrast, $D(\qcoh(\bX))\simeq D(\text{mod-}H^0(A))$, which depends only on the classical truncation $X=t_0(\bX)$.
\item[(iv)] There is a notion of when a complex $\cE^\bu$ in $L_\qcoh(\bX)$ is {\it perfect}. When $\bX=\bSpec A$ is affine as in (iii), perfect complexes in $L_\qcoh(\bX)$ correspond to {\it perfect dg-modules\/} in $\text{dgmod-}A$, that is, to modules which may be built from finitely many copies of $A[k]$ for $k\in\Z$ by repeated extensions and splitting of idempotents.
\end{itemize}

If $\bX$ is a derived Artin $\K$-stack, then To\"en and Vezzosi
\cite[\S 4.2.5]{Toen}, \cite[\S 1.4]{ToVe1} or Lurie \cite[\S
3.2]{Luri1} define an ({\it absolute\/}) {\it cotangent complex\/}
$\bL_\bX$ in $L_\qcoh(\bX)$. If $\bs f:\bX\ra\bY$ is a morphism of
derived Artin stacks, they construct a morphism $\bL_{\bs f}:\bs
f^*(\bL_\bY)\ra\bL_\bX$ in $L_\qcoh(\bX)$, and the ({\it
relative\/}) {\it cotangent complex\/} $\bL_{\bX/\bY}$ is defined to
be the cone on this, giving the distinguished triangle
\begin{equation*}
\xymatrix@C=30pt{ \bs f^*(\bL_\bY) \ar[r]^(0.55){\bL_{\bs f}} &
\bL_\bX \ar[r] & \bL_{\bX/\bY} \ar[r] & \bs f^*(\bL_\bY)[1]. }
\end{equation*}

Here are some properties of these:
\begin{itemize}
\setlength{\itemsep}{0pt}
\setlength{\parsep}{0pt}
\item[(a)] If $\bX$ is a derived $\K$-scheme then
$h^k(\bL_\bX)=0$ for $k>0$. Also, if $X=t_0(\bX)$ is the
corresponding classical $\K$-scheme, then $h^0(\bL_\bX)\cong
h^0(\bL_X)\cong\Om_X$ under the equivalence $\qcoh(\bX)\simeq
\qcoh(X)$.
\item[(b)] Let $\bX\,{\buildrel\bs f\over\longra}\,\bY\,
{\buildrel\bs g\over\longra}\,\bZ$ be morphisms of derived Artin
$\K$-stacks. Then by Lurie \cite[Prop.~3.2.12]{Luri1} there is a
distinguished triangle in~$L_\qcoh(\bX)$:
\begin{equation*}
\xymatrix@C=30pt{ \bs f^*(\bL_{\bY/\bZ}) \ar[r] &
\bL_{\bX/\bZ} \ar[r] & \bL_{\bX/\bY} \ar[r] & \bs f^*(\bL_{\bY/\bZ})[1]. }
\end{equation*}
\item[(c)] Suppose we have a Cartesian diagram of derived Artin
$\K$-stacks:
\begin{equation*}
\xymatrix@C=90pt@R=14pt{ *+[r]{\bW} \ar[r]_(0.3){\bs f} \ar[d]^{\bs e}
& *+[l]{\bY} \ar[d]_{\bs h} \\
*+[r]{\bX} \ar[r]^(0.7){\bs g} & *+[l]{\bZ.\!\!{}} }
\end{equation*}
Then by To\"en and Vezzosi \cite[Lem.s 1.4.1.12 \&
1.4.1.16]{ToVe1} or Lurie \cite[Prop.~3.2.10]{Luri1} we have
{\it base change isomorphisms\/}:
\e
\bL_{\bW/\bY}\cong\bs e^*(\bL_{\bX/\bZ}),\qquad
\bL_{\bW/\bX}\cong\bs f^*(\bL_{\bY/\bZ}).
\label{da3eq1}
\e
Note that the analogous result for classical $\K$-schemes
requires $g$ or $h$ to be flat, but \eq{da3eq1} holds with no
flatness assumption.
\item[(d)] Suppose $\bX$ is a {\it locally finitely presented\/}
derived Artin $\K$-stack \cite[\S 2.2.3]{ToVe1} (also called
{\it locally of finite presentation\/} in \cite[\S 3.1.1]{Toen},
and {\it fp-smooth\/} in \cite[\S 4.4]{ToVe2}). Then $\bL_\bX$
is a perfect complex in $L_\qcoh(\bX)$.

If $\bs f:\bX\ra\bY$ is a morphism of locally finitely presented
derived Artin $\K$-stacks, then $\bL_{\bX/\bY}$ is perfect.
\item[(e)] Suppose $\bX$ is a derived $\K$-scheme, $X=t_0(\bX)$
its classical truncation, and $i_\bX:X\hookra\bX$ the natural
inclusion. Then Sch\"urg, To\"en and Vezzosi
\cite[Prop.~1.2]{STV} show that $h^k(\bL_{X/\bX})=0$ for $k\ge
-1$, and deduce that $\bL_{i_\bX}:i_\bX^*(\bL_\bX)\ra \bL_X$ in
$D(\qcoh(X))$ is an {\it obstruction theory\/} on $X$ in the
sense of Behrend and Fantechi \cite{BeFa}.
\end{itemize}

Now suppose $A$ is a cdga over $\K$, and $\bX$ a derived $\K$-scheme
with $\bX\simeq\bSpec A$ in $\dSch_\K$. Then we have an equivalence
of triangulated categories $L_\qcoh(\bX)\simeq \ab D(\mathop{\rm
mod}A)$, where $D(\mathop{\rm mod} A)$ is the derived category of
dg-modules over $A$. This equivalence identifies cotangent complexes
$\bL_\bX\simeq\bL_A$. If $A$ is of standard form, then as in
\S\ref{da23} the K\"ahler differentials $\Om^1_A$ are a model for
$\bL_A$ in $D(\mathop{\rm mod}A)$, and Proposition \ref{da2prop}
gives a simple explicit description of $\Om^1_A$. Thus, if $\bX$ is
a derived $\K$-scheme with $\bX\simeq\bSpec A$ for $A$ a standard
form cdga, we can understand $\bL_\bX$ well. We will use this to do
computations with $k$-{\it shifted\/ $p$-forms\/} and $k$-{\it
shifted closed\/ $p$-forms\/} on $\bX$, as in~\S\ref{da34}.

\subsection{\texorpdfstring{$k$-shifted symplectic structures on derived stacks}{k-shifted symplectic structures on derived stacks}}
\label{da34}

We now outline the main ideas of our principal reference Pantev,
To\"en, Vaqui\'e and Vezzosi \cite{PTVV}. Let $X$ be a classical
smooth $\K$-scheme. Its tangent and cotangent bundles $TX,T^*X$ are
vector bundles on $X$. A $p$-{\it form\/} on $X$ is a global section
$\om$ of $\La^pT^*X$. There is a de Rham differential
$\dd:\La^pT^*X\ra \La^{p+1}T^*X$, which is a morphism of sheaves of
$\K$-vector spaces but not of sheaves of $\O_X$-modules, and induces
$\dd:H^0(\La^pT^*X)\ra H^0(\La^{p+1}T^*X)$. A {\it closed\/
$p$-form\/} is $\om\in H^0(\La^pT^*X)$ with $\dd\om=0$. A {\it
symplectic form\/} on $X$ is a closed 2-form $\om$ such that the
induced morphism $\om:TX\ra T^*X$ is an isomorphism.

The {\it derived loop stack\/} $\cL X$ of $X$ is the fibre product
$X\t_{\De_X,X\t X,\De_X}X$ in $\dSch_\K$, where $\De_X:X\ra X\t X$
is the diagonal map. When $X$ is smooth, $\cL X$ is a quasi-smooth
derived $\K$-scheme. It is interpreted in derived algebraic geometry
as $\cL X=\R{\bf Map}({\cal S}^1,X)$, the moduli stack of `loops' in
$X$. Here the circle ${\cal S}^1$ may be thought of as the
simplicial complex $\xymatrix@C=15pt{\bu \ar@<.3ex>@/^4pt/[r]
\ar@<-.3ex>@/_4pt/[r] & \bu}$, so a map from ${\cal S}^1$ to $X$
means two points $x,y$ in $X$, corresponding to the two vertices
`$\bu$' in ${\cal S}^1$, plus two relations $x=y$, $x=y$
corresponding to the two edges `$\ra$' in ${\cal S}^1$. This agrees
with $\cL X$, as points of $X\t_{\De_X,X\t X,\De_X}X$ correspond to
points $x,y$ in $X$ satisfying the relation $(x,x)=(y,y)$ in $X\t
X$. Note that in derived algebraic geometry, imposing a relation
twice is not the same as imposing it once, as the derived structure
sheaf records the relations in its simplicial structure.

Consider the projection $p: \cL X=X\t_{X\t X}X \rightarrow X$, where $X$ is a classical smooth $\K$-scheme. By the Hochschild--Kostant--Rosenberg Theorem, there is a decomposition $p_{*}(\O_{\cL X})\cong\bigoplus_{p} \La^pT^*X[p]$
in $D(\qcoh(X))$. Thus, $p$-forms on $X$ may be interpreted as `functions on the loop space $\cL X$ of weight $p\,$', while closed $p$-forms can be interpreted as `${\cal S}^1$-invariant functions on the loop space $\cL X$ of weight~$p\,$' by identifying the ${\cal S}^{1}$-action with the de Rham differential. See To\"en and Vezzosi \cite{ToVe3} and Ben-Zvi and Nadler~\cite{BZNa}.

Now let $\bX$ be a locally finitely presented derived Artin
$\K$-stack. The aim of Pantev et al.\ \cite{PTVV} is to define good
notions of $p$-{\it forms}, {\it closed\/ $p$-forms}, and {\it
symplectic structures\/} on $\bX$, and show that these occur
naturally on certain derived moduli stacks. The rough idea is as
above: we form the derived loop stack $\cL \bX:=\bX\t_{\De_\bX,\bX\t
\bX,\De_\bX}\bX$ in $\dSch_\K$, and then (closed) $p$-forms on $X$
are `(${\cal S}^1$-invariant) functions on $\cL\bX$ of
weight~$p\,$'.

However, the problem is more complicated than the classical smooth
case in two respects. First, as the $p$-forms $\La^p\bL_\bX$ are a
complex, rather than a vector bundle, one should consider cohomology
classes $H^k(\bX,\La^p\bL_\bX)$ for $k\in\Z$ rather than just global
sections $\om\in H^0(\bX,\La^pT^*\bX)$. This leads to the idea of a
$p$-{\it form of degree\/ $k$ on\/} $\bX$ for $p\ge 0$ and $k\in\Z$,
which is roughly an element of $H^k(\bX,\La^p\bL_\bX)$. Essentially,
$\O_{\cL\bX}$ has two gradings, the cohomological grading $k$ and the grading by 
weight $p$.

Secondly, as the ${\cal S}^1$-action on $\cL\bX=\R{\bf Map}({\cal
S}^1,\bX)$ by rotating the domain ${\cal S}^1$ is only up to
homotopy, to be `${\cal S}^1$-invariant' is not a property of a
function on $\cL\bX$, but an extra
structure. As an analogy, if an algebraic $\K$-group $G$ acts on a
$\K$-scheme $V$, then for a vector bundle $E\ra V$ to be
$G$-invariant is not really a property of $E$: the right question to
ask is whether the $G$-action on $V$ lifts to a $G$-action on $E$,
and there can be many such lifts, each a different way for $E$ to be
$G$-invariant. Similarly, the definition of `closed $p$-form on
$\bX$' in \cite{PTVV} is not just a $p$-form satisfying an extra
condition, but a $p$-form with extra data, a `closing structure',
satisfying some conditions.

In fact, for technical reasons, for a general locally finitely
presented derived Artin $\K$-stack $\bX$, Pantev et al.\ \cite{PTVV}
do not define $p$-forms as elements of $H^{-p}(\O_{\cL\bX})$, and so
on, as we have sketched above. Instead, they first define explicit
notions of (closed) $p$-forms on affine derived $\K$-schemes
$\bY=\bSpec A$ which are spectra of commutative differential graded
algebras (cdgas) $A$, and show these satisfy \'etale descent. Then for general $\bX$, $k$-shifted (closed) $p$-forms are defined as a mapping stack; basically, a $k$-shifted (closed) $p$-form $\om$ on $\bX$ is the functorial choice for all $\bY,\bs f$ of a $k$-shifted (closed) $p$-form $\bs f^*(\om)$ on $\bY$ whenever $\bY=\bSpec A$ is affine and $\bs f:\bY\ra\bX$ is a morphism.

The families (simplicial sets) of $p$-forms and of closed $p$-forms
of degree $k$ on $\bX$ are written $\cA^p_\K(\bX,k)$ and
$\cA^{p,\cl}_\K(\bX,k)$, respectively. There is a morphism
$\pi:\cA^{p,\cl}_\K(\bX,k)\ra\cA^p_\K(\bX,k)$, which is in general
neither injective nor surjective.

A 2-form $\om^0$ of degree $k$ on $\bX$ induces a morphism
$\om^0:\bT_\bX\ra\bL_\bX[k]$ in $L_\qcoh(\bX)$. We call $\om^0$ {\it
nondegenerate\/} if $\om^0:\bT_\bX\ra\bL_\bX[k]$ is an equivalence.
As in \cite[Def.~0.2]{PTVV}, a closed 2-form $\om$ of degree $k$ on
$\bX$ for $k\in\Z$ is called a $k$-{\it shifted symplectic
structure\/} if the corresponding 2-form $\om^0=\pi(\om)$ is
nondegenerate.

A 0-shifted symplectic structure on a classical $\K$-scheme is a
equivalent to a classical symplectic structure. Pantev et al.\
\cite{PTVV} construct $k$-shifted symplectic structures on several
classes of derived moduli stacks. In particular, if $Y$ is a
Calabi--Yau $m$-fold and $\bs\cM$ a derived moduli stack of coherent
sheaves or perfect complexes on $Y$, then $\bs\cM$ has a
$(2-m)$-shifted symplectic structure.

\section{Local models for derived schemes}
\label{da4}

The next theorem, based on Lurie \cite[Th.~8.4.3.18]{Luri2} and
proved in \S\ref{da41}, says every derived $\K$-scheme $\bX$ is
Zariski locally modelled on $\bSpec A$ for $A$ a minimal standard
form cdga. Recall that all derived $\K$-schemes in this paper are
assumed {\it locally finitely presented}. A morphism $\bs
f:\bX\ra\bY$ in $\dSch_\K$ is called a {\it Zariski open
inclusion\/} if $\bs f$ is an equivalence from $\bX$ to a Zariski
open derived $\K$-subscheme $\bY'\subseteq\bY$. A morphism of cdgas
$\al:B\ra A$ is a {\it Zariski open inclusion\/} if
$\bSpec\al:\bSpec A\ra\bSpec B$ is a Zariski open inclusion.

\begin{thm} 
\label{da4thm1}
Let\/ $\bX$ be a locally finitely presented derived\/ $\K$-scheme, and\/
$x\in\bX$. Then there exist a standard form cdga $A$ over $\K$ which
is minimal at a point\/ $p\in\Spec H^0(A),$ in the sense of
Example\/ {\rm\ref{da2ex2}} and Definitions\/ {\rm\ref{da2def4}}
and\/ {\rm\ref{da2def6},} and a morphism $\bs f:\bSpec A\ra\bX$ in
$\dSch_\K$ which is a Zariski open inclusion with\/~$\bs f(p)=x$.
\end{thm}

We think of $A,\bs f$ in Theorem \ref{da4thm1} as like a coordinate
system on $\bX$ near $x$. As well as being able to choose
coordinates near any point, we want to be able to compare different
coordinate systems on their overlaps. That is, given local
equivalences $\bs f:\bSpec A\ra\bX$, $\bs g:\bSpec B\ra\bX$, we
would like to compare the cdgas $A,B$ on the overlap of their images
in~$\bX$.

This is related to the discussion after Question \ref{da3quest} in
\S\ref{da32}: for general $A,B$ we cannot (even locally) find a cdga
morphism $\al:B\ra A$ with $\bs f\simeq\bs g\ci\bSpec\al$. However,
the next theorem, proved in \S\ref{da42}, shows we can find a third
cdga $C$ and open inclusions $\al:A\ra C$, $\be:B\ra C$ with~$\bs
f\ci\bSpec\al\simeq\bs g\ci\bSpec\be$.

\begin{thm} 
\label{da4thm2}
Let\/ $\bX$ be a locally finitely presented derived\/ $\K$-scheme, $A,B$ be
standard form cdgas over $\K,$ and\/ $\bs f:\bSpec A\ra\bX,$ $\bs
g:\bSpec B\ra\bX$ be Zariski open inclusions in $\dSch_\K$. Suppose
$p\in\Spec H^0(A)$ and\/ $q\in\Spec H^0(B)$ with\/ $\bs f(p)=\bs
g(q)$ in $\bX$. Then there exist a standard form cdga $C$ over $\K$
which is minimal at\/ $r$ in $\Spec H^0(C)$ and morphisms of cdgas
$\al:A\ra C,$ $\be:B\ra C$ which are Zariski open inclusions, such
that\/ $\bSpec\al:r\mapsto p,$ $\bSpec\be:r\mapsto q,$ and\/ $\bs
f\ci\bSpec\al\simeq\bs g\ci\bSpec\be$ as morphisms $\bSpec C\ra\bX$
in~$\dSch_\K$.

If instead\/ $\bs f,\bs g$ are \'etale rather than Zariski open
inclusions, the same holds with\/ $\al,\be$ \'etale rather than
Zariski open inclusions.
\end{thm}

\subsection{Proof of Theorem \ref{da4thm1}}
\label{da41}

The proof is a variation on Lurie \cite[Th.~8.4.3.18]{Luri2}. We
give an outline, referring the reader to \cite{Luri2} for more
details.

As $\bX$ is a locally finitely presented derived $\K$-scheme, it is
covered by affine finitely presented derived $\K$-subschemes. So we
can choose an open neighbourhood $\bY$ of $x$ in $\bX$, a cdga $B$
of finite presentation over $\K$, and an equivalence $\bs g:\bSpec
B\ra\bY\subseteq\bX$. There is then a unique $q\in\Spec H^0(B)$
with~$\bs g(q)=x$. 

Since $B$ is of finite presentation, the cotangent complex $\bL_B$
has finite Tor-amplitude, say in the interval $[-n,0]$. We will show
that a localization $B'$ of $B$ at $q$ is equivalent to a standard
form cdga $A=A(n)$ over $\K$, constructed inductively in a sequence
$A(0),A(1),\ldots,A(n)$ as in Example \ref{da2ex2}, with $A$ minimal
at the point $p\in\Spec H^0(A)$ corresponding to $q\in\Spec
H^0(B')\subseteq\Spec H^0(B)$. For the inductive construction, we do
not fix a particular model for the cdga $B$, but understand it as an
object in the $\iy$-category of cdgas. Similarly, when we assert
the existence of a map $A(k) \ra B$ from some level of the inductive 
construction to $B$, this map should be understood as a map in 
the $\iy$-category of cdgas. 

Let $m_0=\dim \Om^1_{H^0(B)}\vert_q$, the embedding dimension of
$\Spec H^0(B)$ at $q$. Localizing $H^0(B)$ at $q$ if necessary, we
can find a smooth algebra $A(0)$ of dimension $m_0$, an ideal
$I\subset A(0)$, and an isomorphism of algebras $A(0)/I\ra H^0(B)$
such that the induced surjection of modules $\Om^1_{A(0)}\ot_{A(0)}
H^0(B)\ra\Om^1_{H^0(B)}$ is an isomorphism at $q$. Geometrically, we
have chosen an embedding of $\Spec H^0(B)$ into a smooth scheme
$\Spec A(0)$ that is minimal at $q$. Let $p\in\Spec A(0)$ be the
image of~$q\in\Spec H^0(B)$.

Since $B$ is the homotopy limit of its Postnikov tower $\cdots\ra\tau_{\geq -1}B\ra \tau_{\geq 0}B\simeq H^{0}(B)$ in which each map is a square-zero extension of cdgas \cite[Prop.~7.1.3.19]{Luri2}, and since $A(0)$ is smooth and hence maps out of it can be lifted along square-zero extensions, we can lift the surjection $A(0)\ra H^0(B)$ along the canonical map $B\ra H^0(B)$ to obtain a map $\al^0:A(0)\ra B$ 
which is a surjection on $H^0$. Consider the fibre sequence $\fib(\al^0)\ra A(0)\ra B$. One can show that there is an isomorphism $H^0(B)\ot_{A(0)}H^0(\fib(\al^0))\cong
H^{-1}(\bL_{B/A(0)})$. Furthermore, we see that there is a
surjection $H^0(\fib(\al^0))\twoheadrightarrow I$ and hence a
surjection $H^{-1}(\bL_{B/A(0)})=H^0(B)\ot_{A(0)} H^0(\fib(\al^0))
\twoheadrightarrow I/I^2$.

Localizing if necessary and using Nakayama's Lemma, we may choose a
free finite rank $A(0)$-module $M^{-1}$ together with a surjection
$M^{-1}\ra H^{-1}(\bL_{B/A(0)})$ that is an isomorphism at $p$. We
therefore obtain a surjection $M^{-1} \ra I/I^2$ which can be lifted
through the map $I\ra I/I^2$. Localizing again if necessary and
using Nakayama's Lemma, we may assume the lift $M^{-1}\ra I$ is
surjective.

Using this choice of $M^{-1}$ together with the map $M^{-1}\ra I
\subset A(0)$, we define a cdga $A(1)$ as in Example \ref{da2ex2}.
Note that by construction, the induced map $\al^{-1}:A(1)\ra B$
induces an isomorphism $H^0(A(1))\ra H^0(B)$.

Now consider the fibre sequence $B\ot_{A(1)}\bL_{A(1)/A(0)}\ra
\bL_{B/A(0)}\ra\bL_{B/A(1)}$. By construction, $\bL_{A(1)/A(0)}
\simeq A(1)\ot_{A(0)}M^{-1}[1]$, and the map $H^0(B)
\ot_{A(0)}M^{-1}=H^{-1}(\bL_{A(1)/A(0)})\ra H^{-1}(\bL_{B/A(0)})$ is
surjective, so $H^i(\bL_{B/A(1)})$ vanishes for $i=0,-1$. One can
then deduce that there are isomorphisms
\begin{equation*}
H^{-1}(\fib(\al^{-1})) \cong H^{-2}({\rm cof}(\al^{-1}))\cong
H^{-2}(\bL_{B/A(1)}).
\end{equation*}

Localizing if necessary, we can choose a free $A(1)$-module $M^{-2}$
of finite rank together with a surjection $H^0(B) \ot_{H^0(A(1))}
H^0(M^{-2}) \ra H^{-1}(\fib(\al^{-1})) \cong H^{-2}(\bL_{B/A(1)})$
that is an isomorphism at $p$. Choosing a lift $M^{-2}[1] \ra
\fib(\al^{-1})\allowbreak \ra A(1)$, we construct $A(2)$ from $A(1)$
as in Example \ref{da2ex2}.

Continuing in this manner, we construct a sequence of cdgas
\begin{equation*}
A(0) \ra A(1) \ra \cdots \ra A(k-1)
\ra A(k) \ra \cdots \ra A(n-1)
\end{equation*}
so that $\bL_{A(k)/A(k-1)}\simeq A(k) \ot_{A(k-1)} M^{-k}[k]$ is a
free $A(k)$-module and the natural map $H^{-k}(\bL_{A(k)/A(k-1)})
\ra H^{-k}(\bL_{B/A(k-1)})$ is surjective and is an isomorphism at
$p$.

By induction on $k$, we see that $\bL_{A(k)}$ has Tor-amplitude in
$[-k,0]$. Considering the fibre sequence $B\ot_{A(n-1)}\bL_{A(n-1)}
\ra\bL_B\ra\bL_{B/A(n-1)}$ and bearing in mind that $\bL_B$ has
Tor-amplitude in $[-n,0]$ by assumption, we see that
$\bL_{B/A(n-1)}$ has Tor-amplitude in $[-n,0]$. But by \eq{da2eq5},
$H^i(\bL_{B/A(n-1)})=0$ for $i>-n$, so we see that
$\bL_{B/A(n-1)}[-n]$ is a projective $B$-module of finite rank.
Localizing if necessary, we may assume it is free.

Finally, choose a free $A(n-1)$-module $M^{-n}$ with $\bL_{B/A(n-1)}
\simeq B\ot_{A(n-1)}M^{-n}[n]$. Then $H^0(M^{-n})\cong
H^{-n}(\bL_{B/A(n-1)}) \cong H^{1-n}(\fib(\al^{1-n}))$, so using
$M^{-n}[n-1]\ra A(n-1)$ to build $A(n)$, we have a fibre sequence
$\bL_{A(n)/A(n-1)}\ra \bL_{B/A(n-1)}\ra\bL_{B/A(n)}$ in which the
first arrow is an equivalence. Thus $\bL_{B/A(n)}\simeq 0$, and
since the map $\al^{-n}:A(n)\ra B$ induces an isomorphism on $H^0$,
it must be an equivalence. Set $A=A(n)$, so that $A$ is a standard
form cdga over $\K$. As $\bSpec\al^{-n}:\bSpec B\ra\bSpec A$ is an
equivalence in $\dSch_\K$ with $\bSpec\al^{-n}:q\mapsto p$, there
exists a quasi-inverse $\bs h:\bSpec A\ra\bSpec B$ for
$\bSpec\al^{-n}$, with $\bs h(p)=q$. Then $\bs f=\bs g\ci\bs
h:\bSpec A\ra\bX$ is a Zariski open inclusion with $\bs f(p)=\bs
g\ci\bs h(p)=\bs g(q)=x$, as we have to prove.

It remains to show that $A$ is minimal at $p$. For this we must
check that for each $k$, the composition \eq{da2eq8} is zero at $p$.
Using the commutative diagram \eq{da2eq6}, and \eq{da2eq6} with
$k-1$ in place of $k$, we have a commutative diagram
\begin{equation*}
\xymatrix@R=13pt@C=130pt{*+[r]{H^{-k}(\bL_{A(k)/A(k-1)})}
\ar@{=}[r] \ar[d] & *+[l]{H^{-k}(\bL_{A(k)/A(k-1)})} \ar[d] \\
*+[r]{H^{-k}(\bL_{B/A(k-1)})} \ar[r] \ar[d] &
*+[l]{H^{1-k}(\bL_{A(k-1)})} \ar[d] \\
*+[r]{H^{1-k}(\bL_{A(k-1)/A(k-2)})} \ar@{=}[r] &
*+[l]{H^{1-k}(\bL_{A(k-1)/A(k-2)}),\!\!{}}}
\end{equation*}
where the right hand column is \eq{da2eq8}. It is therefore enough
to see that composition in the left hand column is zero at~$p$.

But by construction, the first map in \eq{da2eq8} is an isomorphism
at $p$, so we need the second map to be zero at $p$. But again by
construction, we have an exact sequence $H^{-k}(\bL_{B/A(k-1)})\ra
H^{1-k}(\bL_{A(k-1)/A(k-2)}) \ra H^{1-k}(\bL_{B/A(k-2)})$ in which
the second map is an isomorphism at $p$, and hence
$H^{-k}(\bL_{B/A(k-1)}) \ra H^{1-k}(\bL_{A(k-1)/A(k-2)})$ is indeed
zero at $p$. This completes the proof.

\subsection{Proof of Theorem \ref{da4thm2}}
\label{da42}

Given a derived scheme $\bX$ with a point $x\in\bX$ and Zariski open
neighbourhoods $\bs f:\bU=\bSpec A\ra\bX$, and $\bs g:\bV=\bSpec
B\ra\bX$ of $x$, we may assume that $A$ and $B$ are given as
standard form cdgas. We want to show that we can find a Zariski open
neighbourhood $\bs h:\bW=\bSpec C\ra\bX$ of $x$ contained in $\bU$
and $\bV$, where $C$ is a standard form cdga minimal at $r\in\Spec
H^0(C)$ with $\bs h(r)=x$ and such that in the homotopy commutative
diagram
\e
\begin{gathered}
\xymatrix@!0@C=60pt@R=30pt{ && \bW=\bSpec C \ar[dd]_{\bs h}
\ar[dl]^(0.4){\bs c} \ar[dr]_(0.4){\bs d} \\
& \bU=\bSpec A \ar[dl] \ar[dr]_(0.4){\bs f}
&& \bV=\bSpec B \ar[dl]^(0.4){\bs g} \ar[dr] \\
\Spec A(0) && \bX && \Spec B(0), }
\end{gathered}
\label{da4eq1}
\e
the maps $\bs c$ and $\bs d$ are induced by maps of cdgas $\al:A\ra
C$ and $\be:B\ra C$.

To begin with, choose a Zariski open immersion of an affine derived
scheme $\bW\ra\bU\t_\bX\bV\ra\bX$ whose image contains $x$. From the
above commutative diagram, we have maps $\bW\ra\bSpec A\ra \Spec
A(0)$ and $\bW\ra\bSpec B\ra \Spec B(0)$ in which the first arrows
in each sequence are open immersions and the second arrows are
closed immersions on the underlying classical schemes. The induced
map $\bW\ra\Spec A(0)\t\Spec B(0)$ is therefore a locally closed
immersion on the underlying classical scheme $W=t_0(\bW)$.
Localizing if necessary, we may assume that in fact $\bW \ra \Spec
A(0)\t\Spec B(0)$ factors via a closed immersion $W\hookra\Spec
C(0)$ into a smooth, affine, locally closed $\K$-subscheme $\Spec
C(0)\subseteq\Spec A(0)\t\Spec B(0)$ of minimal dimension
at~$r:=(p,q)$.

Proceeding as in the proof of Theorem \ref{da4thm1}, we may build a
standard form model $C$ for $\bW=\Spec C$ that is free over $C(0)$
and minimal at $x$. Since $A$ is free over $A(0)$ the map
$\bW=\bSpec C\ra\Spec C(0)\ra\Spec A(0)$ factors through $\bs
c:\bSpec C\ra\bSpec A$, and the latter is induced by an actual map
of cdgas $\al:A\ra C$. Similarly, the map $\bs d:\bSpec C\ra\bSpec
B$ is induced by an actual map of cdgas $\be:B\ra C$, since $B$ is
free over $B(0)$. This proves the first part of
Theorem~\ref{da4thm2}.

For the second part, if instead $\bs f:\bU=\bSpec A\ra\bX$ and $\bs
g:\bV=\bSpec B\ra\bX$ are \'etale neighbourhoods of $x$, we apply
the same fibre product construction. Since \'etale maps are stable
under pullbacks, $\bs c$ and $\bs d$ are now \'etale, rather than
Zariski open inclusions. However, the induced map $\bW\ra\Spec
A(0)\t\Spec B(0)$ is still a locally closed immersion on
$W=t_0(\bW)$. So in the same way we obtain a standard form model $C$
for $\bW=\bSpec C$ that is free over $C(0)$ and minimal at $r=(p,q)$
with \'etale maps $\bs c,\bs d$ that are induced by actual maps of
cdgas.

\section{A derived Darboux Theorem}
\label{da5}

Section \ref{da51} explains what is meant by a $k$-shifted
symplectic structure $\om=(\om^0,\om^1,\om^2,\ldots)$ on an affine
derived $\K$-scheme $\bSpec A$ for $A$ of standard form and $k<0$,
expanding on \S\ref{da34}, and \S\ref{da52} proves that up to
equivalence we can take $\om=(\dd\phi,0,0,\ldots)$. Sections
\ref{da53}--\ref{da54} define standard models for $k$-shifted
symplectic structures on $\bSpec A$, which we call `Darboux form'.

Our main result Theorem \ref{da5thm1}, stated in \S\ref{da55} and
proved in \S\ref{da56}, says that every $k$-shifted symplectic
derived $\K$-scheme $(\bX,\ti\om)$ for $k<0$ is locally equivalent
to $(\bSpec A,\om)$ for $A,\om$ in Darboux form. Section \ref{da57} explains how to compare different Darboux
form presentations of $(\bX,\ti\om)$ on their overlaps.

\subsection{\texorpdfstring{$k$-shifted symplectic structures on $\bSpec A$}{k-shifted symplectic structures on Spec A}}
\label{da51}

Let $A=A(n)$ be a standard form cdga over $\K$, as in Example
\ref{da2ex2}, constructed from a smooth $\K$-algebra $A(0)$ and free
finite rank modules $M^{-1},\ab M^{-2},\ab\ldots,\ab M^{-n}$ over
$A(0),\ldots,A(n-1)$. Write $\bX=\bSpec A$ for the corresponding
affine derived $\K$-scheme. We will explain in more detail how the
material of \S\ref{da34} works out for~$\bX=\bSpec A$.

As in \S\ref{da33} there is an equivalence $L_\qcoh(\bX)\simeq \ab
D(\mathop{\rm mod}A)$ which identifies $\bL_\bX\simeq\bL_A$, and as
$A$ is of standard form, as in \S\ref{da23} the K\"ahler
differentials $\bigl(\Om^1_A,\d\bigr)$ are a model for $\bL_A$ in
$D(\mathop{\rm mod}A)$, and Proposition \ref{da2prop} gives an
explicit description of $\bigl(\Om^1_A,\d\bigr)$. Write
$\Om^1_A=\bigop_{k=-\iy}^0(\Om^1_A)^k$ for the decomposition of
$\Om^1_A$ into graded pieces, so that the differential maps
$\d:(\Om^1_A)^k\ra(\Om^1_A)^{k+1}$.

As in \S\ref{da21} and \S\ref{da23}, the {\it de Rham algebra\/} of
$A$ is
\begin{equation*}
\DR(A):=\Sym_A(\Om^1_A[1])=\bigop_{p=0}^\iy\La^p \Om^1_A[p]=
\bigop_{p=0}^\iy\bigop_{k=-\iy}^0(\La^p\Om^1_A)^k[p].
\end{equation*}
It has two gradings, {\it degree\/} and {\it weight}, where the
component $(\La^p\Om^1_A)^k[p]$ has degree $k-p$ and weight $p$. It
has two differentials, the {\it internal differential\/}
$\d:\DR(A)\ra \DR(A)$ of degree 1 and weight 0, so that $\d$ maps
\begin{equation*}
\d:(\La^p\Om^1_A)^k[p]\longra(\La^p\Om^1_A)^{k+1}[p],
\end{equation*}
and the {\it de Rham differential\/} $\dd:\DR(A)\ra \DR(A)$ of
degree $-1$ and weight 1, so that $\dd$ maps
\begin{equation*}
\dd:(\La^p \Om^1_A)^k[p]\longra (\La^{p+1}\Om^1_A)^k[p+1].
\end{equation*}
They satisfy $\d\ci\d=\dd\ci\dd=\d\ci\dd+\dd\ci\d=0$. The
multiplication `$\,\cdot\,$' on $\DR(A)$ maps
\begin{equation*}
(\La^p \Om^1_A)^k[p]\t (\La^p
\Om^1_A)^l[q]\longra (\La^{p+q}\Om^1_A)^{k+l}[p+q].
\end{equation*}

In \cite[Def.~1.7]{PTVV}, Pantev et al.\ define a simplicial set
$\cA^p_\K(\bX,k)$ of $p$-{\it forms of degree\/} $k\in\Z$ on the
derived $\K$-scheme $\bX=\bSpec A$ by
\e
\cA^p_\bX(Y,k)=\bmd{\La^p\bL_A[k]}.
\label{da5eq1}
\e
In our case we may take $\La^p\bL_A=\La^p\Om^1_A$. If $E^\bu$ is a
complex of $\K$-modules, then $\md{E^\bu}$ means: take the
truncation $\tau_{\le 0}(E^\bu)$, and turn it into a simplicial set
via the Dold--Kan correspondence. To avoid dealing with simplicial
sets, note that this implies that $\pi_0\bigl(\md{E^\bu}\bigr)\cong
H^0\bigl(\tau_{\le 0}(E^\bu)\bigr)\cong H^0\bigl(E^\bu\bigr)$. Thus
\eq{da5eq1} yields
\e
\pi_0\bigl(\cA^p_\K(\bX,k)\bigr)\cong H^k\bigl(\La^p\Om^1_A,\d\bigr)
=H^{k-p}\bigl(\La^p\Om^1_A[p],\d\bigr).
\label{da5eq2}
\e

So, (connected components of the simplicial set of) $p$-forms of
degree $k$ on $\bX$ are just $k$-cohomology classes of the complex
$\bigl(\La^p\Om^1_A,\d\bigr)$. We prefer to deal with explicit
representatives, rather than cohomology classes. So we define:

\begin{dfn} 
\label{da5def1}
In the situation above with $\bX=\bSpec A$, a $p$-{\it
form of degree $k$ on\/} $\bX$ for $p\ge 0$ and $k\le 0$ is
$\om^0\in(\La^p\Om^1_A)^k$ with $\d\om^0=0$ in $(\La^p
\Om^1_A)^{k+1}$. Two $p$-forms $\om^0,\ti\om^0$ of degree $k$ are
{\it equivalent}, written $\om^0\sim\ti\om^0$, if there exists
$\al^0\in(\La^p \Om^1_A)^{k-1}$ with $\om^0-\ti\om^0=\d\al^0$.

Then equivalence classes $[\om^0]$ of $p$-forms $\om^0$ of degree
$k$ on $\bX$ in the sense of Definition \ref{da5def1} correspond to
connected components of the simplicial set of $p$-forms of degree
$k$ on $\bX$ in the sense of Pantev et al. \cite{PTVV}. The reason
for including the superscripts $0$ in $\om^0,\al^0$ will become
clear in Definition~\ref{da5def2}.

Letting $\bT_A=\ul{{\rm Hom}}_A(\Om^1_A,A)=\ul{{\rm Der}}(A)$ denote
the {\it tangent complex\/} of $A$, a $2$-form $\om^0$ of degree $k$
on $\bX$ defines an antisymmetric morphism
\e
\om^0:\bT_A\longra \Om^1_A[k],
\label{da5eq3}
\e
via $X\mapsto \io_X\om^0$ for $X \in \ul{{\rm Der}}(A)$. The
$2$-form $\om^0$ is said to be {\it non-degenerate} if this induced
map is a quasi-isomorphism.
\end{dfn}

The definition of the simplicial set $\cA^{p,\cl}_\K(\bX,k)$ of {\it
closed\/ $p$-forms of degree\/} $k\in\Z$ on $\bX=\bSpec A$ in Pantev
et al.\ \cite[Def.~1.7]{PTVV} yields
\e
\cA^{p,\cl}_\K(\bX,k)=\bmd{\ts\prod_{i\ge 0}\La^{p+i\,}\bL_A[k-i]}.
\label{da5eq4}
\e
In our case we may take $\La^{p+i\,}\bL_A[k-i]=
\La^{p+i}\Om^1_A[k-i]$. Then $\prod_{i\ge 0}\La^{p+i\,}\Om^1_A[k-i]$
means the complex which as a graded vector space is the product over
all $i\ge 0$ of the graded vector spaces $\La^{p+i\,}\Om^1_A[k-i]$,
with differential~$\d+\dd$.

The difference between $\prod_{i\ge 0}\La^{p+i}\bL_A[k-i]$ and
$\bigop_{i\ge 0}\La^{p+i}\bL_A[k-i]$ is that elements $(\om^i)_{i\ge
0}\in\bigop_{i\ge 0}\La^{p+i}\bL_A[k-i]$ have $\om^i\ne 0$ for only
finitely many $i$, but elements $(\om^i)_{i\ge 0}\in\prod_{i\ge
0}\La^{p+i}\bL_A[k-i]$ can have $\om^i\ne 0$ for infinitely
many~$i$.

Thus, as for \eq{da5eq1}--\eq{da5eq2}, equation \eq{da5eq4} implies
that
\begin{align*}
\pi_0\bigl(\cA^{p,\cl}_\K(\bX,k)\bigr)&\cong H^0\bigl(
\ts\prod_{i\ge 0}\La^{p+i\,}\Om^1_A[k-i],\d+\dd\bigr)\\
&=H^k\bigl(\ts\prod_{i\ge 0}\La^{p+i}\Om^1_A[-i],\d+\dd\bigr).
\end{align*}

As for Definition \ref{da5def1}, we define:

\begin{dfn} 
\label{da5def2}
In the situation above with $\bX=\bSpec A$, a {\it
closed\/ $p$-form of degree $k$ on\/} $\bX$ for $p\ge 0$ and $k\le
0$ is $\om=(\om^0,\om^1,\om^2,\ldots)$ with $\om^i\in(\La^{p+i}
\Om^1_A)^{k-i}$ for $i=0,1,2,\ldots,$ satisfying the equations
\e
\begin{aligned}
\d\om^0&=0 &&\text{in $(\La^p \Om^1_A)^{k+1}$, and}\\
\dd\om^i+\d\om^{i+1}&=0 &&\text{in $(\La^{p+i+1}\Om^1_A)^{k-i}$, for
all $i\ge 0$.}
\end{aligned}
\label{da5eq5}
\e

We call two closed $p$-forms $\om=(\om^0,\ldots)$,
$\ti\om=(\ti\om^0,\ldots)$ of degree $k$ {\it equivalent}, written
$\om\sim\ti\om$, if there exists $\al=(\al^0,\al^1,\ldots)$ with
$\al^i\in(\La^{p+i}\Om^1_A)^{k-i-1}$ for $i=0,1,\ldots$ satisfying
\begin{align*}
\om^0-\ti\om^0&=\d\al^0 &&\text{in $(\La^p\Om^1_A)^k$, and}\\
\om^{i+1}-\ti\om^{i+1}&=\dd\al^i+\d\al^{i+1} &&\text{in
$(\La^{p+i+1}\Om^1_A)^{k-i-1}$, for all $i\ge 0$.}
\end{align*}

The morphism $\pi:\cA^{p,\cl}_\K(\bX,k)\ra\cA^p_\K(\bX,k)$ from
closed $p$-forms of degree $k$ to $p$-forms of degree $k$ in
\S\ref{da34} corresponds to $\pi:\om=(\om^0,\ldots)\mapsto\om^0$.

A closed 2-form $\om=(\om^0,\om^1,\ldots)$ of degree $k$ on $\bX$ is
called a $k$-{\it shifted symplectic form\/} if $\om^0=\pi(\om)$ is
a nondegenerate 2-form of degree $k$.
\end{dfn}

Then equivalence classes $[\om]$ of closed $p$-forms $\om$ of degree
$k$ on $\bX$ in the sense of Definition \ref{da5def1} correspond to
connected components of the simplicial set $\cA^{p,\cl}_\K(\bX,k)$
of closed $p$-forms of degree $k$ on $\bX$ in the sense of Pantev et
al.~\cite{PTVV}.

\subsection{Closed forms and cyclic homology of mixed complexes}
\label{da52}

In order to work effectively with such symplectic forms, it is very
useful to interpret them in the context of cyclic homology of mixed
complexes. The following definitions are essentially as in Loday
\cite[\S 2.5.13]{Loda}, except that we use cohomological grading and
take into account an extra weight grading in our mixed complexes, as
in Pantev et al.~\cite[\S 1.1]{PTVV}.

\begin{dfn} 
\label{da5def3}
A {\it mixed complex\/} $E$ is a complex over $\K$ with
a differential $b$ of degree $1$ together with an additional
square-zero operator $B$ of degree $-1$ anti-commuting with $b$. A
{\it graded mixed complex} has in addition a {\it weight grading}
giving by a decomposition $E=\bigop_p E(p)$, where $b$ has degree
$0$ and $B$ has degree $1$ with respect to the weight grading. A
morphism between graded mixed complexes $E=\bigop_p E(p)$ and
$F=\bigop_p F(p)$ is a $\K$-linear map $\vp: E \ra F$ of degree zero
with respect to both the cohomological and weight grading that
commutes with $b$ and $B$. The morphism $\vp: E \ra F$ is a {\it
weak equivalence\/} if it is a quasi-isomorphism for the cohomology
taken with respect to the differential $b$. For simplicity and since
it is sufficient for our applications, we shall consider only graded
mixed complexes that are bounded above at $0$ w.r.t.\ the
cohomological grading and bounded below at $0$ w.r.t.\ the weight
grading.
\end{dfn}

\begin{ex} For us, the main example of a graded mixed complex is
$E=\DR(A)=\bigop_p\La^p\Om^1_A[p]$ with $b=\d$ and $B=\dd$, for $A$
a standard form cdga.
\label{da5ex1}
\end{ex}

\begin{dfn} 
\label{da5def4}
Given a graded mixed complex $E$, for each $p$ we
define three complexes, the {\it negative cyclic complex of weight
$p$}, denoted $\NC(E)(p)$, the {\it periodic cyclic complex of
weight $p$}, denoted $\PC(E)(p)$, and the {\it cyclic complex of
weight $p$}, denoted $\CC(E)(p)$. The degree $k$ terms of these
complexes are:
\begin{align*}
\NC^k(E)(p) & = \prod_{i \ge 0} E^{k-2i}(p+i), \\
\PC^k(E)(p) & = \prod_{i} E^{k-2i}(p+i), \\
\CC^k(E)(p) & = \prod_{i \le 0} E^{k-2i}(p+i).
\end{align*}
In each case, the differential is simply $b+B$, and the complexes
can constructed as ${\rm Tot}^{\Pi}$ of appropriate double
complexes. The $k^{\rm th}$ cohomology of the complexes $\NC(E)(p)$,
$\PC(E)(p)$, and $\CC(E)(p)$ are denoted $\HN^k(E)(p)$,
$\HP^k(E)(p)$, and $\HC^k(E)(p)$, respectively.

There is an evident short exact sequence of complexes
\begin{equation*}
0 \longra \NC(E)(p) \longra \PC(E)(p)
\longra \CC(E)(p-1)[2] \longra 0,
\end{equation*}
and an induced long exact sequence of cohomology groups
\e
\begin{gathered}
\xymatrix@C=25pt@R=1pt{ \cdots \ar[r] & \HC^{k+1}(E)(p-1) \ar[r]
& \HN^k(E)(p) \ar[r] & \HP^k(E)(p) \\
{\phantom{\cdots}} \ar[r] & \HC^{k+2}(E)(p-1) \ar[r] &
{}\quad\cdots.\qquad{} }
\end{gathered}
\label{da5eq6}
\e

When $E=\DR(A)$ for $A$ a standard form cdga, we denote the
corresponding cochain groups by $\NC^k(A)(p),\PC^k(A)(p),
\CC^k(A)(p)$, and the cohomology groups by $\HN^k(A)(p),\HP^k(A)(p),
\HC^k(A)(p)$. As in Loday \cite[Ch.~5]{Loda}, these groups are known
to be compatible with other definitions that the reader may have
seen. The connection of all this with the material of \S\ref{da51}
is that closed $p$-forms $\om=(\om^0,\om^1,\om^2,\ldots)$ of degree
$k$ in Definition \ref{da5def2} are cocycles in $\NC^{k-p}(A)(p)$,
and equivalence classes $[\om]$ of closed $p$-forms
$\om=(\om^0,\om^1,\ldots)$ of degree $k$ on $\bX=\bSpec A$ are
elements of $\HN^{k-p}(A)(p)$, by the Hochschild--Kostant--Rosenberg Theorem.
\end{dfn}

Here is a useful vanishing result. (Note that the group
$\HN^{k-2}(A)(2)$ in \eq{da5eq7} classifies closed $2$-forms $\om$
of degree $k$ on $\bX=\bSpec A$ up to equivalence.)

\begin{prop}
\label{da5prop1}
{\bf(a)} Let\/ $A$ be a connective cdga over $\K,$ with $H^{0}(A)$ 
of finite type over $\K$. If\/ $p+k \le 2$, then in the
sequence \eq{da5eq6} the map $\HP^{k-3}(A)(p)\ab \ra \HC^{k-1}(A)(p-1)$
is an injection, hence the map $\HN^{k-3}(A)(p) \ra \HP^{k-3}(A)(p)$ is zero
and the map $\HC^{k-2}(A)(p-1) \ra \HN^{k-3}(A)(p)$ is a surjection.

In particular, for $k<0$ we have a short exact sequence
\e
0 \longra \HP^{k-3}(A)(2) \longra \HC^{k-1}(A)(1) \longra
\HN^{k-2}(A)(2) \longra 0.
\label{da5eq7}
\e

\noindent{\bf(b)} For $k=-1,$ the left hand group $\HP^{-4}(A)(2)$
in \eq{da5eq7} has $\HP^{-4}(A)(2)\cong\mathbin{\rm H}^0_\inf(\Spec
H^0(A)),$ where $\mathbin{\rm H}^*_\inf(X)$ is the algebraic de Rham
cohomology of a\/ $\K$-scheme $X$. Thus, if\/ $\Spec H^0(A)$ is
connected then\/~$\HP^{-4}(A)(2)=\K$.

\smallskip

\noindent{\bf(c)} For $k \le -2,$ we have~$\HP^{k-3}(A)(2)=0$.
\end{prop}

\begin{proof} When $A$ is a cdga concentrated in degree $0$, the
fact that the map $\HP^{k-3}(A)(p) \ra \HC^{k-1}(A)(p-1)$ is an
injection is essentially Emmanouil \cite[Prop.~2.6]{Emma}, noting
that $\HP^{k-3}(A)(p) \cong \mathbin{\rm H}^{2p+k-3}_\inf(A)$ by
\cite[Th.~2.2]{Emma}. We therefore have a short exact sequence as
claimed. Here $\mathbin{\rm H}^{2p+k-3}_\inf(A)$ is `infinitesimal' or
algebraic de Rham cohomology. In particular, for $k=-1$,
$\HP^{-4}(A)(2)\cong \mathbin{\rm H}^0_\inf(A) \cong \K$ when $\Spec
A$ is connected, and for $k\le -2$, $\HP^{k-3}(A)(2)\cong
\mathbin{\rm H}^{k+1}_\inf(A)=0$.

Now let $A$ be a connective cdga with $H^{0}(A)$ of finite type over $\K$ and consider the natural
map $A\ra H^0(A)$. By Goodwillie \cite[Th.~IV.2.1]{Good}, the
induced map $\HP^{k-3}(A)(p) \ra \HP^{k-3}(H^0(A))(p)$ is an isomorphism, so
that (b),(c) follow from the above. The injectivity of the map
$\HP^{k-3}(A)(p) \ra \HC^{k-1}(A)(p-1)$ follows from the injectivity of
the map $\HP^{k-3}(H^0(A))(p) \ra \HC^{k-1}(H^0(A))(p-1)$ by
functoriality of the long exact sequence \eq{da5eq6}, and this
proves~(a).
\end{proof}

Using Proposition \ref{da5prop1} we show that if
$\om=(\om^0,\om^1,\ldots)$ is a $k$-shifted symplectic structure on
$\bX=\bSpec A$ for $k<0$, then up to equivalence we can take $\om^0$
to be exact and $\om^i=0$ for $i=1,2,\ldots,$ which is a
considerable simplification. Also we parametrize how to write $\om$
this way in its equivalence class.

\begin{prop}
\label{da5prop2}
{\bf(a)} Let\/ $\om=(\om^0,\om^1, \om^2,\ldots)$ be a
closed\/ $2$-form of degree $k<0$ on $\bX=\bSpec A,$ for $A$ a
standard form cdga over $\K$. Then there exist\/ $\Phi\in A^{k+1}$
and\/ $\phi\in(\Om^1_A)^k$ such that\/ $\d\Phi=0$ in\/ $A^{k+2}$
and\/ $\dd \Phi+\d\phi=0$ in $(\Om^1_A)^{k+1}$ and\/ $\om\sim (\dd
\phi,0,0,\ldots),$ in the sense of Definition\/ {\rm\ref{da5def2}}.
\smallskip

\noindent{\bf(b)} In the case $k=-1$ in {\bf(a)} we have $\Phi\in
A^0=A(0),$ so we can consider the restriction\/ $\Phi\vert_{X^\red}$
of\/ $\Phi$ to the reduced\/ $\K$-subscheme $X^\red$ of\/
$X=t_0(\bX)=\Spec H^0(A)$. Then $\Phi\vert_{X^\red}$ is locally
constant on $X^\red,$ and we may choose $(\Phi,\phi)$ in {\bf(a)}
such that\/~$\Phi\vert_{X^\red}=0$.
\smallskip

\noindent{\bf(c)} Suppose $(\Phi,\phi)$ and\/ $(\Phi',\phi')$ are
alternative choices in part\/ {\bf(a)\rm} for fixed\/ $\om,k,\bX,A,$
where if\/ $k=-1$ we suppose $\Phi\vert_{X^\red}=0=
\Phi'\vert_{X^\red}$ as in {\bf(b)}. Then there exist\/ $\Psi\in
A^k$ and\/ $\psi\in(\Om^1_A)^{k-1}$ with\/ $\Phi-\Phi'=\d\Psi$
and\/~$\phi-\phi'=\dd\Psi+\d\psi$.

\end{prop}

\begin{proof} As in Definition \ref{da5def4}, the
$\sim$-equivalence class $[\om]$ of $\om$ lies in $\HN^{k-2}(A)(2)$.
Thus by equation \eq{da5eq7} in Proposition \ref{da5prop1}, $[\om]$
lies in the image of the map $\HC^{k-1}(A)(1) \ra \HN^{k-2}(A)(2)$.
A class in $\HC^{k-1}(A)(1)$ is represented by $(\Phi,\phi)\!\in\!
\CC^{k-1}(A)(1)\!=\!A^{k+1}\!\t\!(\Om^1_A)^k$ with $\d
\Phi\!=\!0\!=\!\dd \Phi\!+\!\d\phi$, and the map
$\CC^{k-1}(A)(1)\ra\NC^{k-2}(A)(2)$ takes
$(\Phi,\phi)\mapsto(\dd\phi,0,0,\ldots)$. Note that
\begin{equation*}
\d(\dd\phi)=-\dd\ci\d\phi=-\dd\ci\dd \Phi-\dd\ci\d\phi=-\dd\bigl(\dd
\Phi+\d\phi\bigr)=0,
\end{equation*}
so that $(\dd\phi,0,0,\ldots)$ is a closed 2-form of degree $k$ on
$\bX$. This proves~(a).

For (b), $\Phi\vert_X:X\ra\bA^1$ is a regular function on $X$ when
$k=-1$ . We have $H^0(T^*X)\cong H^0(\Om^1_A)$. As
$\dd\Phi+\d\phi=0$ we have $[\dd\Phi]=0$ in $H^0(\Om^1_A)$, so
$\dd\bigl(\Phi\vert_X\bigr)=0$ in $H^0(T^*X)$. Therefore
$\Phi\vert_{X^\red}:X^\red\ra\bA^1$ is locally constant.

This is related to the isomorphism $\HP^{-4}(A)(2)\cong\mathbin{\rm
H}^0_\inf(\Spec H^0(A))$ in Proposition \ref{da5prop1}(b). There is
a natural isomorphism from $\mathbin{\rm H}^0_\inf(\Spec H^0(A))$ to
the $\K$-vector space of locally constant maps
$\Phi\vert_{X^\red}:X^\red\ra\bA^1$. When we lift
$[\om]\in\HN^{-3}(A)(2)$ to $[\Phi,\phi]\in\HC^{-2}(A)(1)$ in
\eq{da5eq7} for $k=-1$, the possible choice in the lift
$[\Phi,\phi]$ is $\HP^{-4}(A)(2)\cong\mathbin{\rm H}^0_\inf(\Spec
H^0(A))$, which corresponds exactly to the space of locally constant
maps~$\Phi\vert_{X^\red}:X^\red\ra\bA^1$.

So, by adjusting the choice of lift $[\Phi,\phi]$ of $[\om]$ to
$\HC^{-2}(A)(1)$, we can take $\Phi\vert_{X^\red}=0$, proving (b).
Note that this determines the class $[\Phi,\phi]$ in
$\HC^{-2}(A)(1)$ lifting $[\om]$ uniquely. That is, we have
constructed a canonical splitting of the exact sequence \eq{da5eq7}
when~$k=-1$.

For (c), note that for all $k<0$ the class
$[\Phi,\phi]\in\HC^{k-1}(A)(1)$ in (a) lifting
$[\om]\in\HN^{k-2}(A)(2)$ is uniquely determined, requiring
$\Phi\vert_{X^\red}=0$ when $k=-1$ as above, and since
$\HP^{k-3}(A)(2)=0$ in \eq{da5eq7} for $k\le -2$ by Proposition
\ref{da5prop1}(c). Hence, if $(\Phi,\phi)$ and $(\Phi',\phi')$ are
alternative choices in (a), then $[\Phi,\phi]=[\Phi',\phi']\in
\HC^{k-1}(A)(1)$. Thus, there exists $(\Psi,\psi)\in\CC^{k-2}(A)(1)$
with $\d(\Psi,\psi)=(\Phi,\phi)-(\Phi',\phi')$. From the
definitions, this means that $\Psi\in A^k$ and
$\psi\in(\Om^1_A)^{k-1}$ with $\Phi-\Phi'=\d\Psi$
and~$\phi-\phi'=\dd\Psi+\d\psi$.
\end{proof}

When we apply Proposition \ref{da5prop2} in \S\ref{da54}--\S\ref{da57}, we will do it with $k\om$ in place of $\om$, yielding $\Phi,\phi$ with 
$\d\Phi=0$, $\dd \Phi+\d\phi=0$, and $k\om\sim(\dd\phi,0,0,\ldots)$. This will give simpler formulae, eliminating factors of $k,1/k$.

\subsection{\texorpdfstring{`Darboux forms' for $k$-shifted symplectic structures}{\textquoteleft Darboux forms\textquoteright\ for k-shifted symplectic structures}}
\label{da53}

The next four examples give standard models for $k$-shifted
symplectic affine derived $\K$-schemes for $k<0$, which we will call
`in Darboux form'. Theorem \ref{da5thm1} will prove that every
$k$-shifted symplectic derived scheme $(\bX,\om)$ is Zariski/\'etale
locally equivalent to one in Darboux form. We divide into three
cases:
\begin{itemize}
\setlength{\itemsep}{0pt}
\setlength{\parsep}{0pt}
\item[(a)] $k$ is odd, so that $k=-2d-1$ for $d=0,1,2,\ldots;$
\item[(b)] $k\equiv 0\mod 4$, so that $k=-4d$ for $d=1,2,\ldots;$
and
\item[(c)] $k\equiv 2\mod 4$, so that $k=-4d-2$ for $d=0,1,\ldots.$
\end{itemize}

The difference is in the behaviour of 2-forms in the variables of
`middle degree' $k/2$. In case (a) $k/2\notin\Z$, so there is no
middle degree, and this is the simplest case, which we handle in
Example \ref{da5ex2}. In (b) $k/2$ is even, so 2-forms in the middle
degree variables are antisymmetric. We discuss this in Example
\ref{da5ex3}. In (c) $k/2$ is odd, so 2-forms in the middle degree
variables are symmetric. For this case we give both a `strong
Darboux form' in Example \ref{da5ex4}, to which $k$-shifted
symplectic derived $\K$-schemes are equivalent \'etale locally, and
a `weak Darboux form' in Example \ref{da5ex5}, to which they are
equivalent Zariski locally.

For $A$ as in Example \ref{da5ex2} below, as in Bouaziz and Grojnowski \cite{BoGr} we can regard $\bX=\bSpec A$ as a {\it twisted shifted cotangent bundle\/} $T^*_\si[k]\bY$ for $\bY=\bSpec B$, where $B\subset A$ is the sub-cdga generated by the variables $x_j^{-i}$. Then in \eq{da5eq8} the $x_j^{-i}$ are coordinates on the base $\bY$, and $y_j^{k-i}$ are the dual coordinates on the fibres of~$T^*_\si[k]\bY\ra\bY$.

\begin{ex} 
\label{da5ex2}
Fix $d=0,1,\ldots.$ We will explain how to define a class
of explicit standard form cdgas $(A,\d)=A(n)$ for $n=2d+1$ with a
very simple, explicit $k$-shifted symplectic form
$\om=(\om^0,0,0,\ldots)$ for $k=-2d-1$.

First choose a smooth $\K$-algebra $A(0)$ of dimension $m_0$.
Localizing $A(0)$ if necessary, we may assume that there exist
$x^0_1,\ldots,x^0_{m_0}\in A(0)$ such that $\dd x^0_1,\ldots,\dd
x^0_{m_0}$ form a basis of $\Om^1_{A(0)}$ over $A(0)$.
Geometrically, $\Spec A(0)$ is a smooth $\K$-scheme of dimension
$m_0$, and $(x^0_1,\ldots,x^0_{m_0}):\Spec A(0)\ra\bA^{m_0}$ are
global \'etale coordinates on~$\Spec A(0)$.

Next, choose $m_1,\ldots,m_d\in\N=\{0,1,\ldots\}$. Define $A$ as a
commutative graded algebra to be the free algebra over $A(0)$
generated by variables
\e
\begin{aligned}
&x_1^{-i},\ldots,x^{-i}_{m_i} &&\text{in degree $-i$ for
$i=1,\ldots,d$, and} \\
&y_1^{i-2d-1},\ldots,y^{i-2d-1}_{m_i} &&\text{in degree $i-2d-1$ for
$i=0,1,\ldots,d$.}
\end{aligned}
\label{da5eq8}
\e
So the upper index $i$ in $x^i_j,y^i_j$ always indicates the degree.
We will define the differential $\d$ in the cdga $(A,\d)$ later.

As in \S\ref{da21} and \S\ref{da23}, the spaces $(\La^p\Om^1_A)^k$
and the de Rham differential $\dd$ upon them depend only on the
commutative graded algebra $A$, not on the (not yet defined)
differential $\d$. Note that $\Om^1_A$ is the free $A$-module with
basis $\dd x^{-i}_j,\dd y^{i-2d-1}_j$ for $i=0,\ldots,d$ and
$j=1,\ldots,m_i$. Define an element
\e
\om^0=\sum_{i=0}^d\sum_{j=1}^{m_i}\dd x^{-i}_j\,\dd y^{i-2d-1}_j
\qquad \text{in $(\La^2\Om^1_A)^{-2d-1}$.}
\label{da5eq9}
\e
Clearly $\dd\om^0=0$ in $(\La^3\Om^1_A)^{-2d-1}$.

Now choose a superpotential $\Phi$ in $A^{-2d}$, which we will call the
{\it Hamiltonian}, and which we require to satisfy the {\it
classical master equation\/}
\e
\sum_{i=1}^d\sum_{j=1}^{m_i}\frac{\pd \Phi}{\pd x^{-i}_j}\,
\frac{\pd \Phi}{\pd y^{i-2d-1}_j}=0\qquad\text{in $A^{1-2d}$.}
\label{da5eq10}
\e
Note that \eq{da5eq10} is trivial when $d=0$, so that
$k=-1$, as $A^1=0$. We have
\e
\dd \Phi=\sum_{i=0}^d\sum_{j=1}^{m_i}\frac{\pd \Phi}{\pd x^{-i}_j}\,\dd
x^{-i}_j +\sum_{i=1}^d\sum_{j=1}^{m_i}\frac{\pd \Phi}{\pd
y^{i-2d-1}_j}\,\dd y^{i-2d-1}_j,
\label{da5eq11}
\e
with $\frac{\pd \Phi}{\pd x^{-i}_j}\in A^{i-2d}$ and $\frac{\pd \Phi}{\pd
y^{i-2d-1}_j}\in A^{1-i}$. Note that when $i=0$, $\frac{\pd \Phi}{\pd
y^{-2d-1}_j}=0$ for degree reasons. Define the differential $\d$ on
$A$ by $\d=0$ on $A(0)$, and
\e
\d x^{-i}_j =\frac{\pd \Phi}{\pd y^{i-2d-1}_j}, \quad \d
y^{i-2d-1}_j=\frac{\pd \Phi}{\pd x^{-i}_j},\quad\begin{subarray}{l}\ts
i=0,\ldots,d,\\[6pt] \ts j=1,\ldots,m_i.\end{subarray}
\label{da5eq12}
\e
When $i=0$ this gives $\d x^0_j=\frac{\pd \Phi}{\pd y^{-2d-1}_j}=0$,
consistent with $\d=0$ on $A(0)$. To show that $\d\ci\d=0$, note
that
\e
\begin{split}
\d&\ci\d x^{-i'}_{j'}=\d\biggl[\frac{\pd \Phi}{\pd y^{i'-2d-1}_{j'}}\biggr]\\
&=\sum_{i=0}^d\sum_{j=1}^{m_i}\biggl[\d y^{i-2d-1}_j\cdot
\frac{\pd^2 \Phi}{\pd y^{i-2d-1}_j\pd y^{i'-2d-1}_{j'}}+\d
x^{-i}_j\cdot \frac{\pd^2 \Phi}{\pd x^{-i}_j\pd y^{i'-2d-1}_{j'}}\biggr]\\
&=\sum_{i=0}^d\sum_{j=1}^{m_i}\biggl[\frac{\pd \Phi}{\pd x^{-i}_j}\cdot
\frac{\pd^2 \Phi}{\pd y^{i-2d-1}_j\pd y^{i'-2d-1}_{j'}}+ \frac{\pd
\Phi}{\pd y^{i-2d-1}_j}\cdot \frac{\pd^2 \Phi}{\pd x^{-i}_j\pd
y^{i'-2d-1}_{j'}}\biggr]\\
&=(-1)^{i'-2d-1}\,\frac{\pd}{\pd y^{i'-2d-1}_{j'}}\biggl[
\sum_{i=0}^d\sum_{j=1}^{m_i}\frac{\pd \Phi}{\pd x^{-i}_j}\cdot
\frac{\pd \Phi}{\pd y^{i-2d-1}_j}\biggr]=0,
\end{split}
\label{da5eq13}
\e
using \eq{da5eq12} in the first and third steps and \eq{da5eq10} in
the last. In the same way we show that~$\d\ci\d y^{i'-2d-1}_{j'}=0$.

Observe that $(A,\d)$ is a standard form cdga $A=A(n)$ as in Example
\ref{da2ex2} for $n=2d+1$, defined inductively using free modules
$M^{-i}=\langle x^{-i}_1,\ldots,x^{-i}_{m_i} \rangle_{A(i-1)}$ for
$i=1,\ldots,d$ and $M^{i-2d-1}=\langle y^{i-2d-1}_1,
\ldots,y^{i-2d-1}_{m_i}\rangle_{A(i-2d-2)}$ for $i=0,\ldots,d$.

We claim that $\om:=(\om^0,0,0,\ldots)$ is a $k$-shifted symplectic
structure on $\bX=\bSpec A$ for $k=-2d-1$. To show that $\om$ is a
$k$-shifted closed 2-form, since $\dd\om^0=0$, by \eq{da5eq5} we
have to show that $\d\om^0=0$. This follows from
\ea
\d\om^0&=\sum_{i=0}^d\sum_{j=1}^{m_i}\Bigl[\bigl(\d\ci\dd
y^{i-2d-1}_j \bigr)\,\dd x^{-i}_j+ \bigl(\d\ci\dd
x^{-i}_j\bigr)\,\dd y^{i-2d-1}_j\Bigr]
\nonumber\\
&=-\sum_{i=0}^d\sum_{j=1}^{m_i}\Bigl[\bigl(\dd\ci\d y^{i-2d-1}_j
\bigr)\,\dd x^{-i}_j+\bigl(\dd\ci\d x^{-i}_j\bigr)\,\dd
y^{i-2d-1}_j\Bigr]
\nonumber\\
&=-\sum_{i=0}^d\sum_{j=1}^{m_i}\biggl[\dd\biggl[ \frac{\pd \Phi}{\pd
x^{-i}_j}\biggr]\,\dd x^{-i}_j+\dd\biggl[ \frac{\pd \Phi}{\pd
y^{i-2d-1}_j}\biggr]\,\dd y^{i-2d-1}_j\biggr]
\nonumber\\
&=-\dd\biggl[\sum_{i=0}^d\sum_{j=1}^{m_i} \frac{\pd \Phi}{\pd
x^{-i}_j}\,\dd x^{-i}_j+ \frac{\pd \Phi}{\pd y^{i-2d-1}_j}\,\dd
y^{i-2d-1}_j\biggr]
\nonumber\\
&=-\dd\ci\dd \Phi=0,
\label{da5eq14}
\ea
using \eq{da5eq9} in the first step, $\d\ci\dd=-\dd\ci\d$ in the
second, \eq{da5eq12} in the third, $\dd\ci\dd=0$ in the fourth and
sixth, and \eq{da5eq11} in the fifth.

To show $\om^0$ is nondegenerate, we must check \eq{da5eq3} is a
quasi-isomorphism. It is enough to show that $\om^0\ot\id_{A(0)}:
\bT_A\ot_AA(0)\ra\Om^1_A[k]\ot_AA(0)$ is one. Using Proposition
\ref{da2prop}, we see that $\om^0\ot\id_{A(0)}$ may be written
\e
\begin{gathered}
\xymatrix@R=11pt@C=130pt{ *+[r]{\ts\bigl\langle\frac{\pd}{\pd
x^0_1},\ldots, \frac{\pd}{\pd x^0_{m_0}}\bigl\rangle_{A(0)}}
\ar@<3ex>[d]^{\d} \ar[r]_(0.4){\om^0} &
*+[l]{\ts\bigl\langle\dd y^{-2d-1}_1,\ldots,
\dd y^{-2d-1}_{m_0}\bigr\rangle_{A(0)}} \ar@<-6ex>[d]_{\d} \\
{{}\hskip 6ex\vdots} \ar@<3ex>[d]^\d & {\vdots\hskip 12ex{}}
\ar@<-6ex>[d]_\d \\
*+[r]{\ts\bigl\langle\frac{\pd}{\pd x^{-d}_1},\ldots,
\frac{\pd}{\pd x^{-d}_{m_d}}\bigl\rangle_{A(0)}} \ar@<3ex>[d]^{\d}
\ar[r]_(0.44){\om^0} &
*+[l]{\ts\bigl\langle\dd y^{-d-1}_1,\ldots,
\dd y^{-d-1}_{m_d}\bigr\rangle_{A(0)}} \ar@<-6ex>[d]_{\d} \\
*+[r]{\ts\bigl\langle\frac{\pd}{\pd y^{-d-1}_1},\ldots,
\frac{\pd}{\pd y^{-d-1}_{m_d}}\bigl\rangle_{A(0)}} \ar@<3ex>[d]^{\d}
\ar[r]^(0.46){\om^0} &
*+[l]{\ts\bigl\langle\dd x^{-d}_1,\ldots,
\dd x^{-d}_{m_d}\bigr\rangle_{A(0)}} \ar@<-6ex>[d]_{\d} \\
{{}\hskip 6ex\vdots} \ar@<3ex>[d]^\d & {\vdots\hskip 12ex{}}
\ar@<-6ex>[d]_\d \\
*+[r]{\ts\bigl\langle\frac{\pd}{\pd y^{-2d-1}_1},\ldots,
\frac{\pd}{\pd y^{-2d-1}_{m_0}}\bigl\rangle_{A(0)}}
\ar[r]^(0.5){\om^0} &
*+[l]{\ts\bigl\langle\dd x^0_1,\ldots,
\dd x^0_{m_d}\bigr\rangle_{A(0)}.\!\!{}}
}\!\!\!\!\!\!\!\!\!\!\!\!\!\!\!\!{}
\end{gathered}
\label{da5eq15}
\e
But by \eq{da5eq9}, the rows of \eq{da5eq15} are isomorphisms, so
$\om^0\ot\id_{A(0)}$ is an isomorphism of complexes, and thus a
quasi-isomorphism. Hence $\om=(\om^0,0,0,\ldots)$ is a $k$-shifted
symplectic structure on $\bX=\bSpec A$ for $k=-2d-1$.

Finally, define $\phi\in(\Om^1_A)^{-2d-1}$ by
\e
\phi=-\sum_{i=0}^d\sum_{j=1}^{m_i}\bigl[i\,x^{-i}_j\,\dd y^{i-2d-1}_j+(2d+1-i)y^{i-2d-1}_j\,\dd x^{-i}_j\bigr].
\label{da5eq16}
\e
Then calculating using \eq{da5eq9}--\eq{da5eq12} shows that $\d
\Phi=0$ in $A^{1-2d}$, and $\dd\Phi+\d\phi=0$ in $(\Om^1_A)^{-2d}$,
and $k\om^0=\dd\phi$ in $(\La^2\Om^1_A)^{-2d-1}$. Here to prove
$\dd\Phi+\d\phi=0$ we use that $\frac{\pd \Phi}{\pd x^{-i}_j}$ lies in
$A^{i-2d}$ so that
\begin{equation*}
(i\!-\!2d)\frac{\pd \Phi}{\pd x^{-i}_j}=\!\sum_{i'=0}^d\sum_{j'=1}^{m_{i'}}
\biggl[-i'x^{-i'}_{j'}
\frac{\pd^2 \Phi}{\pd x^{-i'}_{j'}\pd x^{-i}_j}+
(i'\!-\!2d\!-\!1)y^{i'-2d-1}_{j'}
\frac{\pd^2 \Phi}{\pd y^{i'-2d-1}_{j'}\pd x^{-i}_j}\biggr],
\end{equation*}
and a similar equation for $\frac{\pd \Phi}{\pd y^{i-2d-1}_j}$. Thus
$(\Phi,\phi)$ satisfy Proposition \ref{da5prop2}(a) with $k\om$ in place of~$\om$.
\end{ex}

\begin{ex} 
\label{da5ex3}
Fix $d=1,2,\ldots.$ We will define a class of explicit
standard form cdgas $(A,\d)=A(n)$ for $n=4d$ with an explicit
$k$-shifted symplectic form $\om=(\om^0,0,0,\ldots)$ for~$k=-4d$.

As in Example \ref{da5ex2}, choose a smooth $\K$-algebra $A(0)$ of
dimension $m_0$, and $x^0_1,\ldots,x^0_{m_0}\in A(0)$ with $\dd
x^0_1,\ldots,\dd x^0_{m_0}$ a basis of $\Om^1_{A(0)}$ over $A(0)$.
Choose $m_1,\ldots,m_{2d}\in\N$, and define $A$ as a commutative
graded algebra to be the free algebra over $A(0)$ generated by
variables, as for \eq{da5eq8}
\e
\begin{aligned}
&x_1^{-i},\ldots,x^{-i}_{m_i} &&\text{in degree $-i$ for
$i=1,\ldots,2d-1$,} \\
&x_1^{-2d},\ldots,x^{-2d}_{m_{2d}},y_1^{-2d},\ldots,y^{-2d}_{m_{2d}}
&&\text{in degree $-2d$, and} \\
&y_1^{i-4d},\ldots,y^{i-4d}_{m_i} &&\text{in degree $i-4d$ for
$i=0,1,\ldots,2d-1$.}
\end{aligned}\!\!\!\!\!\!\!\!\!\!\!\!\!\!\!\!{}
\label{da5eq17}
\e
Note that there are $2m_{2d}$ rather than $m_{2d}$ generating
variables in degree $-2d$.

As for \eq{da5eq9}, define $\om^0\in(\La^2\Om^1_A)^{-4d}$ with
$\dd\om^0=0$ in $(\La^3\Om^1_A)^{-4d}$ by
\e
\om^0=\sum_{i=0}^{2d}\sum_{j=1}^{m_i}\dd x^{-i}_j\,\dd y^{i-4d}_j
\qquad \text{in $(\La^2\Om^1_A)^{-4d}$.}
\label{da5eq18}
\e
Choose a Hamiltonian $\Phi$ in $A^{1-4d}$, which we require to satisfy
the analogue of \eq{da5eq10}. The analogue of \eq{da5eq11} holds. As
for \eq{da5eq12}, define the differential $\d$ on $A$ by $\d=0$ on
$A(0)$, and
\e
\d x^{-i}_j =(-1)^{i+1}\frac{\pd \Phi}{\pd y^{i-4d}_j}, \quad \d
y^{i-4d}_j=\frac{\pd \Phi}{\pd x^{-i}_j},\quad\begin{subarray}{l}\ts
i=0,\ldots,2d,\\[6pt] \ts j=1,\ldots,m_i.\end{subarray}
\label{da5eq19}
\e
We prove that $\d\ci\d=0$ as in \eq{da5eq13}. The analogue of
\eq{da5eq14} holds with signs inserted, and the analogue of
\eq{da5eq15} has rows isomorphisms. Hence $\om:=(\om^0,0,0,\ldots)$
is a $k$-shifted symplectic structure on $\bX=\bSpec A$ for~$k=-4d$.

Finally, defining $\phi\in(\Om^1_A)^{-4d}$ by the analogue of \eq{da5eq16}, we find that $(\Phi,\phi)$ satisfy Proposition \ref{da5prop2}(a) with $k\om$ in place of~$\om$.
\end{ex}

\begin{ex} 
\label{da5ex4}
Fix $d=0,1,2,\ldots.$ We will define a class of explicit
standard form cdgas $(A,\d)=A(n)$ for $n=4d+2$ with an explicit
$k$-shifted symplectic form $\om=(\om^0,0,0,\ldots)$ for~$k=-4d-2$.

As in Examples \ref{da5ex2} and \ref{da5ex3}, choose a smooth
$\K$-algebra $A(0)$ of dimension $m_0$, and
$x^0_1,\ldots,x^0_{m_0}\in A(0)$ with $\dd x^0_1,\ldots,\dd
x^0_{m_0}$ a basis of $\Om^1_{A(0)}$ over $A(0)$. Choose
$m_1,\ldots,m_{2d+1}\in\N$, and define $A$ as a commutative graded
algebra to be the free algebra over $A(0)$ generated by variables,
as for \eq{da5eq8} and \eq{da5eq17}
\e
\begin{aligned}
&x_1^{-i},\ldots,x^{-i}_{m_i} &&\text{in degree $-i$ for
$i=1,\ldots,2d$,} \\
&z_1^{-2d-1},\ldots,z^{-2d-1}_{m_{2d+1}}
&&\text{in degree $-2d-1$, and} \\
&y_1^{i-4d-2},\ldots,y^{i-4d-2}_{m_i} &&\text{in degree $i-4d-2$ for
$i=0,1,\ldots,2d$.}
\end{aligned}\!\!\!\!\!\!\!\!\!\!\!\!\!\!\!\!{}
\label{da5eq20}
\e

As for \eq{da5eq9} and \eq{da5eq18}, define
$\om^0\in(\La^2\Om^1_A)^{-4d-2}$ with $\dd\om^0=0$ by
\e
\om^0=\sum_{i=0}^{2d}\sum_{j=1}^{m_i}\dd x^{-i}_j\,\dd y^{i-4d-2}_j
+\sum_{j=1}^{m_{2d+1}}\dd z^{-2d-1}_j\,\dd z^{-2d-1}_j.
\label{da5eq21}
\e
Choose a Hamiltonian $\Phi$ in $A^{-4d-1}$, which we require to satisfy
the {\it classical master equation\/}
\e
\sum_{i=1}^{2d}\sum_{j=1}^{m_i}\frac{\pd \Phi}{\pd x^{-i}_j}\,
\frac{\pd \Phi}{\pd y^{i-4d-2}_j}+\frac{1}{4}\sum_{j=1}^{m_{2d+1}}
\biggl(\frac{\pd \Phi}{\pd z^{-2d-1}_j}\biggr)^2=0\qquad\text{in
$A^{-4d}$.}
\label{da5eq22}
\e
The analogue of \eq{da5eq11} holds, with extra terms $\frac{\pd
\Phi}{\pd z^{-2d-1}_j}\,\dd z^{-2d-1}_j$. As for \eq{da5eq12} and
\eq{da5eq19}, define the differential $\d$ on $A$ by $\d=0$ on
$A(0)$, and
\e
\begin{gathered}
\d x^{-i}_j =(-1)^{i+1}\frac{\pd \Phi}{\pd y^{i-4d-2}_j}, \quad \d
y^{i-4d-2}_j=\frac{\pd \Phi}{\pd x^{-i}_j},\quad\begin{subarray}{l}\ts
i=0,\ldots,2d,\\[6pt] \ts j=1,\ldots,m_i,\end{subarray} \\
\text{and}\qquad\d z^{-2d-1}_j=\frac{1}{2}\,\frac{\pd \Phi}{\pd
z^{-2d-1}_j}, \qquad j=1,\ldots,m_{2d+1}.
\end{gathered}
\label{da5eq23}
\e

The analogue of \eq{da5eq13} shows that $\d\ci\d=0$. The analogue of
\eq{da5eq14} holds with extra signs and terms from the
$z^{-2d-1}_j$, and the analogue of \eq{da5eq15} has rows
isomorphisms. Thus $\om:=(\om^0,0,0,\ldots)$ is a $k$-shifted
symplectic structure on $\bX=\bSpec A$ for $k=-4d-2$. Defining
\ea
\phi&=-\sum_{i=0}^{2d}\sum_{j=1}^{m_i}\bigl[i\,x^{-i}_j\,\dd y^{i-4d-2}_j+(-1)^{i+1}(4d+2-i)y^{i-4d-2}_j\,\dd x^{-i}_j\bigr]
\nonumber\\
&\qquad-(4d+2)\sum_{j=1}^{m_{2d+1}}z^{-2d-1}_j\,\dd z^{-2d-1}_j
\qquad\text{in $(\Om^1_A)^{-4d-2}$,}
\label{da5eq24}
\ea
the analogue of \eq{da5eq16}, we find that $(\Phi,\phi)$ satisfy
Proposition \ref{da5prop2}(a) with $k\om$ in place of~$\om$.
\end{ex}

\begin{rem} 
\label{da5rem1}
In Example \ref{da5ex4}, if $m_{2d+1}$ is even, or
equivalently if the {\it virtual dimension\/} $\mathop{\rm vdim}\bX$
of $\bX$ is even, given by
\begin{equation*}
\mathop{\rm vdim}\bX=-m_{2d+1}+2\ts\sum_{i=0}^{2d}(-1)^im_i,
\end{equation*}
then we may change variables from $z_1^{-2d-1},\ldots,
z^{-2d-1}_{m_{2d+1}}$ to $x_1^{-2d-1},\ldots,\ab
x^{-2d-1}_{m_{2d+1}/2},\ab
y_1^{-2d-1},\ab\ldots,y^{-2d-1}_{m_{2d+1}/2}$ defined for
$j=1,\ldots,m_{2d+1}/2$ by
\begin{equation*}
x_j^{-2d-1}=z_j^{-2d-1}+\sqrt{-1}\,z_{j+m_{2d+1}/2}^{-2d-1},\quad
y_j^{-2d-1}=z_j^{-2d-1}-\sqrt{-1}\,z_{j+m_{2d+1}/2}^{-2d-1}.
\end{equation*}
Then replacing $m_{2d+1}$ by $m_{2d+1}/2$, we find that $A$ is
freely generated over $A(0)$ by variables
\begin{align*}
&x_1^{-i},\ldots,x^{-i}_{m_i} &&\text{in degree $-i$ for
$i=1,\ldots,2d$,} \\
&x_1^{-2d-1},\ldots,\ab
x^{-2d-1}_{m_{2d+1}},\ab
y_1^{-2d-1},\ab\ldots,y^{-2d-1}_{m_{2d+1}}
\!\!\!\!\!\!\!\!\!\!{}
&&\qquad\text{in degree $-2d-1$, and} \\
&y_1^{i-4d-2},\ldots,y^{i-4d-2}_{m_i} &&\text{in degree $i-4d-2$ for
$i=0,1,\ldots,2d$,}
\end{align*}
and the symplectic form is given by
\begin{equation*}
\om^0=\sum_{i=0}^{2d+1}\sum_{j=1}^{m_i}\dd
x^{-i}_j\,\dd y^{i-4d-2}_j,
\end{equation*}
and everything works out as in Examples \ref{da5ex2}
and~\ref{da5ex3}.
\end{rem}

\begin{ex} 
\label{da5ex5}
Here is a variation on Example \ref{da5ex4}. Choose
$d,A(0),m_0,\ab\ldots,\ab m_{2d+1},\ab x^i_j,y^i_j,z^i_j$ and the
commutative graded algebra $A$ as in Example \ref{da5ex4}. Let
$q_1,\ldots,q_{m_{2d+1}}$ be invertible elements of $A(0)$, and
generalizing \eq{da5eq21} define
\e
\om^0=\sum_{i=0}^{2d}\sum_{j=1}^{m_i}\dd
x^{-i}_j\,\dd y^{i-4d-2}_j+\sum_{j=1}^{m_{2d+1}}\dd\bigl(q_j z^{-2d-1}_j\bigr)\,\dd
z^{-2d-1}_j.
\label{da5eq25}
\e

Choose a Hamiltonian $\Phi$ in $A^{-4d-1}$, which we require to satisfy
the {\it classical master equation\/}
\e
\sum_{i=1}^{2d}\sum_{j=1}^{m_i}\frac{\pd \Phi}{\pd x^{-i}_j}\,
\frac{\pd \Phi}{\pd y^{i-4d-2}_j}+\frac{1}{4}\sum_{j=1}^{m_{2d+1}}
\frac{1}{q_j}\,\biggl(\frac{\pd \Phi}{\pd
z^{-2d-1}_j}\biggr)^2=0\quad\text{in $A^{-4d}$.}
\label{da5eq26}
\e
As for \eq{da5eq23}, define the differential $\d$ on $A$ by $\d=0$
on $A(0)$, and
\begin{gather}
\d x^{0}_j=0,\quad \d y^{-4d-2}_j=\frac{\pd \Phi}{\pd
x^0_j}-\sum_{j'=1}^{m_{2d+1}} \frac{z_{j'}^{-2d-1}}{2q_{j'}}\,\frac{\pd
q_{j'}}{\pd x^0_j}\,\frac{\pd \Phi}{\pd z^{-2d-1}_{j'}}, \quad
j=1,\ldots,m_0,
\nonumber\\
\d x^{-i}_j =(-1)^{i+1}\frac{\pd \Phi}{\pd y^{i-4d-2}_j}, \quad \d
y^{i-4d-2}_j=\frac{\pd \Phi}{\pd x^{-i}_j},\quad\begin{subarray}{l}\ts
i=1,\ldots,2d,\\[6pt] \ts j=1,\ldots,m_i,\end{subarray}
\nonumber\\
\text{and}\qquad\d z^{-2d-1}_j=\frac{1}{2q_j}\,\frac{\pd \Phi}{\pd
z^{-2d-1}_j}, \qquad j=1,\ldots,m_{2d+1}.
\label{da5eq27}
\end{gather}
Similar proofs to \eq{da5eq13}--\eq{da5eq15} show that $\d\ci\d=0$,
so that $(A,d)=A(n)$ for $n=4d+2$ is a standard form cdga over $\K$,
as in Example \ref{da2ex2}, and $\om:=(\om^0,0,0,\ldots)$ is a
$k$-shifted symplectic structure on $\bX=\bSpec A$ for $k=-4d-2$.

Defining 
\begin{align*}
\phi&=-\sum_{i=0}^{2d}\sum_{j=1}^{m_i}\bigl[i\,x^{-i}_j\,\dd y^{i-4d-2}_j+(-1)^{i+1}(4d+2-i)y^{i-4d-2}_j\,\dd x^{-i}_j\bigr]\\
&\qquad-(4d+2)\sum_{j=1}^{m_{2d+1}}q_j\,z^{-2d-1}_j\,\dd z^{-2d-1}_j
\qquad\text{in $(\Om^1_A)^{-4d-2}$,}
\end{align*}
the analogue of \eq{da5eq24}, we find that $(\Phi,\phi)$ satisfy
Proposition \ref{da5prop2}(a) with $k\om$ in place of~$\om$.
\end{ex}

\begin{rem} 
\label{da5rem2}
To relate Examples \ref{da5ex4} and \ref{da5ex5}, note
that if $q_1=\cdots=q_{m_{2d+1}}=1$ then Example \ref{da5ex5}
reduces immediately to Example \ref{da5ex4}. More generally, if
$q_1,\ldots,q_{m_{2d+1}}$ admit square roots $q_i^{1/2}$ in $A(0)$,
then changing variables in $A$ from $x^i_j,y^i_j,z^{-2d-1}_j$ to
$x^i_j,y^i_j$ and $\ti z^{-2d-1}_j=q_j^{1/2}z^{-2d-1}_j$, we find
that Example \ref{da5ex5} reduces to Example \ref{da5ex4} with $\ti
z^{-2d-1}_j$ in place of~$z^{-2d-1}_j$.
\end{rem}

\begin{dfn} 
\label{da5def5}
We say that the standard form cdga $A$ over $\K$ and
$k$-shifted symplectic structure $\om$ on $\bX=\bSpec A$ in Examples
\ref{da5ex2} and \ref{da5ex3} are {\it in Darboux form}. We say that
$A,\om$ in Example \ref{da5ex4} are {\it in strong Darboux form},
and $A,\om$ in Example \ref{da5ex5} are {\it in weak Darboux form}.
\end{dfn}

\subsection{\texorpdfstring{Darboux forms when $k=-1,-2$ and $-3$}{Darboux forms when k=-1,-2 and -3}}
\label{da54}

We now write out the `Darboux form' $k$-shifted symplectic cdgas
$(A,\om)$ of Examples \ref{da5ex2}--\ref{da5ex4} and \ref{da5ex5}
more explicitly in the first three cases $k=-1$, $k=-2$, and $k=-3$.
These correspond to the geometric structures on moduli schemes of
coherent sheaves on Calabi--Yau $m$-folds for~$m=3,4,5$.

\begin{ex} 
\label{da5ex6}
Choose a smooth $\K$-algebra $A(0)$ of dimension $m_0$
and elements $x^0_1,\ldots,x^0_{m_0}\in A(0)$ such that $\dd
x^0_1,\ldots,\dd x^0_{m_0}$ form a basis of $\Om^1_{A(0)}$ over
$A(0)$. Choose an arbitrary Hamiltonian $\Phi\in A(0)$.

Example \ref{da5ex2} with $d=0$ defines $A=A(0)[y^{-1}_1,\ldots,
y^{-1}_{m_0}]$, where $y^{-1}_1,\ab\ldots,\ab y^{-1}_{m_0}$ are
variables of degree $-1$, with differential
\begin{equation*}
\d x^0_i=0,\quad \d y^{-1}_i=\frac{\pd \Phi}{\pd x^0_i}, \quad
i=1,\ldots,m_0,
\end{equation*}
and $-1$-shifted 2-form
\begin{equation*}
\om^0=\dd x^0_1\,\dd y^{-1}_1+\cdots+\dd
x^0_{m_0}\,\dd y^{-1}_{m_0}.
\end{equation*}
Then $\om=(\om^0,0,0,\ldots)$ is a $-1$-shifted symplectic structure
on $\bX=\bSpec A$. We have $H^0(A)=A(0)/(\frac{\pd \Phi}{\pd
x^0_1},\ldots,\frac{\pd \Phi}{\pd x^0_{m_0}})=A(0)/(\dd \Phi)$.

Geometrically, $U=\Spec A(0)$ is a smooth classical $\K$-scheme with
\'etale coordinates $(x^0_1,\ldots,x^0_{m_0}):U\ra\bA^{m_0}$, and
$\Phi:U\ra\bA^1$ is regular, and $\bX$ is the derived critical locus of
$\Phi$, with $X=t_0(\bX)$ the classical critical locus of $\Phi$. As in
Proposition \ref{da5prop2}(b), $\Phi\vert_{X^\red}:X^\red\ra\bA^1$ is
locally constant, and we generally suppose $\Phi$ is chosen so that
$\Phi\vert_{X^\red}=0$. If $X$ is connected this can be achieved by
adding a constant to $\Phi$.

Thus, {\it the important geometric data in writing a $-1$-shifted
symplectic derived\/ $\K$-scheme $(\bX,\om)$ in Darboux form, is a
smooth affine\/ $\K$-scheme $U$ and a regular function
$\Phi:U\ra\bA^1,$ which we may take to satisfy
$\Phi\vert_{\Crit(\Phi)^\red}=0,$ such that\/} $X=t_0(\bX)\cong\Crit(\Phi)$. The remaining data is a choice of \'etale coordinates
$(x^0_1,\ldots,x^0_{m_0}):U\ra \bA^{m_0}$, but this is not very
interesting geometrically.

\end{ex}

\begin{ex} 
\label{da5ex7}
We will work out Example \ref{da5ex5} with $d=0$ and
$k=-2$ in detail. For Example \ref{da5ex4} with $d=0$, set
$q_1=\cdots=q_{m_1}=1$ in what follows.

Choose a smooth $\K$-algebra $A(0)$ of dimension $m_0$ and elements
$x^0_1,\ldots,x^0_{m_0}$ in $A(0)$ such that $\dd x^0_1,\ldots,\dd
x^0_{m_0}$ form a basis of $\Om^1_{A(0)}$ over $A(0)$. Fix $m_1\ge
0$, and as a commutative graded algebra set
$A=A(0)[z^{-1}_1,\ab\ldots,\ab z^{-1}_{m_1},\ab y^{-2}_1,\ab
\ldots,\ab y^{-2}_{m_0}]$, where $z^{-1}_j$ has degree $-1$ and
$y^{-2}_j$ degree $-2$.

Choose invertible functions $q_1,\ldots,q_{m_1}$ in $A(0)$. Define
\begin{align*}
\om^0&=\dd x^0_1\,\dd y^{-2}_1+\cdots+\dd x^0_{m_0}\,\dd y^{-2}_{m_0}\\
&\qquad+\dd\bigl(q_1z^{-1}_1\bigr)\,\dd z^{-1}_1+\cdots+
\dd\bigl(q_{m_1}z^{-1}_{m_1}\bigr)\,\dd z^{-1}_{m_1}
\end{align*}
in $(\La^2\Om^1_A)^{-2}$, as in \eq{da5eq25}. A general element $\Phi$
in $A^{-1}$ may be written
\begin{equation*}
\Phi=z^{-1}_1s_1+\cdots+z^{-1}_{m_1}s_{m_1},
\end{equation*}
for $s_1,\ldots,s_{m_1}\in A(0)$. Then the classical master equation
\eq{da5eq26} reduces to
\e
\frac{(s_1)^2}{q_1}+\cdots+\frac{(s_{m_1})^2}{q_{m_1}}=0\quad\text{in
$A(0)$.}
\label{da5eq28}
\e
By \eq{da5eq27}, the differential $\d$ on $A$ is given by
\begin{equation*}
\d x_i^0=0,\quad \d z_j^{-1}=\frac{s_j}{2q_j}, \quad \d y^{-2}_i=
\sum_{j=1}^{m_1}z_j^{-1}\biggl(\frac{\pd s_j}{\pd x^0_i}
-\frac{s_j}{2q_j}\,\frac{\pd q_j}{\pd x^0_i}\biggr),
\end{equation*}
and $\d\ci\d y^{-2}_i=0$ follows from applying
$\frac{1}{4}\frac{\pd}{\pd x^0_i}$ to \eq{da5eq28}. We have
\begin{equation*}
H^0(A)=A(0)/\bigl(s_1/2q_1,\ldots,s_{m_1}/2q_{m_1}\bigr)=
A(0)/(s_1,\ldots,s_{m_1}),
\end{equation*}
as $q_1,\ldots,q_{m_1}$ are invertible.

Geometrically, we have a smooth classical $\K$-scheme $U=\Spec A(0)$
with \'etale coordinates $(x^0_1,\ldots,x^0_{m_0}):U\ra \bA^{m_0}$,
a trivial vector bundle $E\ra U$ with fibre $\K^{m_1}$, a
nondegenerate quadratic form $Q$ on $E$ given by
$Q(e_1,\ldots,e_{m_1})=q_1^{-1}e_1^2+\cdots+q_{m_1}^{-1}e_{m_1}^2$
for all regular functions $e_1,\ldots,e_{m_1}:U\ra\bA^1$, and a
section $s=(s_1,\ldots,s_{m_1})$ in $H^0(E)$ with $Q(s,s)=0$ by
\eq{da5eq28}. The underlying classical $\K$-scheme $X=t_0(\bX)=\Spec
H^0(A)$ is the $\K$-subscheme $s^{-1}(0)$ in~$U$.

Thus, {\it the important geometric data in writing a $-2$-shifted
symplectic derived\/ $\K$-scheme $(\bX,\om)$ in Darboux form, is a
smooth affine\/ $\K$-scheme $U,$ a vector bundle $E\ra U,$ a
nondegenerate quadratic form $Q$ on $E,$ and a section $s\in H^0(E)$
with\/ $Q(s,s)=0,$ such that\/} $X=t_0(\bX)\cong s^{-1}(0)\subseteq
U$. The remaining data is a choice of \'etale coordinates
$(x^0_1,\ldots,x^0_{m_0}):U\ra \bA^{m_0}$ and a trivialization
$E\cong U\t\bA^{m_1}$, but these are not very interesting
geometrically.
\end{ex}

\begin{ex} 
\label{da5ex8}
We will work out Example \ref{da5ex2} with $d=1$ and
$k=-3$ in detail. Choose a smooth $\K$-algebra $A(0)$ of dimension
$m_0$ and elements $x^0_1,\ldots,x^0_{m_0}$ in $A(0)$ such that $\dd
x^0_1,\ldots,\dd x^0_{m_0}$ form a basis of $\Om^1_{A(0)}$ over
$A(0)$. Fix $m_1\ge 0$, and as a commutative graded algebra set
\begin{equation*}
A=A(0)\bigl[x^{-1}_1,\ldots,x^{-1}_{m_1},y^{-2}_1,
\ldots,y^{-2}_{m_1},y^{-3}_1,\ldots,y^{-3}_{m_0}\bigr],
\end{equation*}
where $x^{-1}_i,y^{-2}_i,y^{-3}_i$ have degrees $-1,-2,-3$
respectively. Define
\begin{align*}
\om^0&=\dd x^0_1\,\dd y^{-3}_1+\cdots+\dd x^0_{m_0}\,\dd y^{-3}_{m_0}\\
&\qquad+\dd x^{-1}_1\,\dd y^{-2}_1+\cdots+\dd x^{-1}_{m_1}\,\dd y^{-2}_{m_1}
\end{align*}
in $(\La^2\Om^1_A)^{-3}$, as in \eq{da5eq9}. A general element $\Phi$
in $A^{-2}$ may be written
\begin{equation*}
\Phi=\ts\sum_{i=1}^{m_1}y^{-2}_is_i+\sum_{i,j=1}^{m_1}x_i^{-1}
x_{j}^{-1}t_{ij},
\end{equation*}
for $s_i,t_{ij}\in A(0)$ with $t_{ij}=-t_{ji}$. Then the classical
master equation \eq{da5eq10} reduces to
$2\sum_{i,j=1}^{m_1}x_i^{-1}t_{ij}s_j=0$ in $A^{-1}$, or
equivalently the $m_1$ equations
\e
\ts\sum_{j=1}^{m_1}t_{ij}s_j=0\quad\text{in $A(0)$ for all
$i=1,\ldots,m_1$.}
\label{da5eq29}
\e

By \eq{da5eq12}, the differential $\d$ on $A$ is given by
\begin{gather*}
\d x^0_i=0,\quad \d x^{-1}_i=s_i,\quad \d y^{-2}_i=2
\ts\sum_{j=1}^{m_1}x_{j}^{-1}t_{ij},\\
\d y^{-3}_{i'}= \ts\sum_{i=1}^{m_1}y^{-2}_i\frac{\pd s_i}{\pd
x^0_{i'}}+\sum_{i,j=1}^{m_1}x_i^{-1}x_{j}^{-1}\frac{\pd t_{ij}}{\pd
x^0_{i'}}.
\end{gather*}
Thus $H^0(A)=A(0)/(s_1,\ldots,s_{m_1})$.

Geometrically, we have a smooth classical $\K$-scheme $U=\Spec A(0)$
equipped with \'etale coordinates $(x^0_1,\ldots,x^0_{m_0}):U\ra
\bA^{m_0}$, a trivial vector bundle $E\ra U$ over $U$ with fibre
$\K^{m_1}$, and sections $s=(s_1,\ldots,s_{m_1})$ in $H^0(E)$ and
$t=(t_{ij})_{i,j=1}^{m_1}$ in $H^0(\La^2E^*)$, where $t$ need not be
nondegenerate, such that regarding $t$ as a morphism $E\ra E^*$ we
have $t\ci s=0$ by \eq{da5eq29}. The underlying classical
$\K$-scheme $X=t_0(\bX)=\Spec H^0(A)$ is the $\K$-subscheme $s=0$
in~$U$.

Thus, {\it the important geometric data in writing a $-3$-shifted
symplectic derived\/ $\K$-scheme $(\bX,\om)$ in Darboux form, is a
smooth affine\/ $\K$-scheme $U,$ a vector bundle $E\ra U,$ and
sections $s\in H^0(E)$ and\/ $t\in H^0(\La^2E^*)$ such that\/ $t\ci
s=0\in H^0(E^*),$ regarding $t$ as morphism $E\ra E^*,$ and\/}
$X=t_0(\bX)\cong s^{-1}(0)\subseteq U$. The remaining data is a
choice of \'etale coordinates $(x^0_1,\ldots,x^0_{m_0}):U\ra
\bA^{m_0}$ and a trivialization $E\cong U\t\bA^{m_1}$, but these are
not very interesting geometrically.
\end{ex}

\subsection{A Darboux-type theorem for derived schemes}
\label{da55}

Here is the main theorem of this paper, a $k$-shifted analogue of
Darboux' theorem in symplectic geometry, which will be proved
in~\S\ref{da56}.

\begin{thm} 
\label{da5thm1}
Let\/ $\bX$ be a derived\/ $\K$-scheme with\/
$k$-shifted symplectic form $\ti\om$ for $k<0,$ and\/ $x\in\bX$.
Then there exists a standard form cdga $A$ over $\K$ which is
minimal at\/ $p\in\Spec H^0(A),$ a $k$-shifted symplectic form\/
$\om$ on $\bSpec A,$ a morphism $\bs f:\bSpec A\ra\bX$ with\/
$\bs f(p)=x$, and\/ a path $\bs f^*(\ti\om)\sim\om$ in the space of\/ $k$-shifted closed\/ $2$-forms on $\bSpec A$ such that
\begin{itemize}
\setlength{\itemsep}{0pt}
\setlength{\parsep}{0pt}
\item[{\bf(i)}] If\/ $k$ is odd or divisible by $4,$ then $\bs f$ is
a Zariski open inclusion, and\/ $A,\om$ are in Darboux form, as
in Examples\/ {\rm\ref{da5ex2}} and\/~{\rm\ref{da5ex3}}.
\item[{\bf(ii)}] If\/ $k\equiv 2\mod 4,$ then $\bs f$ is
a Zariski open inclusion, and\/ $A,\om$ are in weak Darboux
form, as in Example\/~{\rm\ref{da5ex5}}.
\item[{\bf(iii)}] Alternatively, if\/ $k\equiv 2\mod 4,$ then we may
instead take $\bs f$ to be \'etale, and\/ $A,\om$ to be in
strong Darboux form, as in Example\/~{\rm\ref{da5ex4}}.
\end{itemize}
\end{thm}

The theorem has interesting consequences even in classical algebraic
geometry. For example, let $Y$ be a Calabi--Yau $m$-fold over $\K$,
that is, a smooth projective $\K$-scheme with trivial canonical bundle. 
Suppose $\cM$ is a classical moduli $\K$-scheme of simple coherent sheaves in $\coh(Y)$, where we
call $F\in\coh(Y)$ {\it simple\/} if $\Hom(F,F)=\K$. More generally,
suppose $\cM$ is a moduli $\K$-scheme of simple complexes of
coherent sheaves in $D^b\coh(Y)$, where we call $F^\bu\in
D^b\coh(Y)$ {\it simple\/} if $\Hom(F^\bu,F^\bu)=\K$ and
$\mathop{\rm Ext}^{<0}(F^\bu,F^\bu)=0$. (Such moduli spaces $\cM$ are
only known to be algebraic $\K$-spaces in general, but we assume
$\cM$ is a $\K$-scheme. The case of algebraic spaces can be treated
by using \'etale in place of Zariski neighbourhoods.)

Then $\cM=t_0(\bs\cM)$, for $\bs\cM$ the corresponding derived
moduli $\K$-scheme. To make $\cM,\bs\cM$ into schemes rather than
stacks, we consider moduli of sheaves or complexes with fixed
determinant. Then Pantev et al.\ \cite[\S 2.1]{PTVV} prove $\bs\cM$
has a $(2-m)$-shifted symplectic structure $\om$, so Theorem
\ref{da5thm1} shows that $(\bs\cM,\om)$ is Zariski locally modelled
on $(\bSpec A,\om)$ as in Examples \ref{da5ex2}--\ref{da5ex4} and
\ref{da5ex5}, and $\cM$ is Zariski locally modelled on $\Spec
H^0(A)$. In the case $m=3$, so that $k=-1$, from Example
\ref{da5ex6} we deduce:

\begin{cor} 
\label{da5cor1}
Suppose $Y$ is a Calabi--Yau\/ $3$-fold over a field\/
$\K,$ and\/ $\cM$ is a classical moduli $\K$-scheme of simple
coherent sheaves, or simple complexes of coherent sheaves, on $Y$.
Then for each\/ $[F]\in\cM,$ there exist a smooth\/ $\K$-scheme $U$
with\/ $\dim U=\dim\Ext^1(F,F),$ a regular function $f:U\ra\bA^1,$
and an isomorphism from $\Crit(f)\subseteq U$ to a Zariski open
neighbourhood of\/ $[F]$ in\/~$\cM$.
\end{cor}

Here $\dim U=\dim\Ext^1(F,F)$ comes from $A$ minimal at $p$ and $\bs
f(p)=[F]$ in Theorem \ref{da5thm1}. Related results are important in
Donaldson--Thomas theory \cite{JoSo,KoSo1,KoSo2}. When $\K=\C$ and
$\cM$ is a moduli space of simple coherent sheaves on $Y$, using
gauge theory and transcendental complex methods, Joyce and Song
\cite[Th.~5.4]{JoSo} prove that the underlying complex analytic
space $\cM^\an$ of $\cM$ is locally of the form $\Crit(f)$ for $U$ a
complex manifold and $f:U\ra\C$ a holomorphic function. Behrend and
Getzler announced the analogue of \cite[Th.~5.4]{JoSo} for moduli of
complexes in $D^b\coh(Y)$, but the proof has not yet appeared. Over
general $\K$, as in Kontsevich and Soibelman \cite[\S 3.3]{KoSo1}
the formal neighbourhood $\hat\cM_{[F]}$ of $\cM$ at any $[F]\in\cM$
is isomorphic to the critical locus $\Crit(\hat f)$ of a formal
power series $\hat f$ on $\Ext^1(F,F)$ with only cubic and higher
terms.

In the case $m=4$, so that $k=-2$, from Example \ref{da5ex7} we
deduce a local description of Calabi--Yau 4-fold moduli schemes,
which may be new:

\begin{cor} 
\label{da5cor2}
Suppose $Y$ is a Calabi--Yau\/ $4$-fold over a field\/
$\K,$ and\/ $\cM$ is a classical moduli $\K$-scheme of simple
coherent sheaves, or simple complexes of coherent sheaves, on $Y$.
Then for each\/ $[F]\in\cM,$ there exist a smooth\/ $\K$-scheme $U$
with\/ $\dim U=\dim\Ext^1(F,F),$ a vector bundle $E\ra U$ with\/
$\rank E= \dim\Ext^2(F,F),$ a nondegenerate quadratic form $Q$ on
$E,$ a section $s\in H^0(E)$ with\/ $Q(s,s)=0,$ and an isomorphism
from $s^{-1}(0)\subseteq U$ to a Zariski open neighbourhood of\/
$[F]$ in\/~$\cM$.
\end{cor}

If $(S,\om)$ is an algebraic symplectic manifold over $\K$ (that is,
a 0-shifted symplectic derived $\K$-scheme in the language of
\cite{PTVV}) and $L,M\subseteq S$ are Lagrangians, then Pantev et
al.\ \cite[Th.~2.10]{PTVV} show that the derived intersection
$\bX=L\t_S M$ has a $-1$-shifted symplectic structure. So Theorem
\ref{da5thm1} and Example \ref{da5ex6} imply:

\begin{cor} 
\label{da5cor3}
Suppose $(S,\om)$ is an algebraic symplectic manifold,
and\/ $L,M$ are algebraic Lagrangian submanifolds in $S$. Then the
intersection $X=L\cap M,$ as a classical\/ $\K$-subscheme of\/ $S,$
is Zariski locally modelled on the critical locus $\Crit(f)$ of a
regular function $f:U\ra\bA^1$ on a smooth\/ $\K$-scheme\/~$U$.
\end{cor}

In real or complex symplectic geometry, it is easy to prove
analogues of Corollary \ref{da5cor3} using Darboux' Theorem or the
Lagrangian Neighbourhood Theorem. However, these do not hold for
algebraic symplectic manifolds, so it is not obvious how to prove
Corollary \ref{da5cor3} using classical techniques.

\subsection{Proof of Theorem \ref{da5thm1}}
\label{da56}

We divide the proof into four steps.
\smallskip

\noindent{\bf Step 1: locally represent $\bX$ by a minimal standard form cdga $A$.}

\noindent For all $k<0$, first apply Theorem \ref{da4thm1} to get a standard
form cdga $A$ over $\K$ minimal at $p\in\Spec H^0(A)$, and a Zariski
open inclusion $\bs f:\bSpec A\ra\bX$ with $\bs f(p)=x$. Then $\bs
f^*(\ti\om)$ is a closed 2-form of degree $k$ on $\bSpec A$. So
Proposition \ref{da5prop2}(a) with $k\bs f^*(\ti\om)$ in place of $\om$ gives $\Phi\in A^{k+1}$ and $\phi\in(\Om^1_A)^k$ such that $\d\Phi=0$ in $A^{k+2}$, and $\dd\Phi+\d\phi=0$ in $(\Om^1_A)^{k+1}$, and
\begin{equation*}
\bs f^*(\ti\om)\sim\ts\frac{1}{k}(\dd\phi,0,0,\ldots)=:(\om^0,0,0,\ldots)=:\om.
\end{equation*}

By Proposition \ref{da2prop}, $\Om^1_A\ot_A H^0(A)$ is a complex of
free $H^0(A)$-modules
\begin{equation*}
\xymatrix@C=27pt{ 0 \ar[r] & V^k \ar[r]^{\d^k} & V^{k+1}
\ar[r]^{\d^{k+2}} & \cdots  \ar[r]^{\d^{-2}} & V^{-1} \ar[r]^{\d^{-1}}
& V^0 \ar[r] & 0, }
\end{equation*}
with $\d^i\vert_p=0$ for $i=k,k+1,\ldots,-1$ by Definition
\ref{da2def6}, as $A$ is minimal at $p$. Since $\om$ is a
$k$-shifted symplectic form on $\bSpec A$, the morphism
$\om^0:\bT_A\ra\Om^1_A[k]$ in \eq{da5eq3} is a quasi-isomorphism.
Hence $\om^0\ot\id_{H^0(A)}:\bT_A\ot_A H^0(A)\ra\Om^1_A\ot_A
H^0(A)[k]$ is a quasi-isomorphism. That is, in the commutative
diagram
\e
\begin{gathered}
\xymatrix@C=11pt@R=15pt{ 0 \ar[r] & (V^0)^* \ar[d]^{\om^0}
\ar[rr]_{(\d^{-1})^*} && (V^{-1})^* \ar[d]^{\om^0}
\ar[rr]_{(\d^{-2})^*} && \cdots \ar[rr]_{(\d^{k+1})^*} &&
(V^{k+1})^* \ar[d]^{\om^0} \ar[rr]_{(\d^k)^*} &&
(V^k)^* \ar[d]^{\om^0} \ar[r] & 0 \\
0 \ar[r] & V^k \ar[rr]^{\d^k} && V^{k+1} \ar[rr]^{\d^{k+2}} &&
\cdots \ar[rr]^{\d^{-2}} && V^{-1} \ar[rr]^{\d^{-1}} && V^0 \ar[r] &
0,\!\!{} }\!\!\!\!\!\!{}
\end{gathered}
\label{da5eq30}
\e
the columns are a quasi-isomorphism. As the horizontal differentials
$\d^i,(\d^i)^*$ are zero at $p$, so the vertical maps are
isomorphisms at $p$, and hence isomorphisms in a neighbourhood of
$p$. Localizing $A$ at $p$ if necessary, we may therefore assume the
vertical maps $\om^0:(V^{k-i})^*\ra V^{i}$ in \eq{da5eq30} are
isomorphisms.
\smallskip

\noindent{\bf Step 2: proof of Theorem \ref{da5thm1}(i) when $k$ is odd.}

\noindent For the next part of the proof we first suppose $k$ is odd, so that
$k=-2d-1$ for $d=0,1,\ldots.$ Localizing $A$ at $p$ if necessary,
choose $x^0_1,\ldots, x^0_{m_0}\in A(0)$ such that $\dd
x^0_1,\ldots,\dd x^0_{m_0}$ form a basis of $\Om^1_{A(0)}\cong V^0$
over $A(0)$. Next, for $i=1,\ldots,d$, choose
$x^{-i}_1,\ldots,x^{-i}_{m_i}$ in $A^{-i}$ such that $\dd
x^{-i}_1,\ldots,\dd x^{-i}_{m_i}$ form a basis for $V^{-i}$ over
$A(0)$. Then, for $i=0,\ldots,d$, choose
$y^{i-2d-1}_1,\ldots,y^{i-2d-1}_{m_i}$ in $A^{i-2d-1}$ such that
$\dd y^{i-2d-1}_1,\ab\ldots,\ab\dd y^{i-2d-1}_{m_i}$ are the basis
of $V^{i-2d-1}$ over $A(0)$ which is dual to the basis $\dd
x^{-i}_1,\ldots,\dd x^{-i}_{m_i}$ for $V^{-i}$ under the
isomorphism~$\om^0:(V^{i-2d-1})^*\ra V^{-i}$.

Since $A$ is a standard form cdga, these variables $x^i_j$ for $i<0$
and $y^i_j$ generate $A$ freely over $A(0)$ as a commutative graded
algebra. That is, as a commutative graded algebra, $A$ is freely
generated over $A(0)$ by the graded variables \eq{da5eq8}, exactly
as in Example \ref{da5ex2}. Then the condition on $\om^0$ sending
the dual basis of $\dd x^{-i}_1,\ldots,\dd x^{-i}_{m_i}$ to $\dd
y^{i-2d-1}_1,\ab\ldots,\ab\dd y^{i-2d-1}_{m_i}$ implies that
\e
\om^0\ot\id_{A(0)}=\sum_{i=0}^d\sum_{j=1}^{m_i}\dd
x^{-i}_j\,\dd y^{i-2d-1}_j \quad \text{in $(\La^2\Om^1_A)^{-2d-1}\ot_AA(0)$.}
\label{da5eq31}
\e

As above we have $\Phi,\phi$ with $\dd\Phi+\d\phi=0$ and
$\dd\phi=k\om^0$. Using the coordinates $x^i_j,y^i_j$ we may write
\begin{equation*}
\phi=\sum_{i=0}^d\sum_{j=1}^{m_i}\bigl[a_j^{i-2d-1}\dd x^{-i}_{j}+
b_j^{-i}\dd y^{i-2d-1}_{j}\bigr],
\end{equation*}
with $a_j^l,b_j^l\in A^l$. For degree reasons, the $b_j^l$ depend on $A(0)$ and the $x^{-i}_{j'}$, but do not involve the $y^{i-2d-1}_{j'}$. By leaving $\om^0$ unchanged but replacing $\Phi,\phi$ by
\begin{equation*}
\ti\Phi=\Phi-\d\biggl[\sum_{i=0}^d\sum_{j=1}^{m_i}(-1)^ib^{-i}_j
y^{i-2d-1}_{j}\biggr], \quad \ti\phi=\phi-
\dd\biggl[\sum_{i=0}^d\sum_{j=1}^{m_i}(-1)^ib^{-i}_j
y^{i-2d-1}_{j}\biggr],
\end{equation*}
noting that $\dd b_j^l$ includes no terms in $\dd y^{i-2d-1}_{j'}$, we may assume that $b^{-i}_j=0$ for all $i,j$. Then $\dd\phi=k\om^0$ gives
\e
k\om^0=\sum_{i=0}^d\sum_{j=1}^{m_i}\dd a_j^{i-2d-1}\,\dd x^{-i}_{j}.
\label{da5eq32}
\e

Comparing \eq{da5eq31} and \eq{da5eq32} shows that
\begin{equation*}
a_j^{i-2d-1}=ky^{i-2d-1}_j+\text{degree $\ge 2$ terms in
$x^{-i'}_{j'}$ for $i>0$ and $y^{i''-2d-1}_{j''}$.}
\end{equation*}
Thus the $\frac{1}{k}a_j^{i-2d-1}$ are alternative choices for the $y^{i-2d-1}_j$ above. Hence, replacing $y^{i-2d-1}_j$ by
$\frac{1}{k}a_j^{i-2d-1}$ for all $i,j$, we see that $\om^0$ is given by
\eq{da5eq9}, and
\begin{equation*}
\phi=k\sum_{i=0}^d\sum_{j=1}^{m_i}y_j^{i-2d-1}\dd x^{-i}_{j}.
\end{equation*}
Leaving $\om^0$ unchanged but replacing $\Phi,\phi$ by
\begin{align*}
\ti\Phi&=\Phi-\d\biggl[\sum_{i=0}^d\sum_{j=1}^{m_i}(-1)^iix^{-i}_jy^{i-2d-1}_{j}\biggr], \\
\ti\phi&=\phi-\dd\biggl[\sum_{i=0}^d\sum_{j=1}^{m_i}(-1)^ii
x^{-i}_jy^{i-2d-1}_{j}\biggr],
\end{align*}
we find that $\phi$ is given by~\eq{da5eq16}.

Let us summarize our progress so far. In the case $k=-2d-1$ for
$d=0,1,\ldots,$ we have shown that we can identify $A$ as a
commutative graded algebra with the commutative graded algebra $A$
in Example \ref{da5ex2} generated over $A(0)$ by the graded
variables \eq{da5eq8}, we have a $k$-shifted symplectic form
$\om=(\om^0,0,0,\ldots)$ on $\bSpec A$ with $\om^0$ given by
\eq{da5eq9}, we have $\Phi\in A^{k+1}$ and $\phi\in(\Om^1_A)^k$ such
that $\d\Phi=0$, $\dd\Phi+\d\phi=0$, $\dd\phi=k\om^0$, and $\phi$ is
as in~\eq{da5eq16}.

It remains to show that $\Phi$ satisfies the classical master equation
\eq{da5eq10}, and the differential $\d$ on $A$ is given by
\eq{da5eq12}. Expanding $\dd\Phi+\d\phi=0$ using \eq{da5eq16}, comparing coefficients of $\dd y^{i-2d-1}_j,\dd x^{-i}_j$, and rearranging gives
\ea
&(2d+1)\d x_j^{-i}=\frac{\pd}{\pd y_j^{i-2d-1}}\raisebox{-3pt}{$\displaystyle\Biggl[$}\Phi-
\sum_{i'=0}^d\sum_{j'=1}^{m_{i'}}\begin{aligned}[t] &(i'-2d-1)y_{j'}^{i'-2d-1}\d x_{j'}^{-i'}\\
&\qquad -i'x_{j'}^{-i'}\d y_{j'}^{i'-2d-1}\end{aligned}\raisebox{-3pt}{$\displaystyle\Biggr]$},
\label{da5eq33}\\
&(2d+1)\d y_j^{i-2d-1}=\frac{\pd}{\pd x_j^{-i}}\raisebox{-3pt}{$\displaystyle\Biggl[$}\Phi-
\sum_{i'=0}^d\sum_{j'=1}^{m_{i'}}\begin{aligned}[t] &(i'-2d-1)y_{j'}^{i'-2d-1}\d x_{j'}^{-i'}\\
&\qquad -i'x_{j'}^{-i'}\d y_{j'}^{i'-2d-1}\end{aligned}\raisebox{-3pt}{$\displaystyle\Biggr]$}.
\label{da5eq34}
\ea
Write $F$ for the function in brackets $[\cdots]$ on the right hand sides of \eq{da5eq33}--\eq{da5eq34}. Using \eq{da5eq33}--\eq{da5eq34} to substitute for $\d x_{j'}^{-i'},\d y_{j'}^{i'-2d-1}$ gives
\begin{align*}
F&=\Phi-\frac{1}{2d+1}
\sum_{i'=0}^d\sum_{j'=1}^{m_{i'}}(i'-2d-1)y_{j'}^{i'-2d-1}\frac{\pd F}{\pd y_{j'}^{i'-2d-1}}-i'x_{j'}^{-i'}\frac{\pd F}{\pd x_{j'}^{-i'}}\\
&=\Phi-\frac{1}{2d+1}\cdot (-2d)\, F,
\end{align*}
where the second line holds as $F$ has degree $-2d$. Therefore $F=(2d+1)\Phi$, so \eq{da5eq33}--\eq{da5eq34} prove \eq{da5eq12}. Then expanding $\d\Phi=0$ using \eq{da5eq12} yields \eq{da5eq10}. This completes the proof of Theorem \ref{da5thm1}(i) when $k$ is odd.
\smallskip

\noindent{\bf Step 3: proof of Theorem \ref{da5thm1}(i),(ii) when $k$ is even.}

\noindent The remaining cases in Theorem \ref{da5thm1}(i),(ii) are fairly
similar, so we explain the differences with the case of $k$ odd. If
$k$ is divisible by 4, so that $k=-4d$ for $d=1,2,\ldots,$ choose
$x^0_1,\ldots, x^0_{m_0}\in A^0=A(0)$ and
$x^{-i}_1,\ldots,x^{-i}_{m_i}\in A^{-i}$ for $i=1,\ldots,2d-1$ as
above such that $\dd x^{-i}_1,\ldots,\dd x^{-i}_{m_i}$ form a basis
for $V^{-i}$ over $A(0)$ for $i=0,\ldots,2d-1$, and for
$i=0,\ldots,2d-1$, choose $y^{i-4d}_1,\ldots,y^{i-4d}_{m_i}$ in
$A^{i-4d}$ as above such that $\dd y^{i-4d}_1,\ab\ldots,\ab\dd
y^{i-4d}_{m_i}$ are the basis of $V^{i-4d}$ over $A(0)$ which is
dual to the basis $\dd x^{-i}_1,\ldots,\dd x^{-i}_{m_i}$ for
$V^{-i}$ under the isomorphism~$\om^0:(V^{i-4d})^*\ra V^{-i}$.

We have chosen the variables $x^{-i}_j,y^{i-4d}_j$ except in the
middle degree $-2d$. The isomorphism $\om^0:(V^{-2d})^*\ra V^{-2d}$
is antisymmetric, so we can choose
$x^{-2d}_1,\ldots,x^{-2d}_{m_{2d}},y^{-2d}_1,\ldots,y^{-2d}_{m_{2d}}
\!\in\! A^{-2d}$ with $\dd x^{-2d}_1,\ldots,\ab\dd
x^{-2d}_{m_{2d}},\ab\dd y^{-2d}_1,\ab\ldots,\ab\dd y^{-2d}_{m_{2d}}$
a standard symplectic basis of $V^{-2d}$ over $A(0)$. The rest of
the proof goes through with only cosmetic changes, following Example
\ref{da5ex3} rather than Example \ref{da5ex2}, to prove Theorem
\ref{da5thm1}(i) when $k$ is divisible by~4.

Next suppose that $k\equiv 2\mod 4$, so that $k=-4d-2$ for
$d=0,1,\ldots.$ Choose $x^0_1,\ldots, x^0_{m_0}\in A^0=A(0)$ and
$x^{-i}_1,\ldots,x^{-i}_{m_i}\in A^{-i}$ for $i=1,\ldots,2d$ as
above such that $\dd x^{-i}_1,\ldots,\dd x^{-i}_{m_i}$ form a basis
for $V^{-i}$ over $A(0)$ for $i=0,\ldots,2d$, and for
$i=0,\ldots,2d$, choose $y^{i-4d-2}_1,\ldots,y^{i-4d-2}_{m_i}$ in
$A^{i-4d-2}$ as above such that $\dd y^{i-4d-2}_1,\ab\ldots,\ab\dd
y^{i-4d-2}_{m_i}$ are the basis of $V^{i-4d-2}$ over $A(0)$ dual to
the basis $\dd x^{-i}_1,\ldots,\dd x^{-i}_{m_i}$ for $V^{-i}$ under
the isomorphism~$\om^0:(V^{i-4d})^*\ra V^{-i}$.

We have chosen the $x^{-i}_j,y^{i-4d-2}_j$ except in the middle
degree $-2d-1$. In this case the isomorphism $\om^0:(V^{-2d-1})^*\ra
V^{-2d-1}$ is symmetric, that is, $\om^0$ is a nondegenerate
quadratic form on $(V^{-2d-1})^*$. Now in general, nondegenerate
quadratic forms cannot be trivialized Zariski locally, but they can
at least be diagonalized. That is, in general we cannot choose a
basis of $V^{-2d-1}$ over $A(0)$ in which $\om^0$ is represented by
a constant matrix, but we can choose one in which it is represented
by a diagonal matrix.

Thus, we may choose $z_1^{-2d-1},\ldots,z_{m_{2d+1}}^{-2d-1}$ in
$A^{-2d-1}$ such that $\dd z_1^{-2d-1},\ab\ldots,\ab\dd
z_{m_{2d+1}}^{-2d-1}$ form a basis for $V^{-2d-1}$ over $A(0)$, and
there exist invertible elements $q_1,\ldots,q_{m_{2d+1}}$ in $A(0)$
such that $\om^0:(V^{-2d-1})^*\ra V^{-2d-1}$ has matrix $\mathop{\rm
diag}(q_1,\ldots,q_{m_{2d+1}})$ with respect to $\dd
z_1^{-2d-1},\ldots,\dd z_{m_{2d+1}}^{-2d-1}$ and its dual basis.
Then the standard form cdga $A$ is generated over $A(0)$ by the
variables \eq{da5eq20}, so as a commutative graded algebra $A$ is as
in Examples \ref{da5ex4} and \ref{da5ex5}. Also the analogue of
\eq{da5eq25} is
\begin{equation*}
\om^0\ot\id_{A(0)}=\sum_{i=0}^{2d}\sum_{j=1}^{m_i}\dd x^{-i}_j\,\dd
y^{i-4d-2}_j+\sum_{j=1}^{m_{2d+1}}q_j\,\dd z^{-2d-1}_j\,\dd z^{-2d-1}_j
\end{equation*}
in $(\La^2\Om^1_A)^{-4d-2}\ot_AA(0)$, which is a truncation of
\eq{da5eq25}. The rest of the proof then works as above, but with
extra terms in $z^{-2d-1}_j,q_j$ and $\smash{\frac{\pd q_j}{\pd
x^0_{j'}}}$ inserted as in Example \ref{da5ex5}. This proves
Theorem~\ref{da5thm1}(ii).
\smallskip

\noindent{\bf Step 4: proof of Theorem \ref{da5thm1}(iii).}

\noindent Finally, for Theorem \ref{da5thm1}(iii) we first apply part (ii) to
get $\check A,\check\om$ in weak Darboux form, as in Example
\ref{da5ex5}, with $\check A$ minimal at $\check p\in\Spec
H^0(\check A)$, and a Zariski open inclusion $\bs{\check
f}:\bSpec\check A\ra\bX$ with $\bs{\check f}(p)=x$. Then we have
invertible elements $q_1,\ldots,q_{m_{2d+1}}$ in $\check A(0)$.
Define $A(0)=\check A(0)[q_1^{1/2},\ldots,q_{m_{2d+1}}^{1/2}]$ to be
the $\K$-algebra obtained by adjoining square roots of
$q_1,\ldots,q_{m_{2d+1}}$ to $\check A(0)$. Then $A(0)$ is a smooth
$\K$-algebra with inclusion morphism $i_{A(0)}:\check A(0)\ra A(0)$,
such that $\Spec i_{A(0)}:\Spec A(0)\ra\Spec \check A(0)$ is an
\'etale cover of degree $2^{m_{2d+1}}$. Set $A=\check
A\ot_{\smash{\check A(0)}}A(0)$, with inclusion morphism $i_A:\check
A\ra A$. Then $\bSpec i_A:\bSpec A\ra\bSpec\check A$ is an \'etale
cover of derived $\K$-schemes of degree $2^{m_{2d+1}}$. Let
$p\in\bSpec A$ be one of the $2^{m_{2d+1}}$ preimages of $\check p$,
and set $\bs f=\bs{\check f}\ci\bSpec i_A$. Then $\bs f:\bSpec
A\ra\bX$ is \'etale, with $\bs f(p)=x$, as we want.

In $\check A$ we have variables $\check x{}^{-i}_j,\check
y{}^{i-4d-2}_j,\check z{}^{-2d-1}_j$ and a Hamiltonian $\check \Phi$.
Set
\begin{equation*}
x^{-i}_j\!=\!i_A(\check x{}^{-i}_j),\;
y{}^{i-4d-2}_j\!=\!i_A(\check y{}^{i-4d-2}_j),\;
z{}^{-2d-1}_j\!=\!q_j^{1/2}\cdot i_A(\check z{}^{-2d-1}_j),\;
\Phi\!=\!i_A(\check \Phi)
\end{equation*}
in $A$, for all $i,j$. As in Remark \ref{da5rem2}, changing
coordinates from $\check z{}^{-2d-1}_j$ to $z{}^{-2d-1}_j=
q_j^{1/2}\cdot\check z{}^{-2d-1}_j$ has the effect of setting
$q_j=1$, and transforms weak Darboux form in Example \ref{da5ex5} to
strong Darboux form in Example \ref{da5ex4}. One can check that
these $x^{-i}_j,y^{i-4d-2}_j,z^{-2d-1}_j,\Phi$ satisfy the conditions
of Example \ref{da5ex4}, and Theorem \ref{da5thm1}(iii) follows.

\subsection{Comparing Darboux form presentations on overlaps}
\label{da57}

Let $(\bX,\ti\om)$ be a $k$-shifted symplectic derived $\K$-scheme
for $k<0$. Then for $x\in\bX$, Theorem \ref{da5thm1} gives (Zariski
or \'etale local) presentations $\bs f:\bSpec A\ra\bX$ near $x$ with
$\bs f^*(\ti\om)\sim\om=(\om^0,0,\ldots)$ for $(A,\om)$ of (weak or
strong) Darboux form, so we have a Hamiltonian $\Phi\in A^{k+1}$, and
$\phi\in (\Om^1_A)^k$ with $\dd\Phi+\d\phi=0$ and $\dd\phi=k\om^0$ as in Proposition \ref{da5prop2}(a). In the case $k=-1$ we also suppose that $\Phi\vert_{\Spec H^0(A)^\red}=0$, as in Proposition \ref{da5prop2}(b). We think of $A,\ab\om,\ab\bs f,\ab\ab\Phi,\ab\phi$ as like coordinates on $\bX$ near $x$ which write $\bX,\ti\om$ in a nice way.

It is often important in geometric problems to compare different
choices of coordinates on the overlap of their domains. So suppose
$A,\om,\bs f,\Phi,\phi$ and $B,\check\om,\bs g,\check\Phi,\check\phi$ are two choices as above, and $p\in\Spec H^0(A)$, $q\in\Spec H^0(B)$ with $\bs f(p)=\bs g(q)=x$ in $\bX$. We would like to compare the presentations $A,\om,\bs f,\Phi,\phi$ and $B,\check\om,\bs g,\check\Phi,\check\phi$ for $\bX$ near~$x$.

Here is a general method for doing this:
\begin{itemize}
\setlength{\itemsep}{0pt}
\setlength{\parsep}{0pt}
\item[(i)] Apply Theorem \ref{da4thm2} to $\bs f:\bSpec A\ra\bX$ and
$\bs g:\bSpec B\ra\bX$ at $p\in\Spec H^0(A)$ and $q\in\Spec
H^0(B)$. This gives a standard form cdga $C$ over $\K$ minimal
at $r\in\Spec H^0(C)$, and cdga morphisms $\al:A\ra C$,
$\be:B\ra C$ with $\bSpec\al:r\mapsto p$, $\bSpec\be:r\mapsto
q$, with $\bs f\ci\bSpec\al\simeq \bs g\ci\bSpec\be$ as
morphisms $\bSpec C\ra\bX$ in~$\dSch_\K$.

Here $\al,\be$ are Zariski local inclusions if $\bs f,\bs g$ are
(if $A,\ldots,\phi$ and $B,\ldots,\check\phi$ come from Theorem
\ref{da5thm1}(i) or (ii)), and \'etale if $\bs f,\bs g$ are (if
$A,\ldots,\phi$ and $B,\ldots,\check\phi$ come from
Theorem~\ref{da5thm1}(iii)).
\item[(ii)] As $\al,\be$ are Zariski local inclusions or
\'etale, the pushforwards $\al_*(\om)$, $\be_*(\check\om)$ are
$k$-shifted symplectic structures on $\bSpec C$, which are
equivalent as
\begin{align*}
\al_*(\om)&\simeq(\bSpec\al)^*\ci\bs f^*(\ti\om)\simeq
(\bs f\ci\bSpec\al)^*(\ti\om)\\
&\simeq(\bs g\ci\bSpec\be)^*(\ti\om)\simeq(\bSpec\be)^*
\ci\bs g^*(\ti\om)\simeq\be_*(\check\om),
\end{align*}
using $\bs f^*(\ti\om)\sim\om$, $\bs g^*(\ti\om)\sim\check\om$
and $\bs f\ci\bSpec\al\simeq \bs g\ci\bSpec\be$. Also
$\al(\Phi),\al_*(\phi)$ and $\be(\check\Phi),\be_*(\check\phi)$
satisfy Proposition \ref{da5prop2}(a) for $C,\al_*(\om)$ and
$C,\be_*(\check\om)$ respectively, as $\Phi,\phi$ and
$\check\Phi,\phi$ do for $A,\om$ and $B,\check\om$.

Noting that $\al(\Phi)\vert_{\Spec H^0(C)^\red}=0=
\be(\check\Phi)\vert_{\Spec H^0(C)^\red}$ when $k=-1$,
Proposition \ref{da5prop2}(c) applies, yielding $\Psi\in C^k$
and $\psi\in(\Om^1_C)^{k-1}$ with
\e
\begin{aligned}
\al(\Phi)-\be(\check\Phi)&=\d\Psi &&\text{in $C^{k+1}$, and}\\
\al_*(\phi)-\al_*(\check\phi)&=\dd\Psi+\d\psi &&\text{in
$(\Om^1_C)^k.$}
\end{aligned}
\label{da5eq35}
\e
The data $C,\al,\be,\Psi,\psi$ compare the Darboux presentations
$A,\om,\bs f,\Phi,\phi$ and $B,\ab\check\om,\ab\bs g,\ab\check\Phi,\ab\check\phi$ for $\bX$ near $x$.
\end{itemize}

We work out this comparison more explicitly in the case~$k=-1$:

\begin{ex} 
\label{da5ex9}
Let $(\bX,\ti\om)$ be a $-1$-shifted symplectic derived
$\K$-scheme, and suppose $A,\om,\bs f,\Phi,\phi$ and
$B,\check\om,\bs g,\check\Phi,\check\phi$ and $p,q$ with
$\bs f(p)=\bs g(q)$ are as above. Follow (i)--(iii) above to get
$C,r,\al,\be,\Psi,\psi$ satisfying \eq{da5eq35}. Set $U=\Spec
A(0)$, $V=\Spec B(0)$ and $W=\Spec C(0)$, so that $U,V,W$ are smooth
$\K$-schemes with $p\in U$, $q\in V$, $r\in W$, and write
$a=\Spec\al(0):W\ra U$ and $b=\Spec\be(0):W\ra V$. Then $a(r)=p$
and~$b(r)=q$.

Since $A,\om,\Phi,\phi$ are as in Examples \ref{da5ex2} and
\ref{da5ex6}, we have \'etale coordinates
$(x_1^0,\ldots,x_{m_0}^0):U\ra\bA^{m_0}$ on $U$ such that
$A=A(0)[y_1^{-1},\ldots,y_{m_0}^{-1}]$ with $\d x^0_i=0$, $\d
y^{-1}_i=\frac{\pd \Phi}{\pd x^0_i}$ for all $i$, and
\e
\begin{split}
\om^0&=\dd x^0_1\,\dd y^{-1}_1+\cdots+\dd x^0_{m_0}\,\dd y^{-1}_{m_0},\\
\phi&=-y^{-1}_1\,\dd x^0_1-\cdots-y^{-1}_{m_0}\,\dd x^0_{m_0}.
\end{split}
\label{da5eq36}
\e
Similarly, for $B,\check\om,\check\Phi,\check\phi$ we have
\'etale coordinates $(\check x_1^0,\ldots,\check x_{\check
m_0}^0):V\ra\bA^{\check m_0}$ on $V$ such that $B=B(0)[\check
y_1^{-1},\ldots,\check y_{\check m_0}^{-1}]$ with $\d\check
x^0_i=0$, $\d\check y^{-1}_i=\frac{\pd\check \Phi}{\pd\check x^0_i}$
for all $i$, and $\check\om^0,\check\phi$ are given by the analogue of~\eq{da5eq36}.

Localizing $C$ at $r$ if necessary, choose \'etale coordinates
$(z_1^0,\ldots,z_m^0):W\ra\bA^m$, for $m=\dim U$. Since $C$ is
minimal at $r$ and $\bSpec C$ is $-1$-shifted symplectic, we may
choose $e_1^{-1},\ldots,e_m^{-1}$ in $C^{-1}$ such that $\dd
e_1^{-1},\ldots,\ab\dd e_m^{-1}$ form a basis for $V^{-1}$ over
$C(0)$ in Proposition \ref{da2prop}, for the same $m=\dim U$, and
then as a commutative graded algebra $C=C(0)[e_1^{-1},
\ldots,e_m^{-1}]$ is freely generated over $C(0)$ by degree $-1$
variables $e_1^{-1},\ldots,e_m^{-1}$.

For all $i=1,\ldots,m_0$, $\check\imath=1,\ldots,\check m_0$ and
$j,j',j''=1,\ldots,m$, define functions $I_j,J_{ij},K_{\check\imath
j},L_j,M_{jj'},N_{jj'j''}$ in $C(0)$ with $N_{jj'j''}=-N_{j'jj''}$
by
\begin{gather*}
\d e_j^{-1}=I_j,\; a^*(y_i^{-1})=\sum_{j=1}^m J_{ij}e_j^{-1},\;
b^*(\check y_{\check\imath}^{-1})=\sum_{j=1}^m K_{\check\imath
j}e_j^{-1}, \; \Psi=\sum_{j=1}^mL_je_j^{-1},
\nonumber\\
\text{and}\quad\psi=\sum_{j,j'=1}^m M_{jj'}e_j^{-1}\dd
e_{j'}^{-1}+\sum_{j,j',j''=1}^m N_{jj'j''}e_j^{-1}e_{j'}^{-1}\dd
z_{j''}^0.
\end{gather*}
Then $H^0(C)=C(0)/\d C^{-1}=C(0)/(I_1,\ldots,I_m)$, for
$(I_1,\ldots,I_m)\subset C(0)$ the ideal generated by
$I_1,\ldots,I_m$. Since $\al:A\ra C$ induces $\al:H^0(A)\ra H^0(C)$
and both are Zariski open inclusions, with $H^0(A)=A(0)/(\dd \Phi)$ for
$(\dd \Phi)\subset A(0)$ the ideal generated by $\dd \Phi$, and similarly
for $\be$, we have
\e
(I_1,\ldots,I_m)=\bigl(a^*(\dd \Phi)\bigr)=\bigl(b^*(\dd\check
\Phi)\bigr).
\label{da5eq37}
\e

Equation \eq{da5eq35} becomes the equations
\begin{gather}
a^*(\Phi)-b^*(\check\Phi)=\sum_{j=1}^mI_j\,L_j,\quad
0=L_j+\sum_{j'=1}^mI_{j'}M_{j'j},
\label{da5eq38}\\
\sum_{i=1}^{m_0}J_{ij}\frac{\pd\al(x_i^0)}{\pd z_{j'}^0}-\!
\sum_{\check\imath=1}^{\check m_0}K_{\check\imath j}
\frac{\pd\be(\check x_{\check\imath}^0)}{\pd z_{j'}^0}=\!-\frac{\pd
L_j}{\pd z_{j'}^0}+\!\sum_{j''=1}^mM_{jj''}\frac{\pd I_{j''}}{\pd
z_{j'}^0}+\!2\sum_{j''=1}^mI_{j''}N_{j''jj'}, \nonumber
\end{gather}
where the second and third equations are the coefficients of $\dd
e_j^{-1},e_j^{-1}\dd z_{j'}^0$  in the second equation
of~\eq{da5eq35}.

The first two equations of \eq{da5eq38} imply that
\e
a^*(\Phi)\!-\!b^*(\check \Phi)\!\in\!(I_1,\ldots,I_m)^2\!=\!\bigl(a^*(\dd\Phi)\bigr){}^2\!=\!\bigl(b^*(\dd\check \Phi)\bigr){}^2\!\subset\! C(0),
\label{da5eq39}
\e
by \eq{da5eq37}. This will be important in the proof of Theorem
\ref{da6thm4} in~\S\ref{da63}.
\end{ex}

\section[\texorpdfstring{$-1$-shifted symplectic derived schemes and d-critical
loci}{-1-shifted symplectic derived schemes and d-critical
loci}]{$-1$-shifted symplectic derived schemes \\ and d-critical
loci}
\label{da6}

In \S\ref{da61} we introduce {\it d-critical loci\/} from Joyce
\cite{Joyc}. Our second main result Theorem \ref{da6thm4} is stated
and discussed in \S\ref{da62}, and proved in~\S\ref{da63}.

\subsection{Background material on d-critical loci}
\label{da61}

Here are some of the main definitions and results on d-critical
loci, from Joyce \cite[Th.s 2.1, 2.20, 2.28 \& Def.s 2.5, 2.18,
2.31]{Joyc}. In fact \cite{Joyc} develops two versions of the
theory, {\it algebraic d-critical loci\/} on $\K$-schemes and {\it
complex analytic d-critical loci\/} on complex analytic spaces, but
we discuss only the former.

\begin{thm} 
\label{da6thm1}
Let\/ $X$ be a $\K$-scheme. Then there exists a sheaf\/
$\cS_X$ of\/ $\K$-vector spaces on $X,$ unique up to canonical
isomorphism, which is uniquely characterized by the following two
properties:
\begin{itemize}
\setlength{\itemsep}{0pt}
\setlength{\parsep}{0pt}
\item[{\bf(i)}] Suppose $R\subseteq X$ is Zariski open, $U$ is a
smooth\/ $\K$-scheme, and\/ $i:R\hookra U$ is a closed
embedding. Then we have an exact sequence of sheaves of\/
$\K$-vector spaces on $R\!:$
\begin{equation*}
\smash{\xymatrix@C=30pt{ 0 \ar[r] & I_{R,U} \ar[r] &
i^{-1}(\O_U) \ar[r]^{i^\sharp} & \O_X\vert_R \ar[r] & 0, }}
\end{equation*}
where $\O_X,\O_U$ are the sheaves of regular functions on $X,U,$
and\/ $i^\sharp$ is the morphism of sheaves of\/ $\K$-algebras
on $R$ induced by $i$.

There is an exact sequence of sheaves of\/ $\K$-vector spaces on
$R\!:$
\begin{equation*}
\xymatrix@C=20pt{ 0 \ar[r] & \cS_X\vert_R
\ar[rr]^(0.4){\io_{R,U}} &&
\displaystyle\frac{i^{-1}(\O_U)}{I_{R,U}^2} \ar[rr]^(0.4)\d &&
\displaystyle\frac{i^{-1}(T^*U)}{I_{R,U}\cdot i^{-1}(T^*U)}\,, }
\end{equation*}
where $\d$ maps $f+I_{R,U}^2\mapsto \d f+I_{R,U}\cdot
i^{-1}(T^*U)$.
\item[{\bf(ii)}] Let\/ $R\subseteq S\subseteq X$ be Zariski open,
$U,V$ be smooth\/ $\K$-schemes, $i:R\hookra U,$ $j:S\hookra V$
closed embeddings, and\/ $\Psi:U\ra V$ a morphism with\/
$\Psi\ci i=j\vert_R:R\ra V$. Then the following diagram of
sheaves on $R$ commutes:
\e
\begin{gathered}
\xymatrix@C=12pt{ 0 \ar[r] & \cS_X\vert_R \ar[d]^\id
\ar[rrr]^(0.4){\io_{S,V}\vert_R} &&&
\displaystyle\frac{j^{-1}(\O_V)}{I_{S,V}^2}\Big\vert_R
\ar@<-2ex>[d]^{i^{-1}(\Psi^\sharp)} \ar[rr]^(0.4)\d &&
\displaystyle\frac{j^{-1}(T^*V)}{I_{S,V}\cdot
j^{-1}(T^*V)}\Big\vert_R \ar@<-2ex>[d]^{i^{-1}(\d\Psi)} \\
 0 \ar[r] & \cS_X\vert_R \ar[rrr]^(0.4){\io_{R,U}} &&&
\displaystyle\frac{i^{-1}(\O_U)}{I_{R,U}^2} \ar[rr]^(0.4)\d &&
\displaystyle\frac{i^{-1}(T^*U)}{I_{R,U}\cdot i^{-1}(T^*U)}\,.
}\!\!\!\!\!\!\!{}
\end{gathered}
\label{da6eq1}
\e
Here $\Psi:U\ra V$ induces $\Psi^\sharp: \Psi^{-1}(\O_V)\ra\O_U$
on $U,$ so we have
\e
i^{-1}(\Psi^\sharp):j^{-1}(\O_V)\vert_R=i^{-1}\ci
\Psi^{-1}(\O_V)\longra i^{-1}(\O_U),
\label{da6eq2}
\e
a morphism of sheaves of\/ $\K$-algebras on $R$. As $\Psi\ci
i=j\vert_R,$ equation \eq{da6eq2} maps $I_{S,V}\vert_R\ra
I_{R,U},$ and so maps $I_{S,V}^2\vert_R\ra I_{R,U}^2$. Thus
\eq{da6eq2} induces the morphism in the second column of\/
\eq{da6eq1}. Similarly, $\d\Psi:\Psi^{-1}(T^*V)\ra T^*U$ induces
the third column of\/~\eq{da6eq1}.
\end{itemize}

There is a natural decomposition\/ $\cS_X=\cSz_X\op\K_X,$ where\/
$\K_X$ is the constant sheaf on $X$ with fibre $\K,$ and\/
$\cSz_X\subset\cS_X$ is the kernel of the composition
\begin{equation*}
\xymatrix@C=40pt{ \cS_X \ar[r] & \O_X
\ar[r]^(0.47){i_X^\sharp} & \O_{X^\red}, }
\end{equation*}
with\/ $X^\red$ the reduced\/ $\K$-subscheme of\/ $X,$ and\/
$i_X:X^\red\hookra X$ the inclusion.
\end{thm}

\begin{dfn} 
\label{da6def1}
An {\it algebraic d-critical locus\/} over a
field $\K$ is a pair $(X,s)$, where $X$ is a $\K$-scheme and $s\in
H^0(\cSz_X)$ for $\cSz_X$ as in Theorem \ref{da6thm1}, such that for
each $x\in X$, there exists a Zariski open neighbourhood $R$ of $x$
in $X$, a smooth $\K$-scheme $U$, a regular function
$f:U\ra\bA^1=\K$, and a closed embedding $i:R\hookra U$, such that
$i(R)=\Crit(f)$ as $\K$-subschemes of $U$, and
$\io_{R,U}(s\vert_R)=i^{-1}(f)+I_{R,U}^2$. We call the quadruple
$(R,U,f,i)$ a {\it critical chart\/} on~$(X,s)$.

Let $(X,s)$ be an algebraic d-critical locus, and $(R,U,f,i)$ a
critical chart on $(X,s)$. Let $U'\subseteq U$ be Zariski open, and
set $R'=i^{-1}(U')\subseteq R$, $i'=i\vert_{R'}:R'\hookra U'$, and
$f'=f\vert_{U'}$. Then $(R',U',f',i')$ is a critical chart on
$(X,s)$, and we call it a {\it subchart\/} of $(R,U,f,i)$. As a
shorthand we write~$(R',U',f',i')\subseteq (R,U,f,i)$.

Let $(R,U,f,i),(S,V,g,j)$ be critical charts on $(X,s)$, with
$R\subseteq S\subseteq X$. An {\it embedding\/} of $(R,U,f,i)$ in
$(S,V,g,j)$ is a locally closed embedding $\Psi:U\hookra V$ such
that $\Psi\ci i=j\vert_R$ and $f=g\ci\Psi$. As a shorthand we write
$\Psi: (R,U,f,i)\hookra(S,V,g,j)$. If $\Psi:(R,U,f,i)\hookra
(S,V,g,j)$ and $\Xi:(S,V,g,j)\hookra(T,W,h,k)$ are embeddings, then
$\Xi\ci\Psi:(R,U,i,e)\hookra(T,W,h,k)$ is also an embedding.
\end{dfn}

\begin{thm} 
\label{da6thm2}
Let\/ $(X,s)$ be an algebraic d-critical locus, and\/
$(R,U,f,i),\ab(S,\ab V,\ab g,\ab j)$ be critical charts on $(X,s)$.
Then for each\/ $x\in R\cap S\subseteq X$ there exist subcharts
$(R',U',f',i')\subseteq(R,U,f,i),$ $(S',V',g',j')\subseteq
(S,V,g,j)$ with\/ $x\in R'\cap S'\subseteq X,$ a critical chart\/
$(T,W,h,k)$ on $(X,s),$ and embeddings $\Psi:(R',U',f',i')\hookra
(T,W,h,k),$ $\Xi:(S',V',g',j')\hookra(T,W,h,k)$.
\end{thm}

\begin{thm} 
\label{da6thm3}
Let\/ $(X,s)$ be an algebraic d-critical locus, and\/
$X^\red\subseteq X$ the associated reduced\/ $\K$-scheme. Then there
exists a line bundle $K_{X,s}$ on $X^\red$ which we call the
\begin{bfseries}canonical bundle\end{bfseries} of\/ $(X,s),$ which
is natural up to canonical isomorphism, and is characterized by the
following properties:
\begin{itemize}
\setlength{\itemsep}{0pt}
\setlength{\parsep}{0pt}
\item[{\bf(i)}] If\/ $(R,U,f,i)$ is a critical chart on
$(X,s),$ there is a natural isomorphism
\begin{equation*}
\io_{R,U,f,i}:K_{X,s}\vert_{R^\red}\longra
i^*\bigl(K_U^{\ot^2}\bigr)\vert_{R^\red},
\end{equation*}
where $K_U=\La^{\dim U}T^*U$ is the canonical bundle of\/ $U$ in
the usual sense.
\item[{\bf(ii)}] Let\/ $\Psi:(R,U,f,i)\hookra(S,V,g,j)$ be an
embedding of critical charts on $(X,s)$. Then
{\rm\cite[Def.~2.26]{Joyc}} defines an isomorphism of line
bundles
\begin{equation*}
J_\Psi:i^*\bigl(K_U^{\ot^2}\bigr)\big\vert_{R^\red}\,{\buildrel
\cong\over\longra}\,j^*\bigl(K_V^{\ot^2}\bigr)\big\vert_{R^\red}
\end{equation*}
on $R^\red,$ and we must have
\begin{equation*}
\io_{S,V,g,j}\vert_{R^\red}=J_\Psi\ci
\io_{R,U,f,i}:K_{X,s}\vert_{R^\red}\longra j^*
\bigl(K_V^{\ot^2}\bigr)\big\vert_{R^\red}.
\end{equation*}
\end{itemize}
\end{thm}

\begin{dfn} 
\label{da6def2}
Let $(X,s)$ be an algebraic d-critical locus, and
$K_{X,s}$ its canonical bundle from Theorem \ref{da6thm3}. An {\it
orientation\/} on $(X,s)$ is a choice of square root line bundle
$K_{X,s}^{1/2}$ for $K_{X,s}$ on $X^\red$. That is, an orientation
is a line bundle $L$ on $X^\red$, together with an isomorphism
$L^{\ot^2}=L\ot L\cong K_{X,s}$. A d-critical locus with an
orientation will be called an {\it oriented d-critical locus}.
\end{dfn}

\subsection{The second main result, and applications}
\label{da62}

Here is the second main result of this paper, which will be proved
in~\S\ref{da63}:

\begin{thm} 
\label{da6thm4}
Suppose\/ $(\bX,\ti\om)$ is a $-1$-shifted symplectic
derived\/ $\K$-scheme, and let\/ $X=t_0(\bX)$ be the associated
classical\/ $\K$-scheme of\/ ${\bX}$. Then $X$ extends uniquely to
an algebraic d-critical locus\/ $(X,s),$ with the property that
whenever\/ $(\bSpec A,\om)$ is a  $-1$-shifted symplectic derived\/
$\K$-scheme in Darboux form with Hamiltonian $\Phi\in A(0),$ as in
Examples\/ {\rm\ref{da5ex2}} and\/ {\rm\ref{da5ex6},} and\/ $\bs
f:\bSpec A\ra\bX$ is an equivalence in $\dSch_\K$ with a Zariski
open derived\/ $\K$-subscheme\/ $\bR\subseteq\bX$ with\/ $\bs
f^*(\ti\om)\sim\om,$ writing\/ $U=\Spec A(0),$ $R=t_0(\bR),$
$f=t_0(\bs f)$ so that\/ $\Phi:U\ra\bA^1$ is regular and\/
$f:\Crit(\Phi)\ra R$ is an isomorphism, for $\Crit(\Phi)\subseteq U$ the
classical critical locus of\/ $\Phi,$ then $(R,U,\Phi,f^{-1})$ is a
critical chart on\/~$(X,s)$.

The canonical bundle $K_{X,s}$ from Theorem\/ {\rm\ref{da6thm3}} is
naturally isomorphic to the determinant line bundle
$\det(\bL_{\bX})\vert_{X^\red}$ of the cotangent complex\/
$\bL_{\bX}$ of\/~$\bX$.
\end{thm}

We can think of Theorem \ref{da6thm4} as defining a {\it truncation
functor}
\e
\begin{split}
F:\bigl\{&\text{category of $-1$-shifted symplectic derived
$\K$-schemes $(\bX,\om)$}\bigr\}\\
&\longra\bigl\{\text{category of algebraic d-critical loci
$(X,s)$ over $\K$}\bigr\},
\end{split}
\label{da6eq3}
\e
where the morphisms $\bs f:(\bX,\om)\ra(\bY,\om')$ in the first line
are (homotopy classes of) \'etale maps $\bs f:\bX\ra\bY$ with $\bs
f^*(\om')\sim\om$, and the morphisms $f:(X,s)\ra (Y,t)$ in the
second line are \'etale maps $f:X\ra Y$ with~$f^*(t)=s$.

In \cite[Ex.~2.17]{Joyc} we give an example of $-1$-shifted
symplectic derived schemes $(\bX,\om),(\bY,\om')$, both global
critical loci, such that $\bX,\bY$ are not equivalent as derived
$\K$-schemes, but their truncations $F(\bX,\om),F(\bY,\om')$ are
isomorphic as algebraic d-critical loci. Thus, the functor $F$ in
\eq{da6eq3} is not full.

Suppose $Y$ is a Calabi--Yau 3-fold over $\K$ and $\cM$ a classical
moduli $\K$-scheme of simple coherent sheaves in $\coh(Y)$. Then
Thomas \cite{Thom} defined a natural {\it perfect obstruction
theory\/} $\phi:\cE^\bu\ra\bL_\cM$ on $\cM$ in the sense of Behrend
and Fantechi \cite{BeFa}, and Behrend \cite{Behr} showed that
$\phi:\cE^\bu\ra\bL_\cM$ can be made into a {\it symmetric
obstruction theory}. More generally, if $\cM$ is a moduli
$\K$-scheme of simple complexes of coherent sheaves in $D^b\coh(Y)$,
then Huybrechts and Thomas \cite{HuTh} defined a natural symmetric
obstruction theory on~$\cM$.

Now in derived algebraic geometry $\cM=t_0(\bs\cM)$ for $\bs\cM$ the
corresponding derived moduli $\K$-scheme, and the obstruction theory
$\phi:\cE^\bu\ra\bL_\cM$ from \cite{HuTh,Thom} is
$\bL_{t_0}:\bL_{\bs\cM}\vert_\cM\ra\bL_\cM$. Pantev et al.\ \cite[\S
2.1]{PTVV} prove $\bs\cM$ has a $-1$-shifted symplectic structure
$\om$, and the symmetric structure on $\phi:\cE^\bu\ra\bL_\cM$ from
\cite{Behr} is $\om^0\vert_\cM$. So as for Corollary \ref{da5cor1},
Theorem \ref{da6thm4} implies:

\begin{cor} 
\label{da6cor1}
Suppose $Y$ is a Calabi--Yau\/ $3$-fold over\/ $\K,$
and\/ $\cM$ is a classical moduli\/ $\K$-scheme of simple coherent
sheaves in $\coh(Y),$ or simple complexes of coherent sheaves in
$D^b\coh(Y),$ with perfect obstruction theory\/
$\phi:\cE^\bu\ra\bL_\cM$ as in Thomas\/ {\rm\cite{Thom}} or
Huybrechts and Thomas\/ {\rm\cite{HuTh}}. Then $\cM$ extends
naturally to an algebraic d-critical locus $(\cM,s)$. The canonical
bundle $K_{\cM,s}$ from Theorem\/ {\rm\ref{da6thm3}} is naturally
isomorphic to $\det(\cE^\bu)\vert_{\cM^\red}$.
\end{cor}

If $(S,\om)$ is an algebraic symplectic manifold over $\K$ and
$L,M\subseteq S$ are Lagrangians, then Pantev et al.\
\cite[Th.~2.10]{PTVV} show that the derived intersection $\bX=L\t_S
M$ has a $-1$-shifted symplectic structure. If $X=t_0(\bX)$ then
$\bL_\bX\vert_X\simeq[T^*S\vert_X\ra T^*L\vert_X\op T^*M\vert_X]$
with $T^*S\vert_X$ in degree $-1$ and $T^*L\vert_X\op T^*M\vert_X$
in degree zero. Hence
\begin{equation*}
\det(\bL_\bX\vert_X)\cong K_S\vert_X^{-1}\ot K_L\vert_X\ot K_M\vert_X
\cong K_L\vert_X\ot K_M\vert_X,
\end{equation*}
since $K_S\cong\O_S$. So as for Corollary \ref{da5cor3}, Theorem
\ref{da6thm4} implies:

\begin{cor} 
\label{da6cor2}
Suppose $(S,\om)$ is an algebraic symplectic manifold
over $\K,$ and\/ $L,M$ are algebraic Lagrangians in $S$. Then the
intersection $X=L\cap M,$ as a $\K$-subscheme of\/ $S,$ extends
naturally to an algebraic d-critical locus\/ $(X,s)$. The canonical
bundle $K_{X,s}$ from Theorem\/ {\rm\ref{da6thm3}} is isomorphic
to\/ $K_L\vert_{X^\red}\ot K_M\vert_{X^\red}$.

\end{cor}

Bussi \cite[\S 3]{Buss} proves a complex analytic analogue of
Corollary~\ref{da6cor2}.

We will take these ideas further in Brav, Bussi, Dupont, Joyce, and
Szendr\H oi \cite{BBDJS} and Bussi, Joyce and Meinhardt \cite{BJM}.
In \cite[\S 6]{BBDJS}, given an algebraic d-critical locus $(X,s)$
with an orientation $\smash{K_{X,s}^{1/2}}$, we construct a natural
perverse sheaf $\smash{P_{X,s}^\bu}$ on $X$ such that if $(X,s)$ is
locally modelled on $\Crit(f:U\ra\bA^1)$ for $U$ a smooth
$\K$-scheme then $P_{X,s}^\bu$ is locally modelled on the perverse
sheaf of vanishing cycles $\PV_{U,f}^\bu$. This has applications to
the categorification of Donaldson--Thomas theory of Calabi--Yau
3-folds, and to constructing a `Fukaya category' of algebraic
Lagrangians in an algebraic symplectic manifold.

In \cite[\S 5]{BJM}, given an algebraic d-critical locus $(X,s)$
with an orientation $K_{X,s}^{1/2}$, we construct a natural motive
$MF_{X,s}$ in a certain ring of motives $\oM^{\hat\mu}_X$ over $X$,
such that if $(X,s)$ is locally modelled on $\Crit(f:U\ra\bA^1)$
then $MF_{X,s}$ is locally modelled on $\bL^{-\dim
U/2}\bigl([X]-MF^{\rm mot}_{U,f}\bigr),$ where $MF^{\rm mot}_{U,f}$
is the motivic Milnor fibre of $f$. This has applications to motivic
Donaldson--Thomas invariants. We refer the reader to
\cite{BBDJS,BJM,Joyc} for more details.

\subsection{Proof of Theorem \ref{da6thm4}}
\label{da63}

Let $(\bX,\ti\om)$ be a $-1$-shifted symplectic derived $\K$-scheme,
and $X=t_0(\bX)$. For the first part of Theorem \ref{da6thm4}, we
must construct a section $s\in H^0(\cSz_X)$ such that $(X,s)$ is a
d-critical locus, and if $A,\om,\Phi,\bs f,\bR,U,R,f$ are as in Theorem
\ref{da6thm4}, then $(R,U,\Phi,f^{-1})$ is a critical chart on $(X,s)$.
The condition that $(R,U,\Phi,f^{-1})$ is a critical chart determines
$s\vert_R$ uniquely, as in Definition~\ref{da6def1}.

Theorem \ref{da5thm1}(i) implies that for any $x\in\bX$, we can find
such $A,\ab\om,\ab \Phi,\ab\bs f,\ab\bR,\ab U,\ab R,\ab f$ with $x\in
R\subseteq X$. So the condition in Theorem \ref{da6thm4} determines
$s\vert_R$ for Zariski open $R\subseteq X$ in an open cover of $X$.
Thus $s\in H^0(\cSz_X)$ satisfying the conditions of the theorem is
unique if it exists, and it exists if and only if the prescribed
values $s\vert_R,s\vert_S$ agree on overlaps $R\cap S$ between open
sets $R,S\subseteq X$.

So suppose $A,\om,\Phi,\bs f,\bR,U,R,f$ and $B,\check\om,\check \Phi,\bs
g,\bS,V,S,g$ are two choices above. Write $s_R$ and $s_S$ for the
sections of $\cSz_X$ on $R,S\subseteq X$ determined by the critical
charts $(R,U,\Phi,f^{-1})$ and $(S,V,\check \Phi,g^{-1})$, so that by
Definition \ref{da6def1}
\e
\io_{R,U}(s_R)=(f^{-1})^{-1}(\Phi)+I_{R,U}^2, \qquad
\io_{S,V}(s_S)=(g^{-1})^{-1}(\check \Phi)+I_{S,V}^2.
\label{da6eq4}
\e
We must show that $s_R\vert_{R\cap S}=s_S\vert_{R\cap S}$.

Let $x\in R\cap S$, so that $x=f(p)=g(q)$ for unique $p\in\Crit
\Phi\subseteq U$ and $q\in\Crit(\check \Phi)\subseteq V$. Then using the
method of \S\ref{da57}, Example \ref{da5ex9} constructs a standard
form cdga $C$ minimal at $r\in \Spec H^0(C)$, and Zariski open
inclusions $\al:A\ra C$, $\be:B\ra C$ with $\bs
f\ci\bSpec\al\simeq\bs g\ci\bSpec\be$, such that the smooth
$\K$-scheme $W=\Spec C(0)$ and $\K$-scheme morphisms
$a=\Spec\al(0):W\ra U$, $b=\Spec\be(0):W\ra V$ satisfy by
\eq{da5eq39}
\e
a^*(\Phi)\!-\!b^*(\check \Phi)\!\in\!\bigl(\d
C^{-1}\bigr){}^2\!=\!\bigl(a^*(\dd
\Phi)\bigr){}^2\!=\!\bigl(b^*(\dd\check \Phi)\bigr){}^2\!\subset\! C(0).
\label{da6eq5}
\e

Write $Z=\Spec H^0(C)$, regarded as a closed $\K$-subscheme of
$W=\Spec C(0)$. Then $f\ci a\vert_Z=g\ci b\vert_Z:Z\ra X$ is an
isomorphism with a Zariski open $\K$-subscheme $T\subseteq R\cap
S\subseteq X$ with $x\in T$. Define $s_T\in
H^0\bigl(\cSz_X\vert_T\bigr)$ by
\e
\begin{split}
\io_{T,W}(s_T)&=\bigl((f\ci a\vert_Z)^{-1}\bigr){}^{-1}(a^*(\Phi))
+I_{T,W}^2\\
&=\bigl((g\ci b\vert_Z)^{-1}\bigr)^{-1}(b^*(\check \Phi))+I_{T,W}^2,
\end{split}
\label{da6eq6}
\e
using the notation of Theorem \ref{da6thm1}(i) for the embedding
$(f\ci a\vert_Z)^{-1}=(g\ci b\vert_Z)^{-1}:T\hookra W$ of $T$ in the
smooth $\K$-scheme $W$, where the two expressions on the right hand
side of \eq{da6eq6} are equal by \eq{da6eq5}, since
\begin{equation*}
I_{T,W}=\bigl((f\ci a\vert_Z)^{-1}\bigr){}^{-1}\bigl((a^*(\dd \Phi))
\bigr)=\bigl((g\ci b\vert_Z)^{-1}\bigr){}^{-1}\bigl((b^*(\dd\check \Phi))
\bigr).
\end{equation*}

We now have
\begin{align*}
\io_{T,W}(s_R\vert_T)&=\bigl((f\ci a\vert_Z)^{-1}\bigr){}^{-1}
(a_\#)\ci \io_{R,U}\vert_T(s_R\vert_T)\\
&=\bigl((f\ci a\vert_Z)^{-1}\bigr){}^{-1}
(a_\#)\ci\bigl((f^{-1})^{-1}(\Phi)+I_{R,U}^2\bigr)\big\vert_T\\
&=\bigl((f\ci a\vert_Z)^{-1}\bigr){}^{-1}(a^*(\Phi))
+I_{T,W}^2=\io_{T,W}(s_T),
\end{align*}
using \eq{da6eq1} with $T,W,(f\ci a\vert_Z)^{-1},R,U,f^{-1},a$ in
place of $R,U,i,S,V,j,\Psi$ in the first step, \eq{da6eq4} in the
second, and \eq{da6eq6} in the fourth. Hence $s_R\vert_T=s_T$, as
$\io_{T,W}$ is injective in Theorem \ref{da6thm1}(i). Similarly
$s_S\vert_T=s_T$, so $s_R\vert_T=s_S\vert_T$. As we can cover $R\cap
S$ by such open $x\in T\subseteq R\cap S$, this implies that
$s_R\vert_{R\cap S}=s_S\vert_{R\cap S}$, and the first part of
Theorem \ref{da6thm4} follows.

For the second part of the theorem, let $A,\om,\Phi,\bs f,\bR,U,R,f$ be
as in Theorem \ref{da6thm4}, so that $(R,U,\Phi,f^{-1})$ is a critical
chart on $(X,s)$, and write $Y=\Spec H^0(A)\subseteq U$, so that
$f:Y\ra R$ is an isomorphism. Then Theorem \ref{da6thm3}(i) gives a
natural isomorphism
\e
\io_{R,U,\Phi,f^{-1}}:K_{X,s}\vert_{R^\red}\longra
(f^{-1})^*\bigl(K_U^{\ot^2}\bigr)\vert_{R^\red}.
\label{da6eq7}
\e
Also $\bL_{\bs f}:\bs f^*(\bL_\bX)\ra\bL_A\simeq\Om^1_A$ is a
quasi-isomorphism as $\bs f$ is a Zariski open inclusion. Hence
$\det(\bL_{\bs f})\vert_{Y^\red}:f^*(\det(\bL_\bX)\vert_{R^\red})\ra
\det(\Om^1_A)\vert_{Y^\red}$ is an isomorphism, so pulling back by
$f^{-1}\vert_{R^\red}$ gives an isomorphism
\e
(f^{-1}\vert_{R^\red})^*\bigl(\det(\bL_{\bs
f})\vert_{Y^\red}\bigr):\det(\bL_\bX)\vert_{R^\red}\longra
(f^{-1}\vert_{R^\red})^*\bigl(\det(\Om^1_A)\vert_{Y^\red}\bigr).
\label{da6eq8}
\e

Now by the material of \S\ref{da23} and \S\ref{da33}, we have a
natural isomorphism
\begin{equation*}
\Om^1_A\vert_{Y^\red}\cong \xymatrix@C=40pt{\bigl[TU\vert_{Y^\red}
\ar[r]^{\pd^2\Phi\vert_{Y^\red}} & T^*U\vert_{Y^\red}\bigr], }
\end{equation*}
with $TU\vert_{Y^\red}$ in degree $-1$ and $T^*U\vert_{Y^\red}$ in
degree 0. Thus we have a natural isomorphism
\e
\det(\Om^1_A)\vert_{Y^\red}\cong K_U^{\ot^2}\vert_{Y^\red}.
\label{da6eq9}
\e

Combining \eq{da6eq7}--\eq{da6eq9} gives a natural isomorphism
\e
K_{X,s}\vert_{R^\red}\longra\det(\bL_\bX)\vert_{R^\red},
\label{da6eq10}
\e
for each critical chart $(R,U,\Phi,f^{-1})$ constructed from
$A,\om,\Phi,\bs f,\bR$ as above. Combining Example \ref{da5ex9} on
comparing the charts $(R,U,\Phi,f^{-1})$ with the material of \cite[\S
2.4]{Joyc} defining the isomorphism $J_\Psi$ in Theorem
\ref{da6thm3}(ii), one can show that the canonical isomorphisms
\eq{da6eq10} on $R^\red,S^\red$ from two such charts
$(R,U,\Phi,f^{-1})$ and $(S,V,\check \Phi,g^{-1})$ are equal on the
overlap $(R\cap S)^\red$. Therefore the isomorphisms \eq{da6eq10}
glue to give a global canonical isomorphism
$K_{X,s}\cong\det(\bL_{\bX})\vert_{X^\red}$. This completes the
proof of Theorem~\ref{da6thm4}.

\medskip

\noindent{\small\sc Address for Christopher Brav:

\noindent Faculty of Mathematics, Higher School of Economics, 7 Vavilova Str., Moscow, Russia.

\noindent E-mail: {\tt chris.i.brav@gmail.com}.

\noindent Address for Dominic Joyce:

\noindent The Mathematical Institute, Radcliffe Observatory Quarter, Woodstock Road, Oxford, OX2 6GG, U.K.

\noindent E-mail: {\tt joyce@maths.ox.ac.uk}.}

\end{document}